\documentclass[11pt]{article}
\RequirePackage{amsthm,amsmath,amssymb}
\RequirePackage{natbib}
\RequirePackage[colorlinks,citecolor=blue,urlcolor=blue]{hyperref}
\usepackage{bigints}
\usepackage{verbatim}
\usepackage{graphicx}
\usepackage{caption}
\usepackage{subcaption}
\usepackage{bbm}
\usepackage{color}
\usepackage{enumerate}
\usepackage{fullpage}

\newtheorem{theorem}{Theorem}[section]
\newtheorem{proposition}[theorem]{Proposition}
\newtheorem{lemma}[theorem]{Lemma}
\newtheorem{corollary}[theorem]{Corollary}
\newtheorem{remark}[theorem]{Remark}
\newtheorem{example}[theorem]{Example}

\newcommand{\wbar}{\overline w}
\newcommand{\defeq}{:=}
\providecommand{\norms}[1]{\|{#1}\|}
\newcommand{\normsmall}[1]{\|{#1}\|}
\newcommand{\thetahat}{\widehat\theta}
\newcommand{\zhat}{\widehat z}
\newcommand{\normal}{N}
\newcommand{\thetastar}{\theta^*}
\newcommand{\order}{\mathcal{O}}
\newcommand{\ltwos}[1]{\normsmall{#1}_2}
\newcommand{\E}{{\mathbb E}}
\newcommand{\Q}{{\mathbb Q}}

\newcommand{\R}{{\mathbb R}}
\renewcommand{\P}{{\mathbb P}}
\renewcommand{\Q}{{\mathbb Q}}
\newcommand{\ac}{{\mathcal{A}}}
\newcommand{\A}{{\mathcal{A}}}

\newcommand{\C}{{\mathcal{C}}}

\newcommand{\dr}{{\mathfrak{d}}}

\newcommand{\Ps}{{\mathcal{P}}}

\newcommand{\samp}{{\mathcal{X}}}

\newcommand{\Rmx}{R_{\rm minimax}}
\newcommand{\Rnm}{R_{\rm normalized}}
\newcommand{\RNum}[1]{\uppercase\expandafter{\romannumeral #1\relax}}
\makeatother
\newcommand{\argmin}{\mathop{\mathrm{arg\,min}{}}}
\newcommand{\argmax}{\mathop{\mathrm{arg\,max}{}}}
\newcommand{\Iu}{I_f^{\rm up}}
\newtheorem{claim}{Claim}

\def\qt#1{\quad\text{#1}}

\def\argmin{\mathop{\rm argmin}}
\def\argmax{\mathop{\rm argmax}}
\newcommand{\Thetabar}{\overline \Theta}
\newcommand{\RBayes}{R_{\rm Bayes}}
\newcommand{\Lest}{L_{\rm est}}
\newcommand{\Lpre}{L_{\rm pre}}
\newcommand{\ball}{\mathbb{B}}
\newcommand{\tr}{{\rm tr}}

\newcounter{rcnt}[section]

\def\qt#1{\quad\text{#1}}

\def\argmin{\mathop{\rm argmin}}
\def\argmax{\mathop{\rm argmax}}

\definecolor{DSgray}{cmyk}{0,1,0,0}

\begin{document}

\title{On Bayes risk lower bounds \\}
\date{}
 
\author{Xi Chen \\ New York University \\xchen3@stern.nyu.edu \and Adityanand Guntuboyina \\ UC Berkeley \\ aditya@stat.berkeley.edu \and Yuchen Zhang \\ Stanford University \\zhangyuc@cs.stanford.edu}

\maketitle

\begin{abstract}
This paper provides a general technique for lower bounding the Bayes risk of statistical estimation, applicable to arbitrary loss functions and arbitrary prior distributions. A lower bound on the Bayes risk not only serves as a lower bound on the minimax risk, but also characterizes the fundamental limit of any estimator given the prior knowledge. Our bounds are based on the notion of $f$-informativity~\citep{Csiszar:72}, which is a function of the underlying class of probability measures and the prior. Application of our bounds requires upper bounds on the $f$-informativity, thus we derive new upper bounds on $f$-informativity which often lead to tight Bayes risk lower bounds. Our technique leads to generalizations of a variety of classical minimax bounds (e.g., generalized Fano's inequality). Our Bayes risk lower bounds can be directly applied to several concrete estimation problems, including Gaussian location models, generalized linear models, and principal component analysis for spiked covariance models.  To further demonstrate the applications of our Bayes risk lower bounds to machine learning problems, we present two new theoretical results: (1) a precise characterization of the minimax risk of learning spherical  Gaussian mixture models under the smoothed analysis framework, and (2) lower bounds for the Bayes risk under a natural prior for both the prediction and estimation errors for high-dimensional sparse linear regression under an  improper learning setting.
\end{abstract}

\section{Introduction}

%Parameter estimation is a central task in machine learning.

Consider a standard setting where we observe data points $X$ taking values in a sample space $\samp$. The distribution of $X$ depends on an unknown parameter $\theta\in \Theta$ and is denoted by $P_{\theta}$. The goal is to compute an estimate of $\theta$ based on the observed samples. Formally, we denote the estimator by $\dr(X)$, where $\dr: \samp\to \Theta$ is a mapping from the sample space to the parameter space. The risk of the estimator is defined by $\E_{\theta} L(\theta, \dr(X))$ where $L: \Theta \times \mathcal{A} \mapsto [0, \infty)$ is a non-negative loss function. This framework applies to a broad scope of machine learning problems. Taking sparse linear regression as a concrete example, the data $X$ represents the design matrix and the response vector; the parameter space is the set of sparse vectors; the loss function can be chosen as a squared loss.

Given an estimation problem, we are interested in the lowest possible risk achievable by any estimator, which will be useful in justifying the potential of improving existing algorithms. The classical notion of optimality is formalized by the so-called \emph{minimax risk}. More specifically, we assume that the statistician chooses an optimal estimator $\dr$, then the adversary chooses the worst parameter $\theta$ by knowing the choice of $\dr$. The minimax risk is defined as:
\begin{eqnarray}
  R_{\rm minimax}(L; \Theta):= \inf_{\dr} \sup_{\theta \in \Theta} \E_{\theta} L(\theta, \dr(X)).
  \label{eq:R_minimax}
\end{eqnarray}
The minimax risk has been determined up to multiplicative constants for many  important problems. Examples include sparse linear regression~\citep{raskutti2011minimax}, classification~\citep{yang1999minimax}, additive models over kernel classes~\citep{raskutti2012minimax},  and crowdsourcing~\citep{zhang2014spectral}.

The assumption that the adversary is capable of choosing a worst-case parameter is sometimes over-pessimistic. In practice, the parameter that incurs a worst-case risk may appear with very small probability. To capture the hardness of the problem with this prior knowledge, it is reasonable to assume that the true parameter is sampled from an underlying prior distribution $w$. In this case, we are interested in the \emph{Bayes risk} of the problem. That is, the lowest possible risk when the true parameter is sampled from the prior distribution:
\begin{equation}
  R_{\rm Bayes}(w, L; \Theta) := \inf_{\dr} \int_{\Theta} \E_{\theta} L(\theta, \dr(X)) w(\mathrm{d}\theta).
 \label{eq:R_Bayes}
\end{equation}
If the prior distribution $w$ is known to the learner, then the Bayes estimator attains the Bayes risk~\citep{berger2013statistical}. But in general, the Bayes estimator is computationally hard to evaluate, and the Bayes risk has no closed-form expression. It is thus unclear what is the fundamental limit of estimators when the prior knowledge is available.

In this paper, we present a technique for establishing lower bounds on the Bayes risk for a general prior distribution $w$. When the lower bound matches the risk of any existing algorithm, it captures the convergence rate of the Bayes risk.
The Bayes risk lower bounds are useful for three main reasons:
\begin{enumerate}
  \item  They provide an idea of the difficulty of the problem under a specific prior $w$.

  \item They automatically provide lower bounds for the minimax risk and,  because the minimax regret is always larger than or equal to the minimax risk (see, for example, \cite{rakhlin2013empirical}), they also yield lower bounds for the  minimax regret.

%{\red Moreover,  \cite{rakhlin2013empirical} recently studied the minimax regret for prediction problems. Since the minimax regret is always larger than (or equal to) the minimax risk, our technique will  automatically yield lower bounds on minimax regret.}

  \item  As we will show, they have an important application in establishing the minimax lower bound  under the  \emph{smoothed analysis framework}. %and proving \emph{admissibility results} for general least-square estimators under convex constraints.
\end{enumerate}
Throughout this paper, when the loss function $L$ and the parameter space $\Theta$ are clear from the context, we simply denote the Bayes risk by $R_{\rm Bayes}(w)$. When the prior $w$ is also clear, the notation is further simplified to $R$.

\subsection{Our Main Results}\label{mrin}

In order to give the reader a flavor of the kind of results proved in this paper, let us consider Fano's classical inequality \citep{Han:94,Cover:IT:06, yu1997assouad}  which is one of the most widely used Bayes risk lower bounds in statistics and  information theory. The standard version of Fano's inequality applies to the case when $\Theta = \ac = \{1, \dots, N\}$ for some positive integer $N$ with the indicator loss  $L(\theta, a) := \mathbb{I}\{\theta \neq a\}$ ($\mathbb{I}$ stands for the zero-one valued indicator function) and the prior $w$ being the discrete uniform distribution on $\Theta$. In this setting, Fano's inequality states that
\begin{equation}\label{eq:classcial_Fano}
  R_{\rm Bayes}(w) \geq 1 - \frac{I(w, \mathcal{P}) + \log 2}{\log N}
\end{equation}
where $I(w, \mathcal{P})$ is the mutual information between the random variables $\theta \sim w$ and $X$ with $X \vert \theta \sim P_{\theta}$ (note that this mutual information only depends on $w$ and $\Ps=\{P_{\theta} : \theta \in \Theta\}$ which is why we denote it by $I(w, \Ps)$). Fano's inequality implies  that when $I(w; \mathcal{P})$ is large i.e., when the information that $X$ has about $\theta$ is large, then the risk of estimation is small.

A natural question regarding Fano's inequality, which does not seem to have been asked until very recently, is the following: does there exist an analogue of \eqref{eq:classcial_Fano} when $w$ is not necessarily the uniform prior and/or when $\Theta$ and $\ac$ are arbitrary sets, and/or when the loss function is not necessarily $\mathbb{I}\{\theta \neq a\}$? An interesting result in this direction is the following inequality which has been recently proved by \cite{Duchi:Fano} who termed it the continuum Fano inequality. This inequality applies to the case when $\Theta = \ac$ is a subset of Euclidean space with finite strictly positive Lebesgue measure, $L(\theta, a) = \mathbb{I}\{\|\theta - a\|_2 \geq \epsilon\}$ for a fixed $\epsilon > 0$ ($\|\cdot\|_2$ is the  usual Euclidean metric) and the prior $w$ being the uniform probability measure (i.e., normalized Lebesgue measure) on $\Theta$. In this setting, \cite{Duchi:Fano} proved that
\begin{equation}\label{cf}
  R_{\rm Bayes}(w) \geq 1 + \frac{I(w, \mathcal{P}) + \log 2}{\log \left( \sup_{a \in \ac}
      w\{\theta \in \Theta : \|\theta - a\|_2 < \epsilon\}  \right)}.
\end{equation}

It turns out that there is a very clean connection between inequalities \eqref{eq:classcial_Fano} and \eqref{cf}. Indeed, both these inequalities are special instances of the  following inequality:
\begin{equation}\label{intro_gf}
  R_{\rm Bayes}(w) \geq 1 + \frac{I(w, \mathcal{P}) + \log 2}{\log \left( \sup_{a \in \ac}
      w\{\theta \in \Theta : L(\theta, a) = 0\} \right)}
\end{equation}
Indeed, the term $w\{\theta \in \Theta : L(\theta, a) = 0\}$ equal to  $1/N$  in the setting of \eqref{eq:classcial_Fano} and it is equal to $w\{\theta \in \Theta: \|\theta - a\|_2 < \epsilon\}$ in the setting of \eqref{cf}.

Since both \eqref{eq:classcial_Fano} and \eqref{cf} are special instances of \eqref{intro_gf}, one might reasonably conjecture that inequality \eqref{intro_gf} might hold more generally. In Section \ref{sec:Bayes_01}, we give an affirmative answer by proving that inequality \eqref{intro_gf} holds for any zero-one valued loss function $L$ and any prior $w$. No assumptions on $\Theta$, $\ac$ and $w$ are needed. We refer to this result as \emph{generalized Fano's inequality}.  Our proof of \eqref{intro_gf} is quite succinct and is based on the data processing inequality \citep{Cover:IT:06,Liese12} for Kullback-Leibler (KL) divergence. The use of the data processing inequality for proving Fano-type inequalities was introduced by \cite{Gushchin:03}.

% (this  standard fact from information theory is recalled in the Section \ref{sec:pre}).

The data processing inequality is not only available for the KL divergence. It can be generalized to any divergence belonging to a general family known as $f$-divergences \citep{Csiszar:63, AliSilvey:66}. This family includes the KL divergence, chi-squared divergence, squared Hellinger distance, total variation distance and power divergences as special cases. The usefulness of $f$-divergences in machine learning has been illustrated in \citet{reid2011information, garcia2012divergences, reid2009generalised}.

For every $f$-divergence, one can define a quantity called $f$-informativity \citep{Csiszar:72} which plays the same role as the mutual information for KL divergence. The precise definitions of $f$-divergences and $f$-informativities are given in Section \ref{sec:pre}. Utilizing the data processing inequality for $f$-divergence, we prove  general Bayes  risk lower bounds which hold for every zero-one valued loss $L$ and for arbitrary $\Theta$, $\ac$ and $w$ (Theorem \ref{man}). The generalized Fano's inequality~\eqref{intro_gf} is a special case by choosing the $f$-divergence to be KL. The proposed Bayes risk lower bounds  can also be specialized to other $f$-divergences and have a variety of interesting connections to existing lower bounds in the literature such as Le Cam's inequality,  Assouad's lemma (see Theorem 2.12 in \cite{Tsybakov:nonpara}), Birg\'{e}-Gushchin inequality \citep{Gushchin:03, Birge:ineq}. These results are provided in Section \ref{sec:Bayes_01}.  %In Section \ref{sec:compare}, we also provide some qualitative comparisons among the derived Bayes risk lower bounds for different choices of $f$-divergence. We argue that Hellinger distance leads to inequalities that are qualitatively quite different (and less useful in certain applications) from the inequalities corresponding to KL and chi-squared  divergences. We also reason that our lower bounds involving KL/chi-squared are useful even in situations where all pairwise  KL/chi-squared divergences are infinite.

In Section \ref{sec:Bayes_general}, we deal with nonnegative valued loss functions $L$ which are not necessarily zero-one valued. Basically, we use the standard method of lower bounding the general loss function $L$ by a zero-one valued function and then use our results from Section \ref{sec:Bayes_01} for lower bounding the Bayes risk.
%Our approach here is the standard one: for every $t > 0$, we simply bound the loss function $L$ from below by $t L_t$ where $L_t$ is the zero-one valued loss function given by $L_t(\theta, a) := \mathbb{I}\{L(\theta, a) \geq t\}$. It is then easy to see that $R_{\rm Bayes}(w, L)$ is bounded from below by $t R_{\rm Bayes}(w, L_t)$ and, as a result, lower bounds for $R_{\rm Bayes}(w, L)$ can be obtained by using our results from Section \ref{sec:Bayes_01} for the zero-one valued loss  $L_t$.
This technique, in conjunction with the generalized Fano's inequality, gives the following lower bound (proved in Corollary \ref{cor:bwl})
\begin{equation}\label{eq:intro_I_bayes_KL}
R_{\rm Bayes}(w, L; \Theta)  \geq \frac{1}{2} \sup \left\{t> 0 : \sup_{a \in \ac} w \{\theta : L(\theta, a) <t \} \leq \frac{1}{4} e^{-2 I(w, \mathcal{P})} \right\}.
\end{equation}
A special case of the above inequality has appeared previously in \citet[Theorem 6.1]{zhang2006information} (please refer to Remark \ref{zhc} for a detailed explanation of the connection between inequality \eqref{eq:intro_I_bayes_KL} and \cite[Theorem 6.1]{zhang2006information}).

We also prove analogues of the above inequality for different $f$ divergences. Specifically, using our $f$-divergence inequalities from Section \ref{sec:Bayes_01}, we prove, in Theorem \ref{bwl}, the following inequality which holds for every $f$ divergence:
\begin{equation}\label{ggfe}
R_{\rm Bayes}(w, L; \Theta) \geq \frac{1}{2} \sup \left\{t > 0 : \sup_{a \in \ac} w \{\theta : L(\theta, a) < t\} < 1 - u_f(I_f(w, \mathcal{P}))\right\}
\end{equation}
where $I_f(w, \mathcal{P})$ represents the $f$-informativity and $u_f(\cdot)$ is a non-decreasing $[0, 1]$-valued function that depends only on $f$. This function $u_f(\cdot)$ (see its definition from \eqref{ufi}) can be explicitly computed for many $f$-divergences of interest, which gives useful lower bounds in terms of $f$-informativity. For example, for the case of KL divergence and chi-squared divergence, inequality \eqref{ggfe} gives the lower bound in \eqref{eq:intro_I_bayes_KL} and the following inequality respectively,
\begin{equation}\label{eq:intro_I_bayes_chi}
R_{\rm Bayes}(w, L; \Theta)  \geq \frac{1}{2} \sup \left\{t> 0 : \sup_{a \in \ac} w \{\theta : L(\theta, a) < t\} \leq \frac{1}{4(1 + I_{\chi^2}(w, \mathcal{P}))} \right\}.
\end{equation}
where $I_{\chi^2}(w, \mathcal{P})$ is the chi-squared informativity.

Intuitively, inequality \eqref{ggfe} shows that the Bayes risk is lower bounded by half of the largest possible $t$ such that the maximum prior mass of any $t$-radius ``ball" ($w \{\theta : L(\theta, a) < t \}$) is less than some function of $f$-informativity. To apply \eqref{ggfe},  one needs to obtain upper bounds on the following two quantities:
 \begin{enumerate}
\item The ``small ball probability" $\sup_{a \in \ac} w \{\theta : L(\theta, a) < t \}$, which does not depend  of the family of probability measures $\Ps$.
\item The $f$-informativity $I_f(w, \mathcal{P})$, which does not depend on the loss function $L$.
\end{enumerate}
We note that a nice feature of \eqref{ggfe} is that $L$ and $\Ps$ play separately roles. One may first obtain an upper bound $\Iu$ for the $f$-informativity $I_f(w, \mathcal{P})$, then choose $t$ so that the small ball probability $w \{\theta : L(\theta, a) < t \}$ can be bounded from above by $1 - u_f(\Iu)$. The Bayes risk will be bounded from below by $t/2$.  It is noteworthy that the terminology ``small ball probability" was used by \cite{XuRaginsky:14} (this paper proved information-theoretic lower bounds on the minimum time in a distributed function computation problem).

We do not have a general guideline for bounding the small ball probability. It needs to be dealt with case by case based on the prior and the loss function.  But for upper bounding the $f$-informativity, we offer a general recipe in Section \ref{sec:upper_f_informativity} for a  subclass of divergences of interest (power divergences for $\alpha \notin [0, 1)$), which covers the chi-squared divergence as one of the most important divergences in our applications.  These bounds generalize results of \cite{haussler1997} and \cite{Yang:Barron:99} for mutual information to $f$-informativities involving power divergences. As an illustration of our techniques (inequality \eqref{ggfe} combined with the $f$-informativity upper bounds), we apply them to a concrete estimation problem in Section \ref{sec:upper_f_informativity}. We further apply our results to several popular machine learning and statistics problems (e.g., generalized linear model, spiked covariance model, and Gaussian model with general loss) in Appendix \ref{sec:Bayes_Example}.

In Section~\ref{sec:smoothed-analysis} and Section~\ref{sec:sparse_linear}, we present non-trivial applications of our Bayes risk lower bounds to two learning problems: the first one is a unsupervised learning problem, while the second one is a supervised learning problem. Section~\ref{sec:smoothed-analysis} studies smoothed analysis for learning mixtures of spherical Gaussians with uniform weights. Although learning mixtures of Gaussians is a computationally hard problem, it has been shown recently by~\citet{hsu2013learning}
that under the assumptions that the Gaussian means are linearly independent, it can be learnt in polynomial time by a spectral method. We perform a smoothed analysis on a variant of the algorithm~\citep{hsu2013learning}, showing that the linear independence assumption can be replaced by perturbing the true parameters by a small random noise. The method described in Section~\ref{sec:smoothed-analysis} achieves a better convergence rate than the original algorithm of~\citet{hsu2013learning}. Furthermore, we apply the Bayes risk lower bound techniques to show that the algorithm's convergence rate is unimprovable, even under smoothed analysis (i.e.~when the true parameters are randomly perturbed). Section~\ref{sec:smoothed-analysis} highlights the usefulness of our techniques in proving lower bounds for smoothed analysis, which appears to be challenging using traditional techniques of the minimax theory.

In Section \ref{sec:sparse_linear}, we consider the high-dimensional sparse linear regression problem and we provide Bayes risk lower bounds for both prediction error and estimation error under a natural prior on the regression parameter belonging to the set of $k$-sparse vectors. Although lower bounds for sparse linear regression have been well-studied (see, e.g., \cite{raskutti2011minimax, zhang2014lower} and references therein), these bounds only focus on the minimax or the worst-case scenario and thus are too pessimistic in practice. Indeed, the parameters that usually attain these minimax lower bounds have zero probability under any continuous prior, so that their average effects might be negligible. The fundamental limits of sparse linear regression under a realistic prior is, to the best of the our knowledge, unknown. The developed tool of lower bounding Bayes risks can be directly applied to characterize these limits. Moreover, our Bayes risk lower bound is flexible in the sense that by tuning the variance  of the prior of non-zero elements of $\theta$, it provides a wide spectrum of lower bounds. For one particular choice of the variance, our Bayes risk lower bounds match the minimax risk lower bounds. This gives a natural \emph{least favorable prior} for sparse linear regression, while the known least favorable prior in \cite{raskutti2011minimax} is a non-constructive
discrete prior over a packing set of the parameter space that cannot be sampled from. We also work under the \emph{improper learning} setting where we allow non-sparse estimators for the true regression vector (even though the true regression vector is assumed to be sparse).

% output estimator $\widehat{\theta}$ belong to the space of the true parameters (i.e., the $k$-sparse vectors). However, our theoretical framework does not require  $\widehat{\theta}$ to be $k$-sparse, and allow any non-sparse estimates that achieves small risks.

%In Section \ref{sec:admis}, we apply Bayes risk lower bounds to a recent admissibility result of \cite{Sourav14LS}. The result deals with convex-constrained least squares estimators in the Gaussian sequence model.  Consider the problem of estimating a vector $\theta \in \R^n$ in squared Euclidean loss $L(\theta, a)=\|\theta-a\|_2^2$ from a single $n$-dimensional observation $X \sim N(\theta, I_n)$, i.e., $X$ is Gaussian with mean $\theta$ and identity covariance. The true parameter  $\theta$ is assumed to be in a known closed  convex set $\Theta \subseteq \R^n$. The most commonly used estimator in this setting is the Least Squares Estimator (LSE) defined as $\widehat{\theta}(X) := \argmin_{t \in \Theta} \|X - t\|_2^2$. \citet{Sourav14LS} proves a non-trivial result for least square estimators: the risk of LSE differs from that of the best possible estimator by at most a multiplicative constant $C$. In the original paper of \citet{Sourav14LS}, the proof of this so-called \emph{$C$-admissibility} result is rather difficult. We show in Section \ref{sec:admis} that with our Bayes risk lower bound, we can establish the $C$-admissibility for LSE in a much shorter and cleaner proof, with a better constant $C$ than \citet{Sourav14LS}.

\subsection{Related Works}

Before finishing this introduction section, we briefly describe related work on Bayes risk lower bounds. There are a few results dealing with special cases of finite dimensional estimation problems under (weighted/truncated) quadratic losses. The first results of this kind were established by \cite{VanTree:68}, and ~\cite{Borovkov:80} with extensions by~\cite{Brown:90, Brown:93, gill95,Sato:96, Takada:99}. A few additional papers dealt with even more specialized problems e.g., Gaussian white noise model \citep{BrownLiu:93}, scale models \citep{Gajek:94} and estimating Gaussian variance \citep{Vidakovic:95}. Most of these results are based on the van Trees inequality (see \cite{gill95} and Theorem 2.13 in \cite{Tsybakov:nonpara}). Although the van Trees inequality usually leads to sharp constant in the Bayes risk lower bounds, it only applies to weighted quadratic loss functions (as its proof relies on Cauchy-Schwarz inequality) and requires the underlying Fisher information to be easily computable, which limits its applicability. There is also a vast body of literature on minimax lower bounds (see, e.g., \cite{Tsybakov:nonpara}) which can be viewed as Bayes risk lower bounds for certain priors. These priors are usually discrete and specially constructed so that the lower bounds do not apply to more general (continuous) priors. Another related area of work involves finding lower bounds on posterior contraction rates (see, e.g., \citet{Castillo08Lower}).

%(\textbf{cite Castillo 2008 paper. This is to satisfy Reviewer 1;  do we need to say more?})
% In recent years, there are several works devoted to investigating high-dimensional problems under special priors. For example, \cite{Castillo15Sparse} studied the contraction rate of posterior distribution for high-dimensional linear regression under the sparse prior, which is a mixture of point masses at zero and continuous distributions. \cite{Gao15PCA} constructed a special prior to establish the optimal contraction rate of the posterior distribution for the sparse PCA problems. Nevertheless, the setup of most works on Bayesian approach for high-dimensional problems are quite different from this paper: these works assume the ``frequentist assumption"  that the data is generated according to a single true parameter (instead of a prior distribution over the parameter space).

%\textbf{Why do we need to mention the Bayesian papers; did referee one ask us to cite them?}

\subsection{Outline of the Paper}
\label{sec:conclusion}
The rest of the paper is organized in the following way. In Section \ref{sec:pre}, we describe notations and review preliminaries such as $f$-divergences, $f$-informativity, data processing inequality, etc. Section \ref{sec:Bayes_01} deals with inequalities for zero-one valued loss functions. These inequalities have many connections to existing lower bound techniques. Section \ref{sec:Bayes_general} deals with nonnegative loss functions and we provide inequality \eqref{ggfe} and its special cases. Section \ref{sec:upper_f_informativity} presents upper bounds on the $f$-informativity for power divergences for $\alpha \notin [0, 1)$. Some examples are also given in this section.  Section \ref{sec:smoothed-analysis} studies smoothed analysis for learning mixtures of spherical Gaussians with uniform weights using our technique.
% Section \ref{sec:admis} presents our simplified proof of Chatterjee's approximate admissibility result.
We conclude the paper in Section \ref{sec:conclusion}.   Due to space constraints, we have relegated some proofs and additional examples and results to the appendix.

%\textbf{what about least favorable priors?}

\section{Preliminaries and Notations}
\label{sec:pre}
We first review the notions of $f$-divergence~\citep{Csiszar:63,AliSilvey:66} and $f$-informativity~\citep{Csiszar:72}. Let $\C$ denote the class of all convex functions $f: (0, \infty) \rightarrow \R$ which satisfy $f(1) = 0$. Because of convexity, the limits $f(0) := \lim_{x \downarrow 0} f(x)$ and $f'(\infty) := \lim_{x \uparrow \infty} f(x)/x$  exist (even though they may be $+\infty$) for each $f \in \C$. Each function $f \in \C$  defines a divergence between probability measures which is referred to as $f$-divergence. For two probability measures $P$ and $Q$ on a sample space having densities  $p$ and $q$ with respect to a common measure $\mu$, the $f$-divergence $D_f(P||Q)$ between $P$ and $Q$ is defined as follows:
\begin{equation}\label{eq:df}
  D_f(P||Q) := \int f \left(\frac{p}{q} \right)q  \mathrm{d} \mu  + f'(\infty) P\{q = 0\}.
\end{equation}
We note that the convention $0\cdot \infty= 0$ is adopted here so that $f'(\infty) P\{q = 0\}=0$ when $f'(\infty)=\infty$ and $P\{q=0\}=0$. Note that $D_f(P\|Q) = +\infty$ when $f'(\infty) = +\infty$ and $P\{q = 0\} > 0$.  Also note that $f(1) =  0$ implies that $D_f(P \|Q) = 0$ when $P = Q$.

Certain divergences are commonly used because they can  be easily computed or bounded when $P$ and $Q$ are product measures. These divergences are the power divergences  corresponding to the functions $f_{\alpha}$ defined by
\begin{eqnarray*}
  f_\alpha(x) =\begin{cases}
    x^{\alpha}-1 & \text{for}  \quad \alpha \not \in [0,1];\\
    1- x^{\alpha} & \text{for}  \quad \alpha \in (0,1);\\
    x\log x & \text{for}  \quad \alpha=1; \\
    -\log x & \text{for}  \quad \alpha=0.\\
  \end{cases}
\end{eqnarray*}
%For power divergences, one has the identity
%\begin{equation}\label{eq:power_div_identity}
%  D_{f_{\alpha}} (P||Q) = D_{f_{1-\alpha}}(Q||P) \qt{for all $\alpha \in \R$}.
%\end{equation}
Popular examples of power divergences include:

\textbf{1)} Kullback-Leibler (KL) divergence: $\alpha=1$, $D_{f_1}(P||Q)=\bigintssss p \log (p/q) \mathrm{d} \mu$ if $P$ is absolutely continuous with respect to $Q$ (and it is infinite if $P$ is not absolutely continuous with respect to $Q$).  Following the conventional notation, we denote the KL divergence by $D(P||Q)$ (instead of $D_{f_1}(P||Q)$).

\textbf{2)} Chi-squared divergence: $\alpha=2$, $D_{f_2}(P||Q)= \bigintssss (p^2/q)  \mathrm{d} \mu - 1$ if $P$ is absolutely continuous with respect to $Q$ (and it is infinite if $P$ is not absolutely continuous with respect to $Q$). We denote the chi-squared divergence by $\chi^2(P||Q)$ following the conventional notation.

\textbf{3)} When $\alpha=1/2$, one has $D_{f_{1/2}}(P||Q)=  1- \bigintssss \sqrt{p q}  \mathrm{d} \mu$ which is a half of the squared Hellinger distance. That is, $D_{f_{1/2}}(P||Q)=H^2(P ||Q)/2$, where $H^2(P||Q)= \bigintssss  (\sqrt{p}-\sqrt{q})^2  \mathrm{d} \mu$ is the squared Hellinger distance between $P$ and $Q$.

The total variation distance $\|P-Q\|_{TV}$ is another $f$-divergence (with $f(x) = |x - 1|/2$) but not a power divergence.

One of the most important properties of $f$-divergences is the ``data processing inequality" (\cite{Csiszar:72} and \citet[Theorem 3.1]{Liese12}) which states the   following: let $\samp$ and $\mathcal{Y}$ be two measurable spaces and let $\Gamma: \samp \rightarrow \mathcal{Y}$ be a measurable function. For every $f \in \C$ and every pair of probability measures $P$ and $Q$ on $\samp$, we have
\begin{equation}
  \label{dpi}
D_f(P\Gamma^{-1} || Q \Gamma^{-1}) \leq D_f(P||Q),
\end{equation}
where $P\Gamma^{-1}$ and $Q \Gamma^{-1}$ denote the \emph{induced measures} of $\Gamma$ on $\mathcal{Y}$, i.e., for any measurable set $B$ on the space $\mathcal{Y}$, $P\Gamma^{-1}(B):=P(\Gamma^{-1}(B))$, $Q\Gamma^{-1}(B):=Q(\Gamma^{-1}(B))$ (see the definition of induced measure from Definition 2.2.1. in \cite{Athreya:06}).

Next, we introduce the notion of $f$-informativity~\citep{Csiszar:72}. Let $\mathcal{P}= \{P_{\theta}: \theta \in \Theta\}$ be a family of probability measures on a space $\mathcal{X}$ and $w$ be a probability measure on $\Theta$. For each $f \in \C$, the $f$-informativity, $I_f(w, \Ps)$, is defined as
\begin{eqnarray}\label{eq:def_f_info}
  I_f(w, \Ps)= \inf_{Q} \int D_f(P_{\theta}||Q) w(\mathrm{d} \theta),
\end{eqnarray}
where the infimum is taken over all possible probability measures $Q$ on $\mathcal{X}$. When $f(x)=x\log x$ (so that the corresponding $f$-divergence is the KL divergence), the $f$-informativity is equal to the mutual information and is denoted by $I(w, \mathcal{P})$. We denote the informativity corresponding to the power divergence $D_{f_\alpha}$ by $I_{f_{\alpha}}(w, \Ps)$. For the special case $\alpha = 2$, we use the more suggestive notation $I_{\chi^2}(w, \Ps)$. The informativity corresponding to the total variation distance will be denoted by $I_{TV}(w, \Ps)$.

Additional notations and definitions are described as follows.
Recall the Bayes risk~\eqref{eq:R_Bayes} and the minimax risk~\eqref{eq:R_minimax}. When the loss function $L$ and parameter space $\Theta$ are clear from the context, we drop the dependence on $L$ and $\Theta$. When the prior $w$ is also clear from the context, we denote the Bayes risk by $R$ and the minimax risk by $R_{\rm minimax}$.
%\item Throughout the paper, we assume that all probability measures $P_{\theta}$ and $Q$ have densities with respect to a common dominating measure $\mu$ on $\samp$ and we denote these densities by $p_{\theta}$ and $q$ respectively.
We need certain notation for covering numbers. For a given $f$-divergence  and a subset $S \subset \Theta$, let $M_f(\epsilon, S)$ denote any upper bound on the smallest number $M$ for which there exist probability measures $Q_1, \dots, Q_M$ that form an $\epsilon^2$-cover of $\{P_{\theta}, \theta \in S \}$ under the $f$-divergence i.e.,
\begin{equation}\label{redf}
  \sup_{\theta \in S} \min_{1 \leq j \leq M} D_f(P_{\theta} || Q_j) \leq \epsilon^2.
\end{equation}
We write the covering number as $M_{KL}(\epsilon, S)$ when $f(x) = x \log x$ and $M_{\chi^2}(\epsilon, S)$ when $f(x) = x^2 - 1$. We write $M_{\alpha}(\epsilon, S)$ when $f = f_{\alpha}$ for other $\alpha \in \R$. We note that $\log M_f(\epsilon, S)$ is an upper bound on the metric entropy. The quantity $M_f(\epsilon, S)$ can be infinite if $S$ is arbitrary.
For a vector $x=(x_1, \ldots, x_d)$ and a real number $p \geq 1$, denote by $\|x\|_p$ the $\ell_p$-norm of $x$. In particular, $\|x\|_2$ denotes the Euclidean norm of $x$.
$\mathbb{I}(A)$ denotes the indicator function which takes value 1 when $A$ is true and 0 otherwise.
We use $C$, $c$, etc. to denote generic constants whose values might change from place to place.

\section{Bayes Risk Lower Bounds for Zero-one Valued Loss Functions and Their Applications}
\label{sec:Bayes_01}

In this section, we consider zero-one loss functions $L$ and present a principled approach to derive Bayes risk lower bounds involving $f$-informativity for every $f \in \C$. Our results hold for any given prior $w$ and zero-one loss $L$. By specializing the $f$-divergence to KL   divergence, we obtain the generalized Fano's inequality~\eqref{intro_gf}. When specializing to other $f$-divergences, our bounds lead to some classical minimax bounds of Le Cam and Assouad \citep{Assouad:83}, more recent minimax results of \citet{Gushchin:03, Birge:ineq} and also results in \citet[Chapter 2]{Tsybakov:nonpara}. Bayes risk lower bounds for general nonnegative loss functions will be presented in the next section.

We need additional notations to state the main results of this section. For  each $f \in \C$, let $\phi_f: [0, 1]^2 \rightarrow \R$ be the function defined in the following way: for $a, b \in [0, 1]^2$, $\phi_f(a, b)$ is the $f$-divergence between the two probability measures $P$ and $Q$ on $\{0, 1\}$ given by $P\{1\} = a$ and $Q\{1\} = b$. By the definition \eqref{eq:df}, it is easy to see that $\phi_f(a, b)$ has the following  expression (recall that $f'(\infty) := \lim_{x \uparrow \infty} f(x)/x$):
\begin{eqnarray}\label{eq:phi}
  \phi_f(a, b) =\begin{cases}
    b f \left(\frac{a}{b} \right) + (1 - b) f \left(\frac{1-a}{1-b}
    \right)    & \text{for}  \quad 0 < b < 1;\\
  f(1-a) + a f'(\infty) & \text{for}  \quad b = 0;\\
  f(a) + (1 - a) f'(\infty) & \text{for}  \quad b = 1.  \\
  \end{cases}
\end{eqnarray}
The convexity of $f$ implies monotonicity and convexity properties of $\phi_f$, which is stated in the following lemma. % whose proof is relegated to the supplementary material~\citep{XiAdityaYuchenSupp}.
\begin{lemma}\label{lem:lower_property}
For each $f \in \mathcal{C}$, for every fixed $b$, the map $g(a): a \mapsto \phi_f(a,b)$ is non-increasing for $a \in [0,b]$ and $g(a)$ is convex and continuous in $a$. Further, for every fixed $a$, the map $h(b): b \mapsto \phi_f(a,b)$ is non-decreasing for $b \in  [a, 1]$.
\end{lemma}

We also define the quantity
\begin{equation}\label{R0def}
  R_0 := \inf_{a \in \ac} \int_{\Theta} L(\theta, a) w(\mathrm{d}\theta),
\end{equation}
where the decision $a$ does not depend on data $X$. Note that $R_0$ represents the Bayes risk with respect to $w$ in the ``no data'' problem i.e., when one only has information on $\Theta$, $\ac$, $L$ and the prior $w$ but not the data $X$. For simplicity, our notation for $R_0$ suppresses its dependence  on $w$. Because the loss function is zero-one valued so that $L(\theta,a)=1-\mathbb{I}(L(\theta,a)=0)$, the quantity $R_0$ has the following alternative expression:
\begin{equation}
  \label{alte}
  R_0 = 1 - \sup_{a \in \ac} w(B(a)),
\end{equation}
where
\begin{equation}\label{ba}
B(a) := \left\{\theta \in \Theta: L(\theta, a) = 0 \right\},
\end{equation}
and $w(B(a))$ is the prior mass of the ``ball" $B(a)$.
It will be important in the sequel to observe that the Bayes risk, $R_{\rm Bayes}(w)$ is bounded from above by $R_0$. This is obvious because the risk with some data cannot be greater than the risk in the no data problem (which can be viewed as an application of the data processing inequality). Formally, if $\mathcal{D}=\{\dr:  \exists a \in \ac ~~ \text{such that} ~~ \dr(x) = a~ \forall x \in \mathcal{X}\}$ is the class of the constant decision rules, then $R_0=\inf_{\dr \in \mathcal{D}} \int_{\Theta} \E_{\theta}L(\theta, \dr(X)) w(\mathrm{d}\theta) \geq R_{\rm Bayes}(w)$. Because $0 \leq R_{\rm Bayes}(w) \leq R_0$, we have $R_{\rm Bayes}(w) = 0$ when $R_0 = 0$. We shall therefore assume throughout this section that $R_0 > 0$. %\anote{Xi, can you please  check this statement. Some of our results such as the generalized Fano's inequality become meaningless when $R_0 = 0$}.

The main result of this section is presented next. It provides an implicit lower bound for the Bayes risk in terms of $R_0$ and the $f$-informativity $I_f(w, \Ps)$ for every $f \in \C$. The only assumption is that $L$ is zero-one valued and we do not assume the existence of the Bayes decision rule.
\begin{theorem}\label{man}
  Suppose that the loss function $L$ is zero-one valued. For any $f \in \C$, we have
  \begin{equation}
    \label{man.eq}
    I_f(w, \Ps) \geq \phi_f(R_{\rm Bayes}(w), R_0)
  \end{equation}
   where $\phi_f$ and $R_0$ are defined \eqref{eq:phi} and \eqref{R0def} respectively.
\end{theorem}

\begin{figure}[!t]
\centering
  \includegraphics[width=0.55\textwidth]{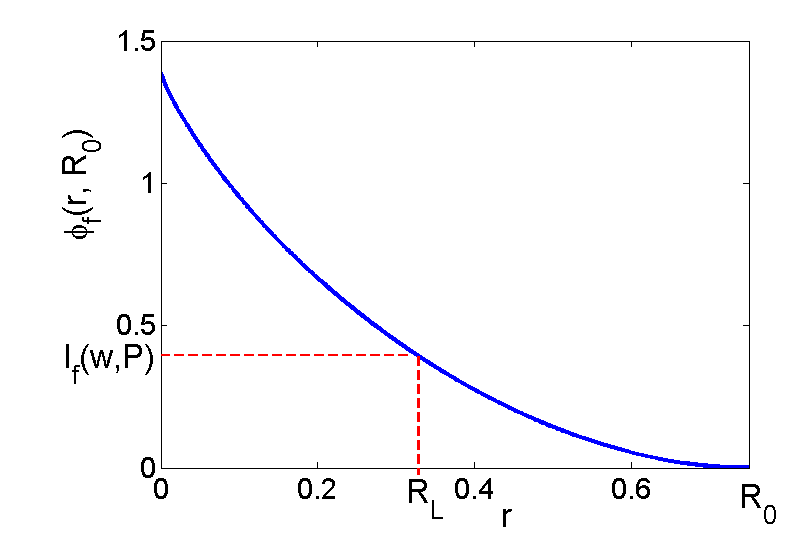}
\caption{Illustration on why \eqref{man.eq} leads to a lower bound on $R_{\rm Bayes}(w)$. Recall that $R \leq R_0$ and $r \mapsto \phi_f(r, R_0)$ is non-increasing in $r$ for $r \in [0, R_0]$.  Given $I_f(w, \mathcal{P})$ as an upper bound of $\phi_f(R_{\rm Bayes}(w), R_0)$, we  have $R_{\rm Bayes}(w) \geq R_L = g^{-1}(I_f(w, \Ps))$ and thus $R_L$ serves as a Bayes risk lower bound.
}
  \label{fig:phi_r}
  \vspace{-5mm}
\end{figure}

Before we prove Theorem \ref{man}, we first show that the inequality \eqref{man.eq} indeed provides an implicit lower bound for the Bayes risk $R := R_{\rm Bayes}(w)$ since  $R \leq R_0$  and $r \mapsto \phi_f(r, R_0)$ is non-increasing in $r$ for $r \in [0, R_0]$ (Lemma \ref{lem:lower_property}). Therefore, let $g(r) \defeq \phi_f(r, R_0)$. We have
\begin{equation}\label{eq:lb_phi_inverse}
  R_{\rm Bayes}(w) \geq g^{-1} (I_f(w, \Ps)),
\end{equation}
where $g^{-1}(x):=\inf\{0 \leq r \leq R_0, g(r) \leq x\}$ is the generalized inverse function of the non-increasing $g(r)$.
 As an illustration, we plot $\phi_f(r, R_0)$ for $f(x)=x \log x$  and the corresponding Bayes risk lower bound $ g^{-1} (I_f(w, \Ps))$ in Figure~\ref{fig:phi_r}. The lower bound \eqref{eq:lb_phi_inverse} can be immediately applied to obtain Bayes risk lower bounds when the $f$-divergence in \eqref{man.eq} is chi-squared divergence, total variation distance, or Hellinger distance (see Corollary \ref{cor:lb_01_f}). However, for the KL divergence, there is no simple form of $g^{-1}(x)$. To obtain the corresponding Bayes risk lower bound, we can invert \eqref{man.eq} by utilizing the convexity of $g(r)$, which will give a generalized Fano's inequality (see Corollary \ref{cor:Fano}). In particular, since $r \mapsto \phi_f(r, R_0)$ is convex (see Lemma \ref{lem:lower_property}),
\begin{equation*}
  \phi_f(R, R_0) \geq \phi_f(r, R_0) + \phi'_f(r-, R_0) (R - r) \qt{for every $0 < r \leq R_0$}
\end{equation*}
where $\phi_f'(r-, R_0)$ denotes the left derivative of $x \mapsto \phi_f(x, R_0)$ at $x = r$. The monotonicity of $\phi_f(r, R_0)$ in $r$ (Lemma \ref{lem:lower_property}) gives $\phi_f'(r-, R_0) \leq 0$ and we thus have,
\begin{equation*}
  R \geq r + \frac{\phi_f(R, R_0) - \phi_f(r, R_0)}{\phi_f'(r-, R_0)} \qt{for every $0 < r \leq R_0$}.
\end{equation*}
Inequality \eqref{man.eq}  $I_f(w, \Ps) \geq \phi_f(R, R_0)$ can now be used  to deduce that (note that  $\phi_f'(r-, R_0) \leq 0$)
\begin{equation}
  \label{expb}
  R \geq r + \frac{I_f(w, \Ps) - \phi_f(r, R_0)}{\phi_f'(r-, R_0)} \qt{for every $0 < r \leq R_0$}.
\end{equation}
The inequalities \eqref{eq:lb_phi_inverse} and \eqref{expb} provide general approaches to convert \eqref{man.eq} to an  explicit lower bound on $R$.

Theorem~\ref{man} is new, but its special case $\Theta = \ac =  \{1, \dots, N\}$, $L(\theta, a) := \mathbbm{I}\{\theta \neq a\}$ and the uniform prior $w$ is known (see \cite{Gushchin:03}  and \cite{Aditya:fdiv}). In such a discrete setting, $w(B(a))=1/N$ for any $a \in \ac$ and thus $R_0=1-1/N$. The proof of Theorem \ref{man} heavily relies on the following lemma, which is a consequence of the data processing inequality for $f$-divergences (see \eqref{dpi} in Section \ref{sec:pre}).

%The proof of Theorem \ref{man} is presented next. It is based on the data processing inequality for $f$-divergences.
\begin{lemma}\label{lem.dpi}
Suppose that the loss function $L$ is zero-one valued. For every $f \in \C$, every  probability measure $Q$ on $\samp$ and every decision rule $\dr$, we have
\begin{equation}\label{main.eq1}
  \int_{\Theta} D_f(P_{\theta}||Q) w(\mathrm{d}\theta) \geq \phi_f(R^{\dr}, R_Q^{\dr})
\end{equation}
where
\begin{equation}\label{eq:def_R_Q}
   R^{\dr} := \int_{\Theta} \E_{\theta} L(\theta, \dr(X)) w(\mathrm{d}\theta) ~~, ~~   R_Q^{\dr} := \int_{\samp} \int_{\Theta} L(\theta, \dr(x)) w(\mathrm{d}\theta) Q(\mathrm{d}x).
\end{equation}
\end{lemma}
We note that Lemma \ref{lem.dpi} is of independent interest, which can be applied to establish minimax lower bound as shown in the following remark.

\begin{proof}[Proof of Lemma \ref{lem.dpi}]

Let $\P$ denote the joint distribution of $\theta$ and $X$ under the prior $w$ i.e.,
$\theta \sim w$ and $X|\theta \sim P_{\theta}$. For any decision rule $\dr$, $R^{\dr}$ in \eqref{eq:def_R_Q} can be written as $R^{\dr} = \E_{\P} L(\theta,\dr(X))$. Let $\Q$ denote the joint distribution of $\theta$ and $X$ under which they are independently distributed according to $\theta \sim w$ and $X \sim Q$ respectively. The quantity $R_Q^{\dr}$ in \eqref{eq:def_R_Q} can then be written as $R_Q^{\dr}=\E_\Q L\left(\theta,  \dr(X)\right)$.

Because the loss function is zero-one valued, the function $\Gamma
(\theta, x) :=  L(\theta, \dr(x))$ maps $\Theta \times \samp$
into $\{0, 1\}$. Our strategy is to fix $f \in \C$ and apply the data processing inequality~\eqref{dpi} to the probability measures $\P, \Q$ and the mapping $\Gamma$. This gives
\begin{equation}\label{dp}
  D_f(\P||\Q) \geq D_f(\P\Gamma^{-1}||\Q \Gamma^{-1}),
\end{equation}
where $\P\Gamma^{-1}$ and $\Q \Gamma^{-1}$ are induced measures on the space $\{0,1\}$ of $\Gamma$. In other words, since $L$ is zero-one valued, both $\P\Gamma^{-1}$ and $\Q \Gamma^{-1}$ are two-point distributions on $\{0,1\}$ with
\begin{equation*}
  \P \Gamma^{-1} \{1\} = \int \Gamma d\P = \E_{\P} L(\theta,
  \dr(X)) = R^{\dr}, \qquad \Q\Gamma^{-1}\{1\}= \int \Gamma d\Q= R_Q^{\dr}.
\end{equation*}
By the definition of the function $\phi_f(\cdot, \cdot)$, it follows that $D_f(\P\Gamma^{-1}||\Q \Gamma^{-1}) = \phi_f (R^{\dr}, R_Q^{\dr})$. It is also easy to see  $D_f(\P||\Q) = \int_{\Theta} D_f(P_{\theta}||Q) w(\mathrm{d}\theta)$. Combining this equation with inequality~\eqref{dp} establishes inequality~\eqref{main.eq1}.
\end{proof}

With Lemma \ref{lem.dpi} in place, we are ready to prove Theorem \ref{man}.

\begin{proof}[Proof of Theorem \ref{man}]

We write $R$ as a shorthand notation of $R_{\rm Bayes}(w)$. By the definition \eqref{eq:def_f_info} of $I_f(w, \Ps)$, it suffices to prove that
\begin{equation}\label{jia}
  \int D_f(P_{\theta} \| Q) w(d\theta) \geq \phi_f(R, R_0)
\end{equation}
for every probability measure $Q$.

Notice that $R \leq  R_0$. If $R = R_0$, then the right hand side of \eqref{man.eq} is zero and hence the inequality immediately holds. Assume that $R < R_0$. Let $\epsilon > 0$ be small enough so that $R + \epsilon < R_0$. Let $\dr$ denote any decision rule for which  $R \leq R^{\dr} < R+\epsilon$ and note that such a rule exists since $R=\inf_{\dr} R^{\dr}$. It is easy to see that
  \begin{equation*}
    R_Q^{\dr} = \int_{\samp} \int_{\Theta} L(\theta, \dr(x))
    w(\mathrm{d}\theta) Q(dx) \geq \int_{\samp}\left( \inf_{a \in \ac} \int_{\Theta }L(\theta, a)
    w(\mathrm{d}\theta) \right) Q(dx) = R_0.
  \end{equation*}
  We thus have $R \leq R^{\dr} < R + \epsilon < R_0 \leq R_Q^{\dr}$. By Lemma  \ref{lem.dpi}, we have
  \begin{equation*}
    \int_{\Theta} D_f(P_{\theta} \| Q) w(\mathrm{d}\theta) \geq \phi_f(R^{\dr},
    R_Q^{\dr}).
  \end{equation*}
Because $x \mapsto \phi_f(x, R_{Q}^{\dr})$ is non-increasing on $x \in
[0, R_Q^{\dr}]$, we have
\begin{equation*}
  \phi_f(R^{\dr}, R_Q^{\dr}) \geq \phi_f(R + \epsilon, R_Q^{\dr}).
\end{equation*}
Because $x \mapsto \phi_f(R + \epsilon, x)$ is non-decreasing
on $x \in [R + \epsilon, 1]$, we have
\begin{equation*}
  \phi_f(R + \epsilon, R_Q^{\dr}) \geq \phi_f(R + \epsilon, R_0).
\end{equation*}
Combining the above three inequalities, we have
\begin{equation*}
  \int_{\Theta} D_f(P_{\theta} \| Q) w(d\theta) \geq \phi_f(R^{\dr},
    R_Q^{\dr}) \geq \phi_f(R + \epsilon, R_Q^{\dr}) \geq \phi_f(R + \epsilon, R_0).
\end{equation*}
The proof of \eqref{jia} completes by letting $\epsilon \downarrow 0$ and  using the continuity of $\phi_f(\cdot, R_0)$ (continuity was noted in Lemma  \ref{lem:lower_property}). This completes the proof of Theorem \ref{man}.
\end{proof}

\begin{remark}\label{rem:minimax}
  Lemma \ref{lem.dpi} can also be used to derive minimax lower bounds in a different  way.  For example, when the minimax decision rule $\dr$ exists (e.g., for finite space  $\Theta$ and $\ac$ \citep{Ferguson76decision}), we have $R^{\dr} \leq R_{\rm minimax}$. If the probability measure $Q$ is chosen so that $R_{\rm minimax} \leq R_Q^{\dr}$, then, by Lemma~\ref{lem:lower_property}, the right hand side of~\eqref{man.eq} can be lower bounded by replacing $R^{\dr}$ with $R_{\rm minimax}$  which yields
\begin{equation}\label{eq:ext_main.eq_1}
  \int_{\Theta} D_f(P_{\theta}||Q) w(\mathrm{d}\theta) \geq \phi_f(R_{\rm minimax}, R_Q^{\dr}).
\end{equation}
Similarly, this inequality can be converted to an explicit lower bound on  minimax risk. We will show an application of this inequality in deriving Birg\'{e}-Gushchin inequality \citep{Gushchin:03, Birge:ineq}  in Section \ref{beig}.
\end{remark}

\subsection{Generalized Fano's Inequality}
In the next result, we derive the generalized Fano's ienquality  \eqref{intro_gf} using Theorem \ref{man}. The inequality proved here is in fact slightly stronger than \eqref{intro_gf}; see Remark \ref{pach} for the clarification.

\begin{corollary}[Generalized Fano's inequality]\label{cor:Fano}
For any given prior $w$ and zero-one loss $L$, we have
\begin{equation}\label{gf}
  R_{\rm Bayes}(w, L; \Theta) \geq 1 + \frac{I(w, \mathcal{P}) + \log (1 + R_0)}{\log \left(\sup_{a \in \ac} w(B(a))\right)},
\end{equation}
where $B(a)$ is defined in \eqref{ba}.
\end{corollary}

\begin{proof}[Proof of Corollary \ref{cor:Fano}]

We simply apply \eqref{expb} to $f(x) = x \log x$ and $r = R_0/(1 + R_0)$, it can then be checked that
\[
\phi_f(r, R_0)= - \log (1+R_0) - \frac{1}{1+R_0} \log(1-R_0), \quad \phi_f'(r-, R_0)= \log (1-R_0),
\]
Inequality \eqref{expb} then gives
\begin{equation*}\label{eq:Fano1}
  R \geq 1 + \frac{I(w, \Ps) + \log (1 + R_0)}{\log (1 - R_0)}
\end{equation*}
which proves \eqref{gf}.
\end{proof}

%To prove this corollary, we simply  apply \eqref{expb} to $f(x) = x \log x$ and $r = R_0/(1 + R_0)$ and detailed calculations are provided in Section \ref{sec:proof_fano} in the appendix. % We note that this generalized Fano's inequality in \eqref{gf} is slightly stronger than \eqref{intro_gf} because $R_0 \leq 1$ (thus $\log(1+R_0) \leq \log 2$). For example, when $\Theta = \ac = \{0, 1\}$, $L(\theta, a) := \mathbbm{I}\{\theta \neq a\}$ and $w\{0\} = w\{1\} = 1/2$, the inequality \eqref{intro_gf} leads to a trivial bound since the right hand side of \eqref{intro_gf} is negative. However, since $R_0=1/2$, the inequality \eqref{gf} still provides a useful lower bound when $I(w, \Ps)$ is strictly  smaller than $\log 2 -\log(3/2)$.

\begin{remark}\label{pach}
%Because $R_0$ has the alternative expression \eqref{alte}, inequality \eqref{gf} is identical to
%  \begin{equation*}
%  R_{\rm Bayes}(w, L; \Theta) \geq 1 + \frac{I(w, \mathcal{P}) + \log (1 + R_0)}{\log \left(\sup_{a \in \ac} w \{\theta \in \Theta : L(\theta, a) = 0\} \right)}.
%\end{equation*}
This inequality is slightly stronger than \eqref{intro_gf} because $R_0 \leq 1$ (thus $\log(1+R_0) \leq \log 2$). For example, when $\Theta = \ac = \{0, 1\}$, $L(\theta, a) := \mathbbm{I}\{\theta \neq a\}$ and $w\{0\} = w\{1\} = 1/2$, the inequality \eqref{intro_gf} leads to a trivial bound since the right hand side of \eqref{intro_gf} is negative. However, since $R_0=1/2$, the inequality \eqref{gf} still provides a useful lower bound when $I(w, \Ps)$ is strictly  smaller than $\log 2 -\log(3/2)$.
\end{remark}

%\begin{remark}\label{pach}
%Because $R_0$ has the alternative expression \eqref{alte}, inequality \eqref{gf} is identical to
%  \begin{equation*}
%  R_{\rm Bayes}(w, L; \Theta) \geq 1 + \frac{I(w, \mathcal{P}) + \log (1 + R_0)}{\log \left(\sup_{a \in \ac} w \{\theta \in \Theta : L(\theta, a) = 0\} \right)}.
%\end{equation*}
%\end{remark}

As mentioned in the introduction, the classical Fano inequality \eqref{eq:classcial_Fano} and the recent continuum Fano inequality \eqref{cf} are both special cases (restricted to uniform priors) of Corollary \ref{cor:Fano}. The proof of \eqref{cf} given in \cite{Duchi:Fano} is rather complicated
with a stronger assumption and a discretization-approximation argument. Our proof based on Theorem \ref{man} is much simpler. Lemma \ref{lem.dpi} also has its independent interest. Using Lemma \ref{lem.dpi}, we are able to recover another recently proposed variant of Fano's inequality in \citet[Proposition 2.2]{BraunPokutta}. Details of this argument are provided in Appendix  \ref{sec:supp_Fano}.

\subsection{Specialization of Theorem \ref{man} to Different $f$-divergences and Their Applications}
In addition to the generalized Fano's inequality, Theorem \ref{man} allows us to derive a class of lower bounds on Bayes risk for zero-one losses by plugging other $f$-divergences. In the next corollary, we consider some widely used $f$-divergences  and provide the corresponding Bayes risk lower bounds by inverting \eqref{man.eq} in Theorem \ref{man}.
\begin{corollary}\label{cor:lb_01_f}
  Let $L$ be zero-one valued, $w$ be any prior on $\Theta$ and $R = R_{\rm Bayes}(w, L, \Theta)$. We then have the following inequalities
  \begin{enumerate}[(i)]
    \item Chi-squared divergence:
      \begin{equation}
        \label{eq:chi}
        R \geq R_0 - \sqrt{R_0(1 - R_0) I_{\chi^2}(w, \Ps)} .
      \end{equation}
    \item Total variation distance:
      \begin{equation}
    R \geq R_0 - I_{TV}(w, \Ps).
      \label{eq:total}
    \end{equation}
    \item Hellinger distance:
      \begin{equation}
        \label{eq:hellinger_3}
        R \geq R_0 - (2 R_0 - 1) \frac{h^2}{2} - \sqrt{R_0(1 - R_0) h^2 (2 - h^2)}.
      \end{equation}
      provided $h^2 \leq 2 R_0$. Here $
        h^2=\int_{\Theta} \int_{\Theta}  H^2(P_{\theta}\| P_{\theta'})  w(\mathrm{d} \theta)w(\mathrm{d} \theta').
     $
\end{enumerate}
\end{corollary}

See Appendix \ref{colr} for the proof of the corollary. The special case of Corollary \ref{cor:lb_01_f} for $\Theta = \ac = \{1, \dots, N\}$, $L(\theta, a) = \mathbbm{I}\{\theta \neq a\}$ and $w$ being the uniform prior has been discovered previously in \cite{Aditya:fdiv}.  It is clear from Corollary~\ref{cor:lb_01_f} that the choice of $f$-divergence will affect the tightness of the lower bound for $R$. In Appendix \ref{sec:compare}, we provide a qualitative comparison of the lower bounds~\eqref{gf}, \eqref{eq:chi} and \eqref{eq:hellinger_3}. In particular, we show that in the discrete setting with $\Theta = \ac = \{1, \dots, N\}$, the lower bounds induced by the KL divergence and the chi-squared divergence are much stronger than the bounds given by the Hellinger distance. Therefore, in most applications in this paper, we shall only use the bounds involving the KL divergence and the chi-squared divergence.

 Corollary \ref{cor:lb_01_f} can be used to recover classical inequalities of Le Cam (for two point hypotheses) and Assouad (Theorem 2.12 in \cite{Tsybakov:nonpara} with both total variation distance and Hellinger distance) and Theorem 2.15 in \cite{Tsybakov:nonpara} that involves fuzzy hypotheses. The details are presented in Appendix \ref{sec:supp_LeCam}.

\subsection{Birg\'e-Gushchin's Inequality}\label{beig}

%Theorem \ref{man}, Lemma \ref{lem.dpi}, and their corollaries can be used to obtain several widely used minimax lower bounds.  As the first example,
In this section, we expand \eqref{eq:ext_main.eq_1} in Remark \ref{rem:minimax} to obtain a minimax risk lower bound due to \cite{Gushchin:03} and \cite{Birge:ineq}, which presents an improvement of the classical Fano's inequality when specializing to KL divergence.
\begin{proposition}\citep{Gushchin:03, Birge:ineq}\label{prop:Birge}
  Consider the finite parameter and action space $\Theta=\ac=\{\theta_0, \theta_1, \ldots, \theta_{N} \}$  and the zero-one valued indicator loss $L(\theta, a) = \mathbbm{I}\{\theta \neq a\}$,  for any $f$-divergence,
\begin{equation}
  \phi_f(R_{\rm minimax}, 1 - R_{\rm minimax}/N) \leq \min_{0 \leq j \leq N} \frac{1}{N} \sum_{i: i\neq j} D_f\left(P_{\theta_i} || P_{\theta_j}  \right).
  \label{eq:Birge}
\end{equation}
\end{proposition}

\begin{proof}[Proof of Proposition \ref{prop:Birge}]

To prove Proposition \ref{prop:Birge}, it is enough to prove that
$\frac{1}{N}\sum_{i:i \neq j} D_f(P_{\theta_i}||P_{\theta_j}) \geq
  \phi_f(R_{\rm minimax}, 1 - R_{\rm minimax}/N) $ for every $j \in \{0, \dots,
N\}$. Without loss of generality, we assume that $j = 0$. We
apply~\eqref{main.eq1} with the uniform distribution on $\Theta
\setminus \{\theta_0\} = \{\theta_1, \dots, \theta_N\}$ as $w$, $Q =
P_{\theta_0}$ and the minimax rule for the problem as $\dr$. Because
$\dr$ is the minimax rule, $R^{\dr} \leq R_{\rm minimax}$. Also
\begin{equation*}
  R_Q^{\dr} = \frac{1}{N} \sum_{i=1}^N \E_{\theta_0} L(\theta_i, \dr(X)) = \frac{1}{N}\E_{\theta_0}  \sum_{i=1}^N \mathbbm{I} \{\theta_i \neq \dr(X)\}.
\end{equation*}
It is easy to verify that $\sum_{i=1}^N \mathbbm{I} \{\theta_i \neq \dr(X)\} = N -\mathbbm{I} \{\theta_0 \neq \dr(X)\}$. We thus have $R_Q^{\dr} = 1 - \E_{\theta_0} L(\theta_0, \dr(X))/N$. Because $\dr$ is minimax, $\E_{\theta_0} L(\theta_0, \dr(X)) \leq R_{\rm minimax}$  and thus
\begin{equation}\label{dan}
R_Q^{\dr} \geq 1 - R_{\rm minimax}/N.
\end{equation}

On the other hand, we have $R_{\rm minimax} \leq N/(N+1)$. To see
this, note that the minimax risk is upper bounded by the maximum risk
of a random decision rule, which chooses among the $N+1$ hypotheses
uniformly  at random. For this random decision rule, its risk is
$\frac{N}{N+1}$ no matter what the true hypothesis is. Thus,
$\frac{N}{N+1}$ is an upper bound on the minimax risk. We thus have,
from~\eqref{dan}, that $R_Q^{\dr} \geq 1 -R_{\rm minimax}/N \geq
R_{\rm minimax}$. We can thus apply~\eqref{eq:ext_main.eq_1} to
obtain
\begin{equation*}
  \frac{1}{N} \sum_{i=1}^N D_f(P_{\theta_i}||P_{\theta_0}) \geq
  \phi_f(R_{\rm minimax}, 1 - R_{\rm minimax}/N).
\end{equation*}
which completes the proof Proposition \ref{prop:Birge}.
\end{proof}

\section{Bayes Risk Lower Bounds for Nonnegative Loss Functions}
\label{sec:Bayes_general}

\begin{figure}[!t]
\centering
  \begin{subfigure}{.4\textwidth}
  \centering
  \includegraphics[width=\textwidth]{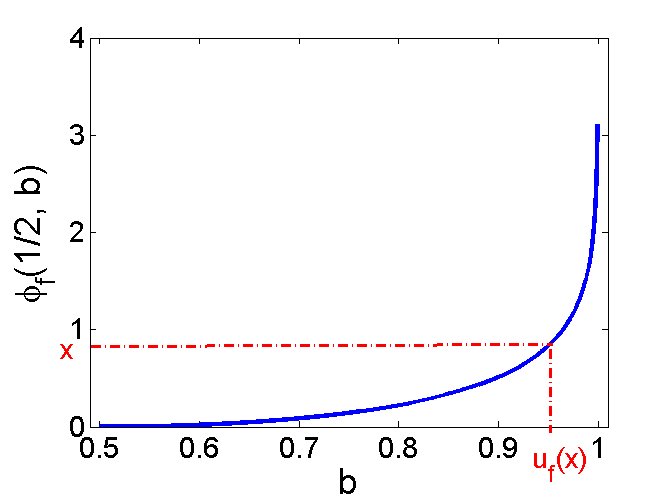}
  \caption{$\phi_f(1/2, b)$}
  \label{fig:phi}
\end{subfigure}\hspace{0.1cm}
\begin{subfigure}{.4\textwidth}
  \centering
  \includegraphics[width=\textwidth]{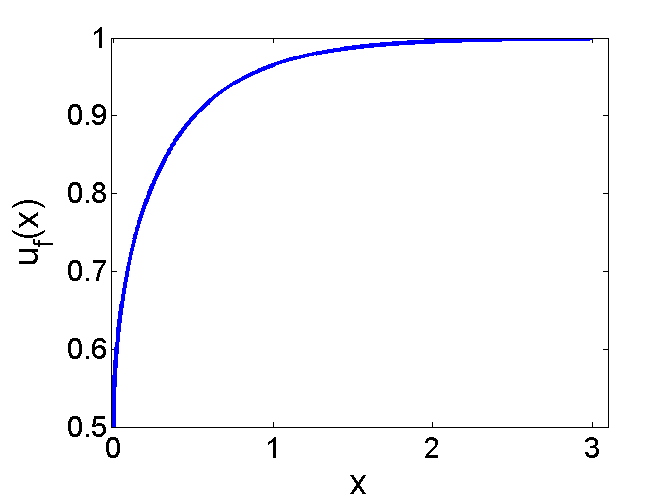}
  \caption{$u_f(x)$}
  \label{fig:u}
\end{subfigure}
\caption{Illustration of $\phi_f(1/2,b)$ and $u_f(x)$ for $f(x)=x\log x$.}
\label{fig:phi_u}\vspace{-5mm}
\end{figure}
In the previous section, we discussed Bayes risk lower bounds for zero-one valued loss functions. We deal with general nonnegative loss functions in this section. The main result of this section, Theorem \ref{bwl}, provides lower bounds for $R_{\rm Bayes}(w, L; \Theta)$ for any given loss $L$ and prior $w$. To state this result, we need the following notion.
%For each $f \in \mathcal{C}$, for each fixed $b$, the map $a \mapsto \phi_f(a,b)$ is non-increasing for $a\in [0,b]$ while for each fixed $a$, the map $b \mapsto \phi_f(a,b)$ is non-decreasing for $b\in [a, 1]$. These properties are summarised
%in Lemma~\ref{lem:lower_property} (stated and proved in the Appendix).
Fix $f \in \C$ and recall the definition of $\phi_f$ in \eqref{eq:phi}. We define   $u_f: [0, \infty) \mapsto [1/2,1]$ by
\begin{equation}\label{ufi}
  u_f(x) := \inf \left\{1/2 \leq b \leq 1 : \phi_f(1/2, b) > x \right\}
\end{equation}
and if $\phi_f(1/2, b) \leq x$ for every $b \in [1/2, 1]$, then we take $u_f(x)$
to be 1. By Lemma \ref{lem:lower_property}, it is easy to see that $u_f(x)$ is a non-decreasing function of $x$. For example, for KL-divergence with $f(x)= x \log x$, we have $\phi_f(1/2, b)=\frac{1}{2} \log \frac{1}{4b(1-b)}$ and $u_f(x) = \frac{1}{2}+ \frac{1}{2} \sqrt{1-e^{-2x}}$ (see Figure \ref{fig:phi_u}). We are now ready to state the main theorem of this paper.
\begin{theorem}\label{bwl}
For every $\Theta, \ac, L, w$ and $f \in \C$, we have
\begin{equation}
  R_{\rm Bayes}(w, L; \Theta)  \geq \frac{1}{2} \sup\left\{t > 0 :  \sup_{a \in \mathcal{A}}  w(B_t(a,L))  < 1- u_f(I_f(w, \Ps)) \right\},
  \label{eq:main_bwl}
\end{equation}
where
\begin{equation}  \label{bta}
B_t(a, L) := \{\theta \in \Theta: L(\theta, a) < t\} \qt{for $a \in \ac$ and $t > 0$}.
\end{equation}
\end{theorem}

\begin{proof}[Proof of Theorem \ref{bwl}]

Fix $\Theta, \ac, L, w$ and $f$. Let $I := I_f(w, \Ps)$ be a shorthand notation. Suppose $t > 0$ is such that
  \begin{eqnarray}
  \label{eq:cond_t}
  \sup_{a \in \ac} w \left(B_t(a,L)\right) < 1-u_f(I).
  \end{eqnarray}
We prove below that $R_{\rm Bayes}(w, L; \Theta) \geq t/2$ and this would complete the proof. Let $L_{t}$ denote the zero-one valued loss function $L_{t}(\theta, a) :=  \mathbbm{I} \left\{L(\theta, a) \geq t \right\}$. It is obvious that $L \geq t L_{t}$ and hence the proof will be complete if we establish that $R_{\rm Bayes}(w, L_{t}; \Theta) \geq 1/2$.  Let $R := R_{\rm Bayes}(w, L_{t}; \Theta)$ for a shorthand notation.

Because $L_t$ is a zero-one valued loss function, Theorem \ref{man} gives
\begin{equation}\label{sb}
  I \geq \phi_f(R, R_0) \qt{where $R_0 = 1 - \sup_{a \in \ac} w \left( B_t(a, L) \right)$}.
\end{equation}
By \eqref{eq:cond_t}, it then follows that $R_0 > u_f(I)$. By definition of $u_f(\cdot)$, it is clear that there exists $b^* \in [1/2, R_0)$ such that $\phi(1/2, b^*) > I$ (this in particular implies that $R_0 \geq 1/2$).  Lemma \ref{lem:lower_property} implies that $b \mapsto \phi_f(1/2, b)$ is non-decreasing for $b\in[1/2, 1]$, which yields $\phi_f(1/2, b^*) \leq \phi_f(1/2, R_0)$. The above two inequalities imply $I <  \phi_f(1/2, R_0)$. Combining this inequality with~\eqref{sb}, we have
\begin{equation*}
  \phi_f(1/2, R_0) > I \geq \phi_f(R, R_0).
\end{equation*}
Lemma \ref{lem:lower_property} shows that  $a \mapsto \phi_f(a, R_0)$ is non-increasing for $a \in [0, R_0]$. Thus, we have $R \geq 1/2$.
\end{proof}

We further note that because $u_f(x)$ is non-decreasing in $x$, one can replace $I_f(w, \Ps)$ in \eqref{eq:main_bwl} by any upper bound $\Iu$ i.e., for any $\Iu \geq I_f(w, \Ps)$, we have
\begin{equation}\label{upbo}
  R_{\rm Bayes}(w, L; \Theta)  \geq \frac{1}{2} \sup\left\{t > 0 :  \sup_{a \in \mathcal{A}}  w(B_t(a,L))  < 1- u_f(\Iu) \right\}.
\end{equation}
This is useful since $I_f(w, \Ps)$ is often difficult to calculate exactly. When $f(x)=x\log x$, \cite{haussler1997} provided a useful upper bound on the mutual information $I(w, \Ps)$. We describe this result in Section \ref{sec:upper_f_informativity} where we also extend it to power divergences $f_\alpha$ for  $\alpha \not \in [0,1]$ (which covers the case of chi-squared divergence).

\begin{remark}
From the proof of Theorem \ref{bwl}, it can be observed that the constant $1/2$ in the right hand side of \eqref{eq:main_bwl} and in the definition of $u_f(\cdot)$ can be replaced by any $c \in (0,1]$. This gives the sharper lower bound:
\small
\begin{eqnarray*}
     R_{\rm Bayes}(w, L; \Theta) \geq \sup_{c \in (0,1]} \left( c \sup \left\{t > 0 : \sup_{a
      \in \ac} w \left(B_t(a,L) \right) < 1- u_{f,c}(I_f(w, \Ps)) \right\} \right),
\end{eqnarray*}
\normalsize
where $u_{f,c}(x)=\inf\{c\leq b \leq 1: \phi_f(c, b) \geq x \}$. Since obtaining exact constants is not our main concern, the inequality \eqref{eq:main_bwl} is usually sufficient to provide Bayes risk lower bounds with correct dependence on the model and prior.
\end{remark}

\begin{remark}\label{rem:Bayes_decomp}
%\anote{This remark is not very clear. Needs to be edited a bit more. For example, the method of decomposition in the concrete example is different from the general  method of decomposition etc.}

We note that the lower bound presented in Theorem \ref{bwl} might not be tight for some special priors, e.g., when the prior $w$ has extremely large density in some small region of the parameter space. We call such regions with unbounded density as \emph{spikes} in the prior distribution. As a concrete example, let $\Theta = \ac$ be a subset of a  finite dimensional Euclidean space containing the origin with $L$ being the Euclidean distance and let $w$ denote the mixture of the uniform priors over the balls $B_1(0, L)$ and $B_\epsilon(0, L)$ for some very small $0 < \epsilon \ll 1$. In this case, the mixture component $B_\epsilon(0, L)$ is a spike. If $\epsilon$ is very small, then the term $\sup_{a \in \mathcal{A}} w(B_t(a,L))$ might be too big for Theorem \ref{bwl} to establish a tight lower bound.

Even in such extreme cases, the tight lower bound can be salvaged by partitioning the parameter space $\Theta$ into finite or countably many disjoint subsets $\Theta_i, i \geq 0$ and to apply Theorem \ref{bwl} to $w$ restricted to each $\Theta_i$. To illustrate this technique, suppose that $w$ has a  Lebesgue density  $\varphi$ that is bounded from above. Let $\varphi_{\max}$ denote the supremum of $\varphi$. We partition the parameter space $\Theta$ into disjoint subsets $\Theta_0,\Theta_1,\dots$ with
 \begin{equation}\label{eq:space_part}
	\Theta_i := \{\theta\in\Theta: 2^{-(i+1)} \varphi_{\max} < \varphi(\theta) \leq 2^{-i} \varphi_{\max} \}.
 \end{equation}
 Then, we apply Theorem \ref{bwl} to $w$ restricted to each $\Theta_i$. More specifically, let $w_i$ denote the probability measure $w$ restricted to $\Theta_i$ i.e., $w_i(S) := w(S \cap \Theta_i)/w(\Theta_i)$ for any measurable set $S \subseteq  \Theta_i$.  we have
\begin{equation}\label{eq:prior_decomp}
  R_{\rm Bayes}(w, L; \Theta) \geq \sum_{i} w(\Theta_i) R_{\rm Bayes}(w_i, L; \Theta_i),
\end{equation}
where $R_{\rm Bayes}(w_i, L; \Theta_i)=\inf_{\dr} \int_{\Theta_i} \E_{\theta} L(\theta, \dr(X)) w_i(\mathrm{d}\theta)$.
To see this, for any decision rule $\dr$, we have
$
    	R^\dr(w, L; \Theta) = \sum_{i=1}^\infty w(\Theta_i)R^{\dr} (w_i, L; \Theta_i);
$
then take infimum over all possible $\dr$ on both sides,
\begin{multline*}
 R_{\rm Bayes}(w, L;\Theta) = \inf_{\dr} R^\dr(w, L; \Theta) \\ \geq \sum_{i=1}^\infty w(\Theta_i) \inf_{\dr}  R^{\dr} (w_i, L; \Theta_i)=\sum_{i=1}^\infty w(\Theta_i) R_{\rm Bayes}(w_i, L; \Theta_i)
\end{multline*}
One can lower bound each Bayes risk $R_{\rm Bayes}(w_i, L; \Theta_i)$ for all $i$ using Theorem \ref{bwl}. Since the density of $w_i$ differs by a factor at most $2$, the spiking prior problem will no longer exist while applying Theorem \ref{bwl} for $w_i$. We also note that another useful application of such a partitioning technique is presented in Corollary \ref{nd}.

Now take the concrete example of the mixture of the uniform priors over  $B_1:=B_1(0, L)$ and $B_\epsilon:= B_\epsilon(0, L)$.  It is clear from \eqref{eq:space_part} that $\Theta_0=B_\epsilon$ and $\Theta_k=B_1\backslash B_\epsilon$ for some $k>0$ and the rest of $\Theta_i$'s are empty sets.
Applying \eqref{eq:prior_decomp},  we have
    \begin{align*}
      R_{\rm Bayes}(w, L; \Theta) \geq & w(B_{\epsilon}) R_{\rm Bayes}(w_1, L; B_{\epsilon})  + w(B_1 \backslash B_{\epsilon}) R_{\rm Bayes}(w_2, L; B_1  \backslash B_{\epsilon})  \\
      \geq & w(B_1 \backslash B_{\epsilon}) R_{\rm Bayes}(w_2, L; B_1  \backslash B_{\epsilon})
    \end{align*}
Note that $w(B_1 \backslash B_{\epsilon})$ is lower bounded by a universal constant. Then we can lower bound $R_{\rm Bayes}(w_2, L; B_1  \backslash B_{\epsilon})$ using Theorem \ref{bwl} and obtain a tight lower bound up to a constant factor that is independent of $\epsilon$ (see an example of deriving Bayes risk lower bound for estimating the mean of a Gaussian model with uniform prior on a ball in Section~\ref{sec:upper_f_informativity}).

\end{remark}

For specific $f \in \C$, the right hand side of \eqref{upbo} can be explicitly evaluated as shown in the next corollary.
 %This is the next corollary whose proof is given in Appendix \ref{sec:supp_bwl}.

\begin{corollary}\label{cor:bwl}
Fix $\Theta, \ac, L, w$ and $\Ps$. The Bayes risk $R_{\rm Bayes}(w, L; \Theta)$ satisfies each of the following inequalities (the quantity $\Iu$ represents an upper bound on the corresponding $f$-informativity):
  \begin{enumerate}[(i)]
    \item KL divergence:
    \small
       \begin{equation}
  R_{\rm Bayes}(w,L;\Theta) \geq \frac{1}{2} \sup \left\{t > 0 : \sup_{a
      \in \ac} w\left(B_t(a,L)\right)< \frac{1}{4} e^{-2 \Iu} \right\}.
    \label{eq:I_bayes_KL}
\end{equation}
\normalsize
    \item Chi-squared divergence:
    \small
      \begin{equation}
    R_{\rm Bayes}(w, L;\Theta) \geq \frac{1}{2} \sup \left\{t > 0 :
      \sup_{a \in \ac} w \left( B_t(a,L) \right) < \frac{1}{4\left(1+ \Iu\right)} \right\}.
        \label{eq:I_bayes_chi}
  \end{equation}
  \normalsize
    \item Total variation distance:
    \small
      \begin{equation}
    R_{\rm Bayes}(w, L;\Theta) \geq \frac{1}{2} \sup \left\{t> 0 :
      \sup_{a \in \ac} w \left(B_t(a,L) \right)  < \frac{1}{2} - \Iu \right\}.
      \label{eq:I_bayes_TV}
  \end{equation}
  \normalsize
    \item Hellinger distance:
    If $\Iu < 1-1/\sqrt{2}$, then we have
    \small
  \begin{equation}
    R_{\rm Bayes}(w,L;\Theta) \geq \frac{1}{2} \sup \left\{t > 0 :
      \sup_{a \in \ac} w \left(B_t(a,L) \right) < \frac{1}{2} - \left(1- \Iu\right)\sqrt{ \Iu\left(2- \Iu\right)} \right\}.
      \label{eq:I_bayes_Hellinger}
  \end{equation}
  \normalsize
  \end{enumerate}
\end{corollary}

\begin{proof}[Proof of Corollary \ref{cor:bwl}]

Inequality \eqref{eq:I_bayes_KL} involving KL divergence:  Suppose $f(x) = x\log x$ so that $D_f(P||Q) = D(P||Q)$ equals the KL divergence. Then the function $u_f(x)$ in~\eqref{ufi}
has the expression for all $x>0$,
\begin{equation*}
  u_f(x) = \inf \left\{1/2 \leq b \leq 1: b(1-b) < e^{-2x}/4 \right\}
  = \frac{1}{2} + \frac{1}{2} \sqrt{1 - e^{-2x}}.
\end{equation*}
The elementary inequality $\sqrt{1 - a} \leq 1 - a/2$ gives for all $x>0$,
\begin{equation*}
  u_f(x) \leq 1 - \frac{1}{4} e^{-2x}.
\end{equation*}
Inequality~\eqref{eq:main_bwl} reduces to the desired inequality \eqref{eq:I_bayes_KL}:
\begin{equation*}
  R_{\rm Bayes}(w,L; \Theta) \geq \frac{1}{2} \sup \left\{t > 0 : \sup_{a
      \in \ac} w\left(B_t(a,L)\right)< \frac{1}{4} e^{-2 \Iu} \right\}.
\end{equation*}

The proof of the Bayes risk lower bounds for  the other three $f$-divergences are similar and thus we only present the form of $u_f(x)$.  Inequality \eqref{eq:I_bayes_chi} involves chi-squared divergence with $f(x)=x^2-1$. Therefore, we have  for all $x >0$,
  \begin{equation*}
    u_f(x) = \inf \left\{1/2 \leq b \leq 1 : \frac{(1 -
        2b)^2}{4b(1-b)} > x \right\} = \frac{1}{2} + \frac{1}{2}
    \sqrt{\frac{x}{1+x}}  \leq 1- \frac{1}{4(1+x)}.
  \end{equation*}

  Inequality \eqref{eq:I_bayes_TV} involves total variation distance with $f(x) = |x-1|/2$. Then
  \begin{equation*}
    u_f(x) = \inf \left\{1/2 \leq b \leq 1 : |1 - 2b| > 2x \right\} =
    \frac{1}{2} + x.
  \end{equation*}

  Inequality \eqref{eq:I_bayes_Hellinger} involves Hellinger divergence with $f(x)=1-\sqrt{x}$ and  thus
\begin{align*}
  u_f(x) & = \inf \left\{1/2 \leq b \leq 1: 1-\sqrt{b/2} -\sqrt{(1-b)/2}  >x \right\} \\
        &= \begin{cases}
    1  & \text{if} \; x\geq 1-1/\sqrt{2} \\
     \frac{1}{2} + (1-x) \sqrt{x(2-x)} & \text{if} \;  x < 1-1/\sqrt{2}.
  \end{cases}
\end{align*}
\end{proof}

\begin{remark}\label{zhc}
  A special case of Corollary \ref{cor:bwl}(i) appeared as
\citet[Theorem 6.1]{zhang2006information}. To see that \citet[Theorem 6.1]{zhang2006information} is indeed a special case of \eqref{eq:I_bayes_KL}, note first that \eqref{eq:I_bayes_KL} is equivalent to
  \begin{equation}\label{zex}
    R_{\text{Bayes}}(w, L; \Theta) \geq \frac{1}{2} \sup \left\{t > 0
      : \inf_{a \in \ac} \frac{1}{w (B_t(a, L))} > 2 I^{\text{up}} +
      \log 4 \right\}.
  \end{equation}
  Here $I^{\text{up}}$ is any upper bound on the mutual
  information. One such upper bound on the mutual information is
  \begin{equation}\label{upp}
    I^{\text{up}} = \int_{\Theta} \int_{\Theta} D(P_{\theta} \|
    P_{\xi}) w(d\xi) w(d \theta)
  \end{equation}
   That $I^{\text{up}}$ is an upper bound on the mutual information
   can be seen for example by using concavity of the logarithm \eqref{eq:upper_KL_1} when the family $\{Q_{\xi}, \xi \in \Xi\}$ is chosen to be the same as $\{P_{\theta}, \theta \in \Theta\}$. Using \eqref{upp} in
   \eqref{zex}, we obtain
   \begin{equation*}
     R_{\text{Bayes}}(w, L; \Theta) \geq \frac{1}{2} \sup \left\{t > 0
      : \inf_{a \in \ac} \frac{1}{w (B_t(a, L))} > 2 \int_{\Theta} \int_{\Theta} D(P_{\theta} \|
    P_{\xi}) w(d\xi) w(d \theta)  + \log 4 \right\}.
   \end{equation*}
   If we now specialize to the setting when the probability measures
   $\{P_{\theta}, \theta \in \Theta \}$ are all $n$-fold product
   measures i.e., when each $P_{\theta}$ is of the form
   $\mathfrak{P}_{\theta}^n$ for some class of probabilities
   $\{\mathfrak{P}_{\theta}, \theta \in \Theta\}$, then the inequality
   becomes
   \begin{equation*}
     R_{\text{Bayes}}(w, L; \Theta) \geq \frac{1}{2} \sup \left\{t > 0
      : \inf_{a \in \ac} \frac{1}{w (B_t(a, L))} > 2 n \int_{\Theta}
      \int_{\Theta} D(\mathfrak{P}_{\theta} \| \mathfrak{P}_{\xi})
      w(d\xi) w(d \theta)  + \log 4 \right\}.
   \end{equation*}
    This inequality is precisely \citet[Theorem
    6.1]{zhang2006information}.
\end{remark}

\section{Upper Bounds on $f$-informativity and Examples}
\label{sec:upper_f_informativity}
Application of Theorem~\ref{bwl} requires upper bounds on the $f$-informativity $I_f(w; \Ps)$. This is the subject of this section. We focus on the power divergence $f_{\alpha}$ for $\alpha \geq 1$ which includes the KL divergence and chi-squared divergence as special cases. Recall that in the comment/paragraph below Corollary \ref{cor:lb_01_f} (see also Section \ref{sec:compare} in the appendix), we provided motivation for restricting our attention to such divergences as opposed to e.g., Hellinger distance.

We assume that there is a measure $\mu$ on $\samp$ that dominates $P_{\theta}$ for every $\theta \in \Theta$. None of our results depend on the choice of the dominating measure $\mu$.

When the $f$-informativity is the mutual information,~\cite{haussler1997} have proved useful upper bounds which we briefly review here. Let $P$ and $\{Q_{\xi}, \xi \in \Xi\}$ be probability measures on $\samp$
having densities $p$ and $\{q_{\xi}, \xi \in \Xi\}$ respectively with respect to $\mu$. Let $\nu$ be an arbitrary probability measure on $\Xi$ and $\bar{Q}$ be  the probability measure on $\samp$ having density $\bar{q}=\int_{\Xi} q_{\xi} \nu(\mathrm{d}\xi)$ with respect to $\mu$. \cite{haussler1997} proved the  following inequality
\begin{eqnarray}
  D\left(P||\bar{Q} \right) \leq - \log\left(\int_{\Xi} \exp\left(-D(P||Q_{\xi}) \right) \nu(\mathrm{d}\xi) \right).
  \label{eq:upper_KL_0}
\end{eqnarray}
Now given a class of probability measures $\{P_{\theta}, \theta \in \Theta\}$, applying the above inequality for each $P_{\theta}$ and integrating the resulting inequalities with respect to a probability measure $w$ on $\Theta$,~\citet[Theorem 2]{haussler1997} obtained the following mutual information upper bound:
\begin{equation}\label{eq:upper_KL_1}
I(w, \Ps) \leq  - \int_{\Theta}  \log \left(\int_{\Xi} \exp\left(-D(P_{\theta}||Q_{\xi}) \right) \nu(\mathrm{d}\xi) \right) w(\mathrm{d} \theta).
\end{equation}

In the special case when $\Xi=\{1,\ldots, M\}$ and $\nu$ is the uniform probability measure on $\Xi$, we have $\bar{Q}=\left(Q_1 + \ldots + Q_M\right)/M$ and inequality  \eqref{eq:upper_KL_0} then becomes
$
    D(P||\bar{Q}) \leq -\log \left(\frac{1}{M} \sum_{j=1}^M \exp
      \left(-D(P||Q_j) \right) \right).
$
Because $\sum_{j=1}^M \exp(-D(P\|Q_j)) \geq \exp \left(-\min_{j} D(P\|Q_j) \right)$, we obtain
\begin{equation*}
  D(P \| \bar{Q}) \leq \log M + \min_{1 \leq j \leq M} D(P\|Q_j).
\end{equation*}
Inequality~\eqref{eq:upper_KL_1} can be further simplified to
  \begin{equation}\label{ybp}
   I(w, \Ps) \leq \log M +
    \int_{\Theta} \min_{1 \leq j \leq M} D(P_{\theta}||Q_j)
    w(\mathrm{d}\theta).
  \end{equation}
This inequality can be used to give an upper bound for $f$-informativity in terms of the KL covering numbers. Recall the definition of $M_{KL}(\epsilon, \Theta)$ from~\eqref{redf}. Applying~\eqref{ybp} to any fixed $\epsilon>0$  and
choosing $\{Q_1, \dots, Q_M\}$ to be an $\epsilon^2$-covering, we have
  \begin{equation}\label{yb}
   I(w, \Ps)\leq
    \inf_{\epsilon > 0} \left(\log M_{KL}(\epsilon, \Theta) + \epsilon^2
    \right).
  \end{equation}
  When $w$ is the uniform prior on a finite
  subset of $\Theta$, the above inequality has been proved by~\citet[Page 1571]{Yang:Barron:99}. If $M_{KL}(\epsilon, \Theta)$ is infinity for all $\epsilon$, then \eqref{yb} gives $\infty$ as the upper bound on $I(w, \Ps)$ and thus \eqref{eq:I_bayes_KL} will lead to a trivial lower bound $0$ for $R_{\rm Bayes}$.
   In such a case, one may find a subset $\tilde{\Theta} \subset \Theta$ for which $M_{KL}(\epsilon, \tilde{\Theta})$ is bounded and contains most prior mass. If $\tilde{w}$ denotes the prior $w$ restricted in $\tilde{\Theta}$, then it is easy to see that $R_{\rm Bayes}(w, L; \Theta) \geq  w(\tilde{\Theta}) R_{\rm Bayes}(\tilde{w}, L; \tilde{\Theta})$. Then we can use~\eqref{eq:I_bayes_KL} and~\eqref{yb} to lower bound $R_{\rm Bayes}(\tilde{w}, L; \tilde{\Theta})$ .

In the next theorem, we extend inequalities \eqref{eq:upper_KL_0} and \eqref{eq:upper_KL_1} to power divergences corresponding to $f_{\alpha}$ for $\alpha  \notin [0, 1]$. We also note that in Appendix \ref{effb}, we demonstrate the tightness of  the bound~\eqref{co1} in Theorem \ref{cd} by a simple example.

\begin{theorem}\label{cd}
Fix $\alpha \notin [0, 1]$ and let $f_{\alpha} \in \C$ be as defined in Section~\ref{sec:pre}. Under the setting of inequalities~\eqref{eq:upper_KL_0} and~\eqref{eq:upper_KL_1}, we have
  \begin{equation}\label{co1}
    D_{f_{\alpha}}(P||\bar{Q}) \leq  \left[\int_{\Xi} \left( D_{f_{\alpha}} (P||Q_{\xi}) +1 \right)^{1/(1-\alpha)} \nu(\mathrm{d}\xi) \right]^{1-\alpha} - 1.
  \end{equation}
and
  \begin{equation}\label{eq:upper_f_power}
    I_{f_\alpha}(w, \mathcal{P}) \leq \int_{\Theta} \left[\int_{\Xi} \left( D_{f_{\alpha}} (P_\theta||Q_{\xi}) +1 \right)^{1/(1-\alpha)} \nu(\mathrm{d}\xi) \right]^{1-\alpha} w(\mathrm{d}\theta) - 1.
  \end{equation}
\end{theorem}

To prove Theorem \ref{cd}, the following lemma is critical (the proof of this lemma in given in Appendix \ref{sec:proof_concave}).

\begin{lemma}\label{lem:concavity}
Fix $r < 1$. Let $\mu$ be a probability measure on the space $T$ and let $S := \{u: T \rightarrow \R_+: u \in L_{\mu}^r(T)\}$. Then the map $f : S \rightarrow \R$ defined by $f(u) := \left( \int_{T} u(t)^r  \mu(\mathrm{d}t)\right)^{1/r}$ is concave in $u$.
\end{lemma}
Note that the discrete version of Lemma \ref{lem:concavity}  states that $f(u)=\left(  \sum_{i=1}^M u_i^r/M\right)^{1/r}$ is a concave function of $u \in \mathbb{R}_{+}^M$ when $r<1$.

In fact, since we will apply this lemma to prove Theorem \ref{cd} with $r=\frac{1}{1-\alpha}$,  the condition $r<1$ in Lemma \ref{lem:concavity} translates into $\alpha \not \in [0,1]$ in Theorem \ref{cd}.  We are now ready to prove Theorem~\ref{cd}.

\begin{proof}[Proof of Theorem \ref{cd}]

  By the identity that $D_{f_\alpha}(P||Q)=D_{f_{1-\alpha}}(Q||P)$, we have
  \begin{align*}
    D_{f_{\alpha}}(P||\bar{Q})  = D_{f_{1-\alpha}}(\bar{Q}||P)  &=\int_{\mathcal{X}} p \left( \int_{\Xi}\frac{q_{\xi}}{p} \nu(\mathrm{d}\xi) \mathrm{d} \mu \right)^{1-\alpha}    -1 \\
&=\int_{\mathcal{X}} p \left( \int_{\Xi}\left[\left(\frac{q_{\xi}}{p} \right)^{1-\alpha} \right]^{1/(1-\alpha)}\nu(\mathrm{d}\xi) \mathrm{d} \mu \right)^{1-\alpha} -1
  \end{align*}
  Let $u(\xi, x) = \left(\frac{q_{\xi}}{p} \right)^{1-\alpha}$. Since $\frac{1}{1-\alpha}<1$ when $\alpha \not \in [0,1]$, Lemma \ref{lem:concavity} implies that
  $u(\xi, x) \mapsto \left( \int_{\Xi} u(\xi, x)^{1/(1-\alpha)} \nu(\mathrm{d}\xi)\right)^{1-\alpha} $ is concave in $u$. Applying Jensen's inequality,
  \begin{align*}
     D_{f_{\alpha}}(P||\bar{Q})   & \leq  \left( \int_{\Xi}\left[\int_{\mathcal{X}} p \left(\frac{q_{\xi}}{p} \right)^{1-\alpha} \mathrm{d} \mu \right]^{1/(1-\alpha)}\nu(\mathrm{d}\xi)  \right)^{1-\alpha} -1 \\
     & =  \left( \int_{\Xi}\left[D_{f_{1-\alpha}}(Q_{\xi} || P) \right]^{1/(1-\alpha)}\nu(\mathrm{d}\xi)  \right)^{1-\alpha} -1.
     %& =  \left( \int_{\Theta}\left[D_{f_{\alpha}}(P || Q_{\xi}) \right]^{1/(1-\alpha)}\nu(\mathrm{d}\xi)  \right)^{1-\alpha} -1
  \end{align*}
This completes the proof of~\eqref{co1} because $D_{f_{1-\alpha}}(Q_{\xi} || P)= D_{f_{\alpha}}(P || Q_{\xi}) $. The proof of~\eqref{eq:upper_f_power} follows by applying~\eqref{co1} for $P = P_{\theta}$ and then integrating the resulting bound with respect to $w(\mathrm{d}\theta)$.
\end{proof}

%%We are now ready to prove Theorem~\ref{cd}.

%As seen from the proof of Theorem \ref{cd}, the condition $r<1$ in Lemma \ref{lem:concavity} translates into $\alpha \not \in [0,1]$ in Theorem \ref{cd}. %It is an interesting future work to derive such upper bounds on $f$-informativity for $\alpha \in [0,1]$. \anote{Should we keep this statement given that we said in Section 3 that $f$-informativity is not very useful for $0 < \alpha < 1$? Also the case $\alpha = 1$ is already done by Haussler and Opper.}
%
%In Subsection \ref{effb} of the supplementary material, we demonstrate the tightness of  the bound~\eqref{co1} by a simple example in Section~\ref{effb}.
%In fact, since we will apply this lemma to prove Theorem \ref{cd} with $p=\frac{1}{1-\alpha}$,  the condition $p<1$ in Lemma \ref{lem:concavity} translates into $\alpha \not \in [0,1]$ in Theorem \ref{cd}.% It is an interesting future work to derive upper bound on $f$-informativity for $\alpha \in [0,1]$.
%Nevertheless, the case that $\alpha \not \in \in [0,1]$ that includes chi-squared divergence is sufficient for our applications.
%
%The proofs of Lemma \ref{lem:concavity} and Theorem \ref{cd} are provided in the supplementary material~\citep{XiAdityaYuchenSupp}. We also demonstrate the tightness of the bound~\eqref{co1} by means of a simple example in Section~\ref{effb} in the supplementary material~\citep{XiAdityaYuchenSupp}.

%We are now ready to prove Theorem~\ref{cd}.

For $\alpha > 1$, one can deduce an upper bound analogous to~\eqref{yb} for the $f_{\alpha}$-informativity which is described in the next corollary. Recall the notion of the covering numbers $M_{\alpha}(\epsilon, \Theta)$ from Section \ref{sec:pre}.

%; recall the notation $M_{\alpha}(\epsilon, S)$ from~\eqref{redf}.
\begin{corollary}\label{cor:chi_dist_upper}
  For every $\alpha > 1$, we have
  \begin{equation}\label{eq:chi_dist_upper}
    I_{f_\alpha}(w, \mathcal{P})  \leq \inf_{\epsilon> 0} (1 + \epsilon^2) M_{\alpha}(\epsilon, \Theta)^{\alpha - 1} - 1.
  \end{equation}
\end{corollary}
In particular, when $D_{f_{\alpha}}$ is the chi-square divergence, Corollary~\ref{cor:chi_dist_upper} implies
\begin{equation}\label{eq:chi_dist_upper_1}
I_{\chi^2}(w, \mathcal{P}) \leq \inf_{\epsilon> 0} (1 + \epsilon^2) M_{\chi^2}(\epsilon, \Theta) - 1.
\end{equation}
Note that Corollary~\ref{cor:chi_dist_upper} gives trivial bound when $M_{\alpha}(\epsilon, \Theta)$ equals $\infty$ for all $\epsilon > 0$. This can be handled in a way similar to that outlined in the discussion after \eqref{yb}.

\begin{proof}[Proof of Corollary \ref{cor:chi_dist_upper}]

Let $Q_1, \dots, Q_M$ be probability measures on $\samp$ and fix $\theta \in \Theta$. Inequality~\eqref{co1} applied to $P = P_{\theta}$, $\Xi := \{1, \dots, M\}$ and the uniform probability measure on $\Xi$ as $\nu$ gives
\begin{equation*}
  D_{f_{\alpha}}(P_{\theta} \| \bar{Q}) \leq  M^{\alpha-1} \left[ \sum_{j=1}^M  (1+D_{f_{\alpha}}(P_{\theta}\| Q_j))^{1/(1-\alpha)} \right]^{1-\alpha} - 1
\end{equation*}
We now use (note that $\alpha > 1$)
\begin{eqnarray*}
  \sum_{j=1}^M  (1+D_{f_{\alpha}}(P_{\theta}\| Q_j))^{1/(1-\alpha)}  \geq  & \max_{1 \leq j \leq M}  (1+D_{f_{\alpha}}(P_{\theta}\| Q_j))^{1/(1-\alpha)} \\
                          = & \left(1 + \min_{1 \leq j \leq M} D_{f_{\alpha}}(P_{\theta}\|Q_j) \right)^{1/(1-\alpha)}.
\end{eqnarray*}
This gives $$D_{f_{\alpha}}(P_{\theta}\|\bar{Q}) \leq M^{\alpha - 1} \left(1 + \min_{1 \leq j \leq M} D_{f_{\alpha}}(P_{\theta}\|Q_j) \right) - 1.$$

We now fix $\epsilon > 0$ and apply the above with $\{Q_1, \dots, Q_M\}$ taken to be an $\epsilon^2$-cover of $\Theta$ under the $f_\alpha$-divergence. We then obtain
\begin{equation*}
  D_{f_{\alpha}}(P_{\theta}\|\bar{Q}) \leq \inf_{\epsilon>0} (1 + \epsilon^2) M_{\alpha}(\epsilon, \Theta)^{\alpha - 1} - 1.
\end{equation*}
The proof is complete by integrating the above inequality with respect to $w(\mathrm{d}\theta)$.

\end{proof}

We now turn to applications of the Bayes risk lower bounds in Corollary \ref{cor:bwl} and the informativity upper bounds in this section. We present a toy example here and postpone more complicated examples (e.g., generalized linear model, spiked covariance model, Gaussian model with general prior and loss) to Appendix \ref{sec:Bayes_Example}.

\begin{example}[Gaussian model with uniform priors on large balls]\label{tutu}
Fix $d \geq 1$. Suppose $\Theta = \ac \subseteq \R^d$ and let $L(\theta, a) := \|\theta - a\|_2^2$. For each $\theta \in \R^d$, let $P_{\theta}$ denote the Gaussian distribution with mean $\theta$ and covariance matrix $\sigma^2 I_{d\times d}$  ($\sigma^2 > 0$ is a constant). Let $w$ be the uniform distribution on the closed ball of  radius $\Gamma$ centered at the origin. Let $\Gamma \geq \sigma \sqrt{d}$. We will show below how to obtain  the tight Bayes risk lower bound using Corollary \ref{cor:bwl} along with the $f$-informativity upper bound in Corollary~\ref{cor:chi_dist_upper}.

%As discussed in Example~\ref{gmc}, the bound given by~\eqref{hok} is not always tight in this setting. We show below how to derive the optimal Bayes risk lower bound on using chi-squared divergence.

We can assume that $\Theta$ (and $\ac$) is the closed ball of radius $\Gamma$ centered at the origin as $w$ puts zero probability outside this ball. We use the inequality~\eqref{eq:I_bayes_chi} induced by the chi-squared divergence. To establish the lower bound, we need to upper bound $\sup_{a \in \ac} w(B_t(a, L))$ and the chi-squared informativity. The former can be easily controlled because $\sup_{a \in \ac} w(B_t(a, L)) \leq \left(\sqrt{t}/\Gamma \right)^d.$ For the latter, we use \eqref{eq:chi_dist_upper_1}, which requires an upper bound on $M_{\chi^2}(\epsilon, \Theta)$. Note that $\chi^2(P_{\theta} \| P_{\theta'}) = \exp \left(\|\theta-\theta'\|_2/\sigma^2 \right) - 1$ for $\theta, \theta' \in \Theta$. As a consequence, $\chi^2(P_{\theta} \| P_{\theta'}) \leq \epsilon^2$ if and only if $\|\theta - \theta\|_2 \leq \epsilon' := \sigma\sqrt{\log (1 + \epsilon^2)}$. Therefore, by a standard volumetric argument, we have
  \begin{eqnarray*}
    M_{\chi^2}(\epsilon, \Theta) \leq \left(\frac{\Gamma+\epsilon'/2}{\epsilon'/2}\right)^d \leq \left(\frac{3\Gamma}{\epsilon'}\right)^d = \left( \frac{3 \Gamma }{ \sigma \sqrt{\log(1+\epsilon^2)}}\right)^d
  \end{eqnarray*}
provided $\epsilon' \leq \Gamma$. In particular, if we take $\epsilon := \sqrt{e^d - 1}$, then $\epsilon' = \sigma \sqrt{d} \leq \Gamma$, we will obtain $M_{\chi^2}(\epsilon, \Theta) \leq (3 \Gamma/(\sigma \sqrt{d}))^d$. Inequality~\eqref{eq:chi_dist_upper_1} then gives
$
  I_{\chi^2}(w, \Ps) \leq \left( \frac{3 e \Gamma }{ \sigma \sqrt{d}}\right)^d-1.
$
Let $\Iu$ be the right hand side. If we choose $t = c d \sigma^2$ for a sufficiently small constant $c > 0$, then we have $ \sup_{a \in \ac} w(B_t(a, L)) < \frac{1}{4} (1 + \Iu)^{-1}$. Inequality~\eqref{eq:I_bayes_chi} then gives
\begin{equation}\label{kolo}
R_{\rm Bayes}(w, L; \Theta) \geq c d \sigma^2.
\end{equation}
This lower bound is tight due to the trivial upper bound $R_{\rm Bayes}(w, L; \Theta) \leq d \min(\sigma^2, \Gamma^2)$ since $R_{\rm Bayes}(w, L; \Theta)$  is smaller than the risk of the constant estimator 0 as well as the trivial estimator of the observation itself.

This example allows us to compare the bound given by Theorem \ref{bwl} for different $f \in \C$.  We argue below that using KL divergence and applying \eqref{eq:I_bayes_KL}  along with inequality~\eqref{yb} for controlling the mutual information will not yield a tight lower bound for this example.
%We argue below that it is not possible to derive~\eqref{kolo} by applying~\eqref{eq:main_bwl} with $f(x) = x\log x$, i.e.,~\eqref{eq:I_bayes_KL} along with inequality~\eqref{yb} for controlling the informativity $I_f(w, \Ps)$.
In other words, the same strategy that works for $f(x) = x^2 - 1$ does not work for $f(x) = x \log x$. To see this, notice that $D(P_{\theta}\|P_{\theta'}) = \|\theta - \theta'\|^2/\sigma^2$ for $\theta, \theta' \in \Theta$. As a result, $D(P_{\theta} \| P_{\theta'}) \leq \epsilon^2$ if and only if $\|\theta - \theta'\| \leq \sqrt{2} \epsilon \sigma$. The same volumetric argument again gives
$
  M_{KL}(\epsilon, \Theta) \leq \left(\frac{3 \Gamma}{\sqrt{2} \epsilon \sigma} \right)^d
$ provided  $\sqrt{2} \epsilon \sigma \leq \Gamma$.
The bound~\eqref{yb} implies that the mutual information $I(w, \Ps)$ is bounded by
\begin{equation*}
  I(w, \Ps) \leq \inf_{0 < \epsilon \leq \Gamma/(\sqrt{2} \epsilon \sigma)} \left(d \log \left(\frac{3  \Gamma}{\sqrt{2} \epsilon \sigma}  \right) + \epsilon^2 \right) = d \log \left(\frac{3 \Gamma}{\sigma \sqrt{d}} \right) + \frac{d}{2}.
\end{equation*}
Let $\Iu$ be the right hand side above. The maximum $t > 0$ for which $(\sqrt{t}/\Gamma)^d < \frac{1}{4} \exp \left(-2\Iu \right)$ is on the order of $d^2 \sigma^4/\Gamma^2$. This means that inequality~\eqref{eq:I_bayes_KL} implies a weaker lower bound $\Omega(d^2 \sigma^4/\Gamma^2)$, which is suboptimal when $d \sigma^2$ is small or when $\Gamma$ is large. This is in contrast with the optimal bound~\eqref{kolo}.
\end{example}
In the above example, a direct application of Theorem~\ref{bwl} with $f(x) = x \log x$ does not produce a tight lower bound. This is mainly because, when the prior is over a large parameter space (e.g., a ball of a constant radius), the upper bound of mutual information over the entire parameter space $\Theta$ in \eqref{yb} could be too loose.
This can be corrected by partitioning the parameter space $\Theta$ into small hypercubes, and applying our bounds for the prior restricted to each hypercube separately  so that the mutual information inside the partition can be appropriately upper bounded using \eqref{yb}. This is another illustration of the idea described in Remark \ref{rem:Bayes_decomp}. We first describe this method in a more general setting in the following corollary and then apply it to the setting of Example \ref{tutu}. We use the  following notation. For measurable subsets $S$ of a Euclidean space, $\text{Vol}(S)$ denotes the volume (Lebesgue measure) of $S$.

\begin{corollary}\label{nd}
Let $\Theta=\mathcal{A} \subseteq \mathbb{R}^d$. Suppose that the prior $w$ has a Lebesgue density $f_w$ that is positive over $\Theta$. For each $\theta \in \Theta$ and $\delta>0$, let
\begin{equation*}
  r_{\delta}(\theta) := \sup \left\{ \frac{f_w(\theta_1)}{f_w(\theta_2)} : \theta_i \in \Theta \text{ and } \|\theta_i - \theta\|_2 \leq \sqrt{d} \delta \text{ for } i = 1, 2 \right\}.
\end{equation*}
Suppose also the existence of $A > 0$ such that $D(P_{\theta_1} \| P_{\theta_2}) \leq A  \|\theta_1 - \theta_2\|_2^2$  for all $\theta_1, \theta_2 \in\Theta$ and the existence of $V > 0$ (which may depend on $d$) and $p > 0$ such that $\sup_{a \in \ac} \mathrm{Vol} (B_t(a, L)) \leq V t^{d/p}$ for every $t > 0$. Then
     \begin{equation}
           R_{\rm Bayes}(w, L; \Theta)  \geq \frac{1}{2} \sup_{0 < \delta \leq A^{-1/2}} \left[e^{-2p}  \delta^p  (8V)^{-p/d} \int_{\Theta} \left(\frac{1}{r_{\delta}(\theta)}\right)^{p/d}  w(\mathrm{d}\theta) \right].
           \label{eq:R_bayes_fano}
     \end{equation}
     \label{cor:KL_bayes_lower}
\end{corollary}

The proof of Corollary is quite technically involved and thus is deferred to Appendix \ref{sec:cor_nd}.

We demonstrate below that this corollary yields the correct rate in Example~\ref{tutu}. More examples (e.g., estimation problem in generalized linear model, spiked covariance model, and Gaussian model with a general loss) are given in Appendix \ref{sec:Bayes_Example}.

\begin{example}[Gaussian model with uniform priors on large balls (continued)] \label{tutug}
  Consider the same setting as in Example~\ref{tutu}. Because $D(P_{\theta}\|P_{\theta'}) = \|\theta - \theta'\|_2^2/(2 \sigma^2)$, we can take $A = (2 \sigma^2)^{-1}$ in Corollary~\ref{nd}. Moreover, because $L(\theta, a) = \|\theta - a\|_2^2$, it is easy to see that $\sup_{a \in \ac} \mathrm{Vol}(B_t(a, L)) \leq t^{d/2} \mathrm{Vol} (B)$ which means that we can take $p = 2$ and $V = \mathrm{Vol}(B)$ in Corollary~\ref{nd} where $B$ is the unit ball in $\R^d$. Finally, because $w$ is the uniform prior, we have  $r_{\delta}(\theta) = 1$ for all $\theta \in \Theta$. Corollary~\ref{nd} therefore gives
  \begin{equation*}
    R_{\rm Bayes}(w, L; \Theta) \geq \frac{1}{2} \sup_{0 < \delta \leq \sqrt{2} \sigma} \left(e^{-4} 8^{-2/d} \delta^2 \mathrm{Vol}(B)^{-2/d} \right).
  \end{equation*}
This matches the tight lower bound~\eqref{kolo} by noting that $\mathrm{Vol}(B)^{1/d} \asymp d^{-1/2}$.
\end{example}

\section{Smoothed Analysis for Spherical Gaussian Mixture Models with Uniform Weights}
\label{sec:smoothed-analysis}

Smoothed analysis is a useful technique for analyzing algorithms that fail in the worst case but succeed with high probability in the average case. For parameter estimation problems, smoothed analysis assumes that the  parameter to be estimated is randomly perturbed by a small noise, and the data is generated with respect to the perturbed parameter as well. Under this setting, if the set of ``bad'' parameters that fail the estimator has zero measure, then the estimator will succeed almost surely after the perturbation. Smoothed analysis has been successfully applied to analyze linear programming~\citep{blum2002smoothed,dunagan2002smoothed,hsu2013learning,spielman2003smoothed}, integer programming~\citep{roglin2007smoothed}, binary search trees~\citep{manthey2007smoothed}, and other combinatorial problems~\citep{banderier2003smoothed}. See the paper by~\citet{spielman2003smoothed} for a survey of existing works.

In this section, we use smoothed analysis to study an important problem in statistical estimation: learning mixture of spherical Gaussians. The problem of computing the maximum log-likelihood estimator is NP-hard~\citep{arora2005learning}. However, if the true parameters are perturbed by a random noise, then we demonstrate that a variant of the polynomial-time algorithm proposed by~\citet{hsu2013learning} succeeds in estimating the Gaussian means. We present an upper bound on the algorithm's mean-squared error using smoothed analysis, which achieves a better rate than the original algorithm of~\citet{hsu2013learning}.
Furthermore, we apply the Bayes risk lower bound developed in this paper to show that, the mean squared-error achieved by this algorithm is unimprovable, even under smoothed analysis. To the best of our knowledge, the lower bound cannot be established by traditional information-theoretic techniques for lower bounding minimax risks.

\subsection{Learning Mixture of Gaussians}
\label{sec:learn-gmm}

We study estimating the parameter of a Gaussian mixture model (GMM). The parameter of a GMM
is a $d$-by-$k$ matrix $\theta \defeq (\theta_1,\dots,\theta_k)$. Each $\theta_i\in \R^d$
represents the mean of the $i$-th mixture component. We assume that the number of components $k$ is much less than the dimensionality $d$.  Suppose that $n$ i.i.d.~instances $\{x_i\}_{i=1}^n$  are sampled from the GMM with each $x_i\in \R^d$.
Equivalently, it is generated by the following procedure: First, an integer $z_i$ is uniformly sampled from $\{1,\dots,k\}$. This integer is called the \emph{membership} of the $i$-th instance\footnote{For simplicity, we focus on the case when all mixture components have equal weights, but our argument can be easily generalized to the case of non-uniform weights.}. Then, the vector $x_i$
is drawn from the spherical Gaussian distribution $\normal(\theta_{z_i}; I_{d\times d})$. The goal is to estimate the parameters $\theta$.

Information theoretically, the GMM model is learnable if the Gaussian means are well seperated. Let $D$ represent the minimum distance between two distinct component means. \citet{vempala2004spectral} show that, as long as $D > C$ for $C$ being a sufficiently large constant, the estimation error on $\theta$ scales as $\order(n^{-1/2})$. However, the algorithm achieving this rate has $\order(k^k)$ time complexity. When the mutual distance $D$ is large enough, there are ${\rm poly}(n,d,k)$-time algorithms to estimate the model parameters. In particular, \citet{dasgupta1999learning} presents an algorithm for $D=\Omega(\sqrt{d})$. \citet{arora2005learning} and \citet{dasgupta2000two} present algorithms for $D=\Omega(d^{1/4})$. \citet{vempala2004spectral} reduce this distance lower bound to $\Omega(k^{1/4})$. However, designing ${\rm poly}(n,d,k)$-time algorithm for $\widetilde\Omega(1)$-separated GMMs is a long-standing open problem.

\citet{hsu2013learning} proposed a method that does not need the well-separation condition. The only assumption is that $\{\theta_1,\dots,\theta_k\}$ are linearly independent. Let $\sigma_{\min} > 0$ be the smallest singular value of the matrix $\theta$. Their algorithm runs in ${\rm poly}(n,d,k)$-time and achieves the following bound for estimator $\widehat \theta$:
\begin{align}\label{eqn:spectral-method-bound}
	\norms{\widehat \theta - \theta}_F^2 = \order\left( \frac{{\rm poly}(d,k,1/{\sigma_{\min}}) \log(1/\delta)}{n} \right)\quad\mbox{with probability at least $1 - \delta$}.
\end{align}
Here, $\|\cdot\|_F$ denotes the matrix Frobenius norm.
In general, we cannot guarantee that $\sigma_{\min} > 0$. However, if we add a small perturbation on the true component means, then the assumption is satisfied almost surely.
More precisely, we assume that there is a matrix $\theta^*\in \R^{d\times k}$ so that each entry of matrix $\theta$ is sampled from $\theta_{ij} \sim \normal(\theta^*_{ij}; \rho^2)$. The following lemma lower bounds the smallest singular value.
\begin{lemma}[\citet{ge2015learning}, Lemma G.16]\label{lemma:matrix-perturb}
Let $\theta^*\in \R^{d\times k}$ and suppose that $d\geq 3k$. If all entries of $\theta^*$ are independently perturbed by $N(0, \rho^2)$ to yield matrix $\theta$. For any $\epsilon > 0$, with probability at least $1 - c_1(c_2 \epsilon)^d$, the smallest singular value of matrix $\theta$ is lower bounded by:
\[
	\sigma_{\min} > \epsilon \rho \sqrt{d}.
\]
Here, $c_1,c_2$ are universal constants.
\end{lemma}
\noindent We choose $\epsilon, \rho \sim n^{-c}$ for a sufficiently small $c > 0$, then the perturbation diminishes to zero, and if $\sigma_{\min} > \epsilon \rho \sqrt{d}$ holds, then
the right-hand side of equation~\eqref{eqn:spectral-method-bound} converges to zero at a polynomial rate as $n\to\infty$. Lemma~\ref{lemma:matrix-perturb} implies that the probability of this event is at least $1 - \order(n^{-cd})$. Thus,
with high probability, the estimator $\widehat \theta$ is consistent  under the smoothed analysis.

The convergence rate of the estimator $\widehat \theta$ can be improved if we add a mild assumption that $D = \widetilde\order(\sqrt{\log (n k)})$. Although the main focus of the paper is on lower bounds, the upper bound result on the estimation of $\widehat \theta$ in learning mixture of Gaussians is of its independent interest. To obtain the upper bound on $\E[\norms{\thetahat - \theta}_{F}^2]$, we first establish the following lemma:
\begin{lemma}\label{lemma:distance}
Let the mutual distance satisfy $D \geq c_1\sqrt{\log(n k/\delta)}\geq  3$ for a sufficiently large constant $c_1$. With probability at least $1-\delta$, the inequality $\ltwos{x_i - \theta_{j}} - \ltwos{x_i - \theta_{z_i}} \geq c_2 (d \log(n k/\delta))^{-1/2}$ holds for a  constant $c_2 > 0$, for any $i\in [n]$ and any $j\in [k]\backslash\{z_i\}$.
\end{lemma}

The proof of this technical lemma is relegated to Appendix  \ref{sec:distance}. Lemma~\ref{lemma:distance} shows that with high probability, the distance of a random sample to its true component mean is significantly less than the distance to any other means. Let $\thetahat_j$ represent the $j$-th column of $\widehat \theta$. When the sample size $n$ is sufficiently large, the method of \citet{hsu2013learning} guarantees that $\ltwos{\thetahat_j - \theta_j} < o((d \log(n k/\delta))^{-1/2})$ for any $j\in[k]$. Thus, Lemma~\ref{lemma:distance} implies that the distance of $x_i$ to $\thetahat_{z_i}$ is smaller than the distance to any other estimated centers. As a consequence, we may recover the membership of instances by computing the center that is the closest to them.
\[
	\zhat_i = \arg\min_{j\in [k]} \ltwos{x_i - \thetahat_j}.
\]
According to Lemma~\ref{lemma:distance}, with high probability we have $\zhat_i = z_i$ for any $i\in[n]$. Given the membership, we refine the mean estimates by:
\[
	\thetahat_j \leftarrow \frac{\sum_{i:\zhat_i=j} x_i}{|\{i:\zhat_i=j\}|}.
\]
Since the membership is uniformly assigned, with high probability the sample size of the $j$-th Gaussian component is lower bounded by $\frac{n}{2k}$. Thus, with high probability the squared error of $\thetahat_j$ will be upper bounded by $\order(d k/n)$. Since there are $k$ components, the overall squared error is bounded by $\order(d k^2/n)$. Putting pieces together, we have an upper bound on the mean-squared error of parameter estimation.

\begin{proposition}\label{prop:gmm-upper}
Suppose that $d \geq 3k$ and $n$ is greater than a fixed polynomial function of $(d,k,1/\rho)$. Let the true parameter $\theta$ be $\rho$-perturbed from an arbitrary matrix $\thetastar\in \R^{d\times k}$.
In addition, assume that the distances between the columns of $\thetastar$ are at least $D = c\sqrt{\log(nk)}$ for some universal constant $c$. Then there is a universal constant $C$ such that the estimator $\thetahat$ described above achieves mean-square error:
\[
	\E[\norms{\thetahat - \theta}_{F}^2] \leq \frac{C dk^2}{n}.
\]

\end{proposition}

%%%%%%%%%%%%%%%%%%%%%%%%%%%%%%%%%%%%%%%%%%%%%%%%%%%%%%%%%%%%%%%%%%%%%%%%%%%%%%%%%%%%%%%%%%%%%%%%%%%%%%%

\subsection{Minimax Risk of Smoothed Analysis}

In this section, we formalize the notion of minimax risk under smoothed analysis. Similar to the classical statistical setting, the \emph{minimax risk under smoothed analysis} can be defined in a game theoretic way. The learner first chooses an estimator $\thetahat$, then the adversary chooses a parameter $\theta^*$ from the parameter space $\Theta$, which is randomly perturbed to form the true parameter $\theta$. The data $X$ is generated with respect to $\theta$. Under this random perturbation framework, the minimax risk is defined as:
\begin{align}\label{eqn:smooth-minimax-risk}
		R_{\rm minimax} \defeq \inf_{\widehat \theta} \sup_{\theta^*\in \Theta}\; \E_{\theta}[L(\widehat \theta(X), \theta)]
\end{align}
where $L(\cdot,\cdot)$ is the loss function. In our GMM application, the parameters are the means of mixture components. The parameter space is the set of means whose mutual distances are lower bounded by $D$. The true parameter is generated by a random Gaussian perturbation with variance $\rho^2$. The loss is the Frobenius norm of the difference of matrices.

We note that the minimax risk~\eqref{eqn:smooth-minimax-risk} differs from the classical notion of minimax risk in that the adversary is not able to explicitly choose the true parameter $\theta$. Instead, the true parameter is sampled from a prior distribution parametrized by $\theta^*$. This Bayes nature makes it hard to lower bound the minimax risk~\eqref{eqn:smooth-minimax-risk} using the traditional Le Cam's or the Fano's method. In particular, both the Le Cam's method and the Fano's method lower bound the minimax risk by assuming a uniform prior over a carefully constructed discrete set. However, in our GMM setting, the prior distribution of parameter $\theta$ is always continuous.

Our Bayes risk lower bound naturally fits into the setting of smoothed analysis. Let $w^*$ be an arbitrary prior distribution over $\theta^*$. Since $\theta$ is perturbed from $\theta^*$, the prior $w^*$ induces a prior $w$ over $\theta$. It is easy to see that the Bayes risk with respect to $w$ is a lower bound on the minimax risk~\eqref{eqn:smooth-minimax-risk}. Thus, it suffices to lower bound the Bayes risk:
\[
	R_{\rm Bayes}(w, L;\Theta) \defeq \inf_{\widehat \theta} \E_{\theta\sim w}[L(\widehat \theta(X), \theta)].
\]
For the GMM example, we construct the prior distribution $w^*$ as follow: the $j$-th column of $\thetastar$, namely the vector $\theta_j^*\in \R^d$, is sampled from the normal distribution $N(D e_j; I_{d\times d})$, where $e_j$ is the unit vector of the $j$-th coordinate. As a consequence, the prior distribution $w$ samples the $j$-th column of $\theta$ from the normal distribution $N(D e_j; (1+\rho^2)I_{d\times d})$.

In the GMM setting, the membership variables $z_i$ are unknown to the estimator. If we assume that the memberships are given to the estimator, it makes the problem easier so that the associated Bayes risk is a smaller than or equal to the original Bayes risk. Since we want to derive a lower bound, we make the assumption that the memberships are given, then partition the instances into $k$ disjoint subsets according to their memberships. Let the $j$-th subset $S_j$ be defined as $S_j \defeq \{x_i: z_i = j\}$. Conditioning on the memberships, the distributions of $\{(\theta_j, S_j)\}_{j=1}^k$ are mutually independent. Thus, we have
\begin{align}\label{eqn:smooth-bayes-risk}
	R_{\rm Bayes}(w, L;\Theta) \geq \sum_{j=1}^k \inf_{\widehat \theta_j}\; \E_{\theta_j\sim w_j}[L(\thetahat_j(S_j), \theta_j)]\geq
	\sum_{j=1}^k \E\Big[\inf_{\widehat \theta_j}\; \E_{\theta_j\sim w_j}[L(\thetahat_j(S_j), \theta_j)|n_j] \Big]
\end{align}
where $w_j$ is the prior distribution $N(D e_j; (1+\rho^2)I_{d\times d})$ and $n_j$ is the cardinality of $S_j$. We focus on the inner term on the right-hand side, namely $\inf_{\widehat \theta_j} \E_{\theta_j\sim w_j}[L(\thetahat_j(S_j), \theta_j)|n_j]$, and find that it is the Bayes risk of Gaussian mean estimation with $n_j$ i.i.d.~samples, with the true parameter $\theta_j$ satisfying a Gaussian prior $w_j$. This Bayes risk can be easily lower bounded by the techniques that we develop in this paper.

\begin{lemma}\label{lemma:gaussian-prior}
Suppose that the standard deviation of normal perturbation $\rho \leq 1$ and $n_j\geq 1$. For a universal constant $c$, the Bayes risk is lower bounded by
\[
	R_{\rm Bayes}(w_j, n_j) \defeq \inf_{\widehat \theta_j} \E_{\theta_j\sim w_j}[L(\thetahat_j(S_j), \theta_j)|n_j] \geq \frac{c d}{n_j}.
\]
\end{lemma}

\begin{proof}[Proof of Lemma \ref{lemma:gaussian-prior}]

We denote the distribution of instances in $S_j$ by $P_{\theta_j}$ and let $\Ps$ be the set of such distributions. Since the support of $w_j$ is $\R^d$, we start by defining a prior whose support is an Euclidean ball of radius $\Gamma \defeq \sqrt{2d}$. Let $\wbar$ be the truncated prior satisfying:
\[
	\wbar(x) = \left\{\begin{array}{ll}
	w_j(x)/c_1 & \mbox{if $\ltwos{x - De_j} \leq \Gamma$}\\
	0 & \mbox{otherwise}.
	\end{array}\right.
\]
The normalization factor $c_1$ is equal to the total mass of $w$ in the ball $\{x:\ltwos{x - De_j}\leq \Gamma\}$. It is straightforward to verify that the radius $\Gamma$ is sufficiently large so that $c_1$ is lower bounded by a universal constant. The prior $\wbar$ can be viewed as restricting the original prior in a finite radius. According to Remark~\ref{rem:Bayes_decomp}, we may lower bound the Bayes risk by
\[
	R_{\rm Bayes}(w_j, n_j) \geq c_1\cdot R_{\rm Bayes}(\wbar, n_j).
\]
Thus, it suffices to lower bound the second term on the right-hand side.

We follow the similar steps of Example~\ref{tutu} to establish the lower bound.
We start by upper bounding the terms $\sup_{a \in \ac} \wbar(B_t(a, L))$ and the chi-squared informativity $I_{\chi^2}(\wbar, \Ps)$. Using definition of the multivariate normal distribution, it is easy to see that
\[
	\sup_{a \in \ac} \wbar(B_t(a, L)) = \wbar(B_t(De_j, L)) \leq \frac{V({\sqrt{t}})}{c_1(2\pi(1+\rho^2))^{d/2}}
\]
where $V({\sqrt{t}})$ represents the volumn of the Euclidean ball of radius $\sqrt{t}$. Thus, there is a universal constant $c_2$ such that $\sup_{a \in \ac} \wbar(B_t(a, L)) \leq (c_2 \sqrt{t}/\Gamma)^d$.
On the other hand, we follow the same steps of Example~\ref{tutu} to upper bound the chi-square informativity. Note that our setup has $n_j$ i.i.d.~observations, but in Example~\ref{tutu} there is only one observation. In this generalized setup, the chi-square distance $\chi^2(P_{\theta} \| P_{\theta'})$ is equal to $\exp \left(n_j\|\theta-\theta'\|_2^2/\sigma^2 \right) - 1$. Plugging this formula into the argument of Example~\ref{tutu}, we obtain the upper bound $I_{\chi^2}(\wbar, \Ps) \leq (3 e \Gamma \sqrt{n_j/d})^d-1$.

Let $\Iu$ be the obtained informativity upper bound. If we choose $t = c d / n_j$ for a sufficiently small constant $c > 0$, then we have $ \sup_{a \in \ac} \wbar(B_t(a, L)) < \frac{1}{4} (1 + \Iu)^{-1}$. Corollary~\ref{cor:bwl} then gives $R_{\rm Bayes}(\wbar, n_j) \geq c d / n_j$.
\end{proof}

Combining inequality~\eqref{eqn:smooth-bayes-risk} and Lemma~\ref{lemma:gaussian-prior}, we have
\[
	R_{\rm Bayes}(w, L;\Theta) \geq \sum_{j=1}^k \frac{cd k}{2n} \mathbb{P}(n_j \leq 2n/k) .
\]
Recall that every $n_j$ satisfies a binomial distribution $B(n, 1/k)$, which has median $\lfloor n/k \rfloor$ or $\lceil n/k \rceil$, thus the probability $\mathbb{P}(n_j \leq 2n/k)$ will be at least $1/2$. It implies that the Bayes risk is lower bounded by $\Omega(dk^2/n)$. Putting pieces together, we have the following lower bound on the minimax risk.

\begin{proposition}\label{prop:gmm-lower}
Assume that the standard deviation of normal perturbation  $\rho \leq 1$, then for some universal constant $c$ the minimax risk of smoothed analysis is lower bounded by $R_{\rm minimax} \geq c \frac{dk^2}{n}$.
\end{proposition}

Comparing proposition~\ref{prop:gmm-upper} and proposition~\ref{prop:gmm-lower}, we find that both the upper bound and the lower bound are tight. More precisely, under the assumptions of proposition~\ref{prop:gmm-upper}, the minimax risk of smoothed analysis is precisely on the order of $dk^2/n$.

\section{Bayes Risk Lower Bounds for Sparse Linear Regression}
\label{sec:sparse_linear}

Linear regression is a canonical problem in machine learning and statistics. For a fixed design matrix $X\in \R^{n\times d}$ and an unknown parameter $\theta\in \R^d$, the learner observes a noise-corrupted response vector $y = X\theta + \varepsilon$, where $\varepsilon$ satisfies an isotropic normal distribution $\normal(0, \sigma^2 I_{d\times d})$. The goal is to take the response vector as input and find an estimator $\thetahat\in \R^d$ for the true parameter~$\theta$. The risk is measured either by the estimation error $\Lest(\theta,\thetahat) \defeq \ltwos{\thetahat - \theta}^2$, or by the prediction error $\Lpre(\theta,\thetahat)  \defeq \ltwos{X\thetahat - X\theta}^2$. Both errors will be studied in this section.

For high-dimensional linear regression, the dimension $d$ can be much greater than the sample size $n$. In order to prevent over-fitting, one needs to impose structural assumptions on the true parameter, for example, assuming that the the number of non-zero entries in vector $\theta$ is at most $k$ ($k\ll d$). Formally, we use $\ball_0(k)$ to represent the set of $k$-sparse vectors in $\R^d$, and assume that $\theta\in \ball_0(k)$. Under this setting, we want to compute an estimator $\thetahat\in \R^d$ to minimize the estimation error or the prediction error. Note that the estimator $\thetahat$ does not need to be $k$-sparse. Hence, our theoretical framework includes \emph{improper learners} which are allowed to output non-sparse estimates whenever they achieve small  risks.

The minimax risks of sparse linear regression have been well-studied. Under the same problem setting, \citet{raskutti2011minimax} proved information theoretic lower bounds on both the estimation error and the prediction error. Certain lower bounds have also been proved under the computation tractability constraint~\cite{zhang2014lower}, or proved for the family of regularized M-estimators~\cite{zhang2015optimal}. All these lower bounds handle the worst-case scenario --- given an arbitrary estimator, they prove the existence of a parameter $\theta$ that attains the lower bound. This setting might be too pessimistic in practice. %Indeed, the parameters that attain the lower bound have zero probability under any continuous prior, so that their on-average effects might be negligible.
%The fundamental limits of sparse linear regression under a realistic prior, to the best of the authors' knowledge, is unknown.
The goal of this section is to study the Bayes risk of sparse linear regression under a natural prior, whose construction is described in the next subsection.

\subsection{Prior Definition and Assumptions}
\label{sec:sparse-prior}

We define a prior over $k$-sparse $d$-dimensional vectors for the true parameter $\theta\in \R^d$, referred to as distribution $w$, as follows:
\begin{enumerate}
\item Uniformly sample a subset of $k$ indices from the integer set $\{1,2,\dots,d\}$, naming this subset by $K$.
\item For every index $i\in K$, the coordinate $\theta_i$ is generated by sampling from the normal distribution $N(0,\tau^2)$. For any $i\notin K$, define~$\theta_i \defeq 0$.
\end{enumerate}
Given an index set $K$, we use $\theta_K$ as a shorthand notation to denote the coordinates of the vector $\theta\in \R^d$ whose indices belong to the set $K$. Similarly, we use $\theta_{-K}$ to denote the subvector whose indices are not in~$K$. Then the the second step of the above generative process can be rephrased as generating $\theta_K\sim N(0,\tau^2 I_{k\times k})$  and defining $\theta_{-K} = 0$. It is clear that the sampled $\theta$ belongs to the $k$-sparse $\ell_0$-ball $\ball_0(k):=  \left\{\theta \in\R^d: ~ \|\theta\|_0 \leq k \right\}$.

One may consider variants of the the prior defined above. For example, one can assume that the number of non-zero entries of the vector $\theta$ is not exactly equal to $k$, but random sampled from a Poisson distribution with mean $k$. One may also redefine the prior of non-zero entries to be a non-Gaussian distribution. However, these variants don't add essential technical challenge to the analysis, thus we focus on the the prior $w$ as a concrete example for illustrating the general idea.

We make an additional assumption on the design matrix $X$ that is important for  characterizing the minimax risk~\citep[see, e.g.][]{raskutti2011minimax}, and in this section, we study their effects on the Bayes risk. Specifically, the design matrix $X$ satisfies the \emph{sparse eigenvalue conditions} with parameter $(\kappa_u, \kappa_\ell)$ if:
\begin{align}\label{eqn:sparse-eigenvalue-condition}
	\kappa_\ell \ltwos{\beta} \leq \frac{\ltwos{X \beta}}{\sqrt{n}} \leq \kappa_u \ltwos{\beta} \quad \mbox{for any $(2k)$-sparse vector $\beta\in \R^d$}.
\end{align}
Here, both $\kappa_u$ and $\kappa_\ell$ are positive constants. As a concrete example, if entries of the matrix $X$ are i.i.d.~sampled from a normal distribution, then the matrix is called a \emph{Gaussian random design}. This type of matrices have been extensively studied for sparse linear regression~\citep{candes2006stable,guedon2008majorizing}, and proved to satisfy condition~\eqref{eqn:sparse-eigenvalue-condition} with $\kappa_u/\kappa_{\ell} = \order(1)$~\citep{raskutti2010restricted}. For the rest of this section, we assume that the design matrix $X$ satisfies the condition~\eqref{eqn:sparse-eigenvalue-condition}.

\subsection{Bayes Risk Lower Bounds}

For sparse linear regression, we denote the parameter space and action space by $\Theta = \ball_0(k)$ and $\mathcal{A} = \R^d$, respectively. We present a Bayes risk lower bound with respect to the prior distribution defined in Section \ref{sec:sparse-prior}, then demonstrate its consequences.

\begin{theorem}\label{theorem:sparse-regression-lower-bound}
Assume that the design matrix $X$ satisfies the sparse eigenvalue condition~\eqref{eqn:sparse-eigenvalue-condition}, and that $d > k^3$. There are universal constants $c',c'' > 0$ such that for any $\tau > 0$, we have Bayes risk lower bounds: $\RBayes(w,\Lest;\Theta) \geq  c'\,T(\tau) $ and $\RBayes(w,\Lpre;\Theta) \geq  c''\,\kappa_\ell^2 T(\tau)$, where $T(\tau)$ is a term defined by
\begin{align}\label{eqn:sparse-regression-lower-bound}
\quad T(\tau) \defeq k\tau^2\,\max\Big\{\frac{1}{1 + \kappa_u^2 \tau^2 n/\sigma^2}, \exp\Big( - \frac{4\kappa_u^2 n}{\sigma^2} \Big[\tau^2 - \frac{\sigma^2 \log(d/k)}{16\kappa_u^2 n} \Big]_+\Big)\Big\}.
\end{align}
\end{theorem}

~\\
\noindent
The proof of Theorem~\ref{theorem:sparse-regression-lower-bound} follows the general strategy that we sketched in earlier sections: first, we bound the mutual informativity using the techniques described in Section~\ref{sec:upper_f_informativity}, then we upper bound the probability
$\sup_{a\in \A} w(B_t(a,L))$ for a specific scalar $t > 0$. Combining the two  upper bounds with Corollary~\ref{cor:bwl} establishes the theorem. See Appendix~\ref{sec:supp_sparse} for the proof. We make a few important remarks of this result in the below.

\paragraph{Estimation versus prediction} By Theorem~\ref{theorem:sparse-regression-lower-bound}, the lower bounds on the estimator error and the prediction error differ by a factor $\kappa_\ell^2$.
As a consequence, if we multiply a constant to the design matrix, then the term $\kappa_\ell^2$ will  also be scaled. If the scalar is very small, then the lower bound on the prediction error will be close to zero, but the lower on the estimation error won't. These are the right scaling for both risks. Indeed, when the design matrix converges to an all-zero matrix, the true parameters will be hard to identify, but the constant estimator $\thetahat\equiv 0$ will be able to achieve a small prediction error.

\paragraph{Comparison with minimax risk lower bounds}
It is worth comparing Theorem~\ref{theorem:sparse-regression-lower-bound} with the well-studied minimax risk lower bound. Under the sparse eigenvalue condition~\eqref{eqn:sparse-eigenvalue-condition}, \citet{raskutti2011minimax} proved the follow minimax risk lower bound:
\begin{align}\label{eqn:sparse-regression-minimax}
	\inf_{\thetahat} \max_{\theta\in \ball_0(k)} \E[\Lest(\theta,\thetahat)] \geq c'\, \frac{\sigma^2 k \log(d/k)}{\kappa_u^2 n} \quad\mbox{and}\quad
	\inf_{\thetahat} \max_{\theta\in \ball_0(k)} \E[\Lpre(\theta,\thetahat)] \geq c''\, \frac{ \kappa_\ell^2 \sigma^2 k \log(d/k)}{\kappa_u^2 n},
\end{align}
where $c'$ and $c''$ are universal constants. These bounds are matched by Theorem~\ref{theorem:sparse-regression-lower-bound}. In particular, if we assume $d > k^3$ and consider the prior distribution with variance:
\begin{align}\label{eqn:sparse-worst-variance}
	\tau^2 = \left( \tau_*^2 \defeq \frac{\sigma^2 \log(d/k)}{16\kappa_u^2 n}\right),
\end{align}
then expression~\eqref{eqn:sparse-regression-lower-bound} implies $T(\tau) = k\tau^2$, and as a consequence, we have
\begin{align}\label{eqn:sparse-regression-least-favorable}
	\RBayes(w,\Lest;\Theta) \geq c'\, \frac{\sigma^2 k \log(d/k)}{\kappa_u^2 n} \quad\mbox{and}\quad
	\RBayes(w,\Lpre;\Theta) \geq c''\, \frac{ \kappa_\ell^2 \sigma^2 k \log(d/k)}{\kappa_u^2 n},
\end{align}
where $c'$ and $c''$ are universal constants. The minimax risk lower bounds~\eqref{eqn:sparse-regression-minimax} and the Bayes risk lower bounds~\eqref{eqn:sparse-regression-least-favorable} thus match by a universal constant factor. Therefore, using our technique, we can directly obtain this classical minimax result on sparse linear regression. It is worth noting  that the lower bounds of~\citet{raskutti2011minimax} were proved by constructing a uniform prior over a discrete packing set over the parameter space. The existence of the proper packing set was proved in a non-constructive, worst-case fashion, which might be too pessimistic in practice. In contrast, our lower bound was established for a realistic and flexible prior which admits a simple closed-form definition and allows for different levels of variance. The theorem also shows that the prior $w$ with the variance level~\eqref{eqn:sparse-worst-variance} is in fact \emph{a least favorable prior} for sparse linear regression.

\paragraph{Bayes risk on the spectrum of priors} Besides the least-favorable setting~\eqref{eqn:sparse-worst-variance}, let us consider the Bayes risk under other choices of the parameter $\tau^2$. When $\tau^2 < \tau_*^2$, Theorem~\ref{theorem:sparse-regression-lower-bound} implies
\begin{align}\label{eqn:spetrum-case-1}
	\RBayes(w,\Lest;\Theta) \geq c'\, k\tau^2 \quad\mbox{and}\quad
	\RBayes(w,\Lpre;\Theta) \geq c''\, \kappa_\ell^2 k\tau^2.
\end{align}
When $\tau^2 \to +\infty$, Theorem~\ref{theorem:sparse-regression-lower-bound} implies
\begin{align}\label{eqn:spetrum-case-2}
	\RBayes(w,\Lest;\Theta) \geq c'\, \frac{k\sigma^2}{\kappa_u^2 n} \quad\mbox{and}\quad
	\RBayes(w,\Lpre;\Theta) \geq c''\, \frac{\kappa_\ell^2 k\sigma^2}{\kappa_u^2 n}.
\end{align}
In both cases, the Bayes risk lower bounds can be significantly smaller than the minimax risk. We argue that these lower bounds are essentially tight under specific assumptions. That is, when taking the prior information into account, we can indeed achieve better rates than the minimax rate.

First, notice that the upper bound:
\begin{align*}
	\E_{\theta\sim w}[\Lest(\theta,\thetahat)] \leq k\tau^2 \quad\mbox{and}\quad
	\E_{\theta\sim w}[\Lest(\theta,\thetahat)] \leq \kappa_u^2 k\tau^2.
\end{align*}
can always be achieved using the constant estimator $\thetahat \equiv 0$. It means that for the case of $\tau^2 < \tau_*^2$, the lower bounds~\eqref{eqn:spetrum-case-1} are tight under the assumption $\kappa_u/\kappa_\ell = \order(1)$.

For the case of $\tau^2 \to +\infty$, we consider the $\ell_0$-norm constrained estimator:
\begin{align}\label{eqn:l0-estimator}
	\thetahat \defeq \arg\inf_{\beta\in \ball_0(k)} \ltwos{X\beta - y}^2.
\end{align}
Whenever $\kappa_u/\kappa_\ell = \order(1)$, \citet{raskutti2011minimax} showed that the estimator~\eqref{eqn:l0-estimator} achieves an error bound $\ltwos{\thetahat-\theta}^2 \leq c\,\frac{k\log (d)}{n}$ with high probability for a constant $c > 0$. Suppose that $\tau^2 = C\,\frac{k\log (d)}{n}$ with a scaling factor $C > c$. For any $i\in K$, the expectation of $\theta_i^2$ is equal to $\tau^2$, so that the probability of $\theta_i^2 \leq c\,\frac{k\log (d)}{n}$ is bounded by $\order(c/C)$. It means that by choosing a large enough $C$ (specifically, choosing $C\gg ck$), the lower bound $\theta_i^2 > c\,\frac{k\log (d)}{n}$ will hold for every $i\in K$ with a probability close to 1. Combining this fact with the bound $\ltwos{\thetahat-\theta}^2 \leq c\,\frac{k\log (d)}{n}$, we find that the support of $\thetahat$ must agree with $K$, so that the estimator must satisfy:
\[
	\thetahat_K = \arg\inf_{\beta\in \R^k} \ltwos{X_K\beta - y}^2 \quad\mbox{and}\quad \thetahat_{-K} = 0,
\]
where $X_K$ is a submatrix of $X$ consisting of columns indexed by $K$. In other words, the vector $\thetahat_K$ is the least-square estimator for a $k$-dimensional linear regression problem. For estimators taking this form, both the estimation error and the prediction error are known to match the lower bound~\eqref{eqn:spetrum-case-2} with high probability.

\section{Conclusions}
In this paper, we presented lower bounds for the Bayes risk in abstract decision-theoretic problems. Our bounds are quite general and only require upper bounds on $\sup_{a \in \ac} w(B_t(a, L))$ and the $f$-informativity $I_f(w, \Ps)$ for their application. Because of the generality, the bounds are not always tight however. For example, the bounds involve $\sup_{a \in \ac} w(B_t(a, L))$ and this quantity becomes large when the prior $w$ has a spike. In such situations, our main Bayes risk lower bound in Theorem \ref{bwl} will not be tight. In specific examples, this looseness can be remedied by adhoc fixes such as the one described in Remark  \ref{rem:Bayes_decomp}. Obtaining tight lower bounds for the Bayes risk in the generality considered in  this paper is a challenging open problem.

%%As a nontrivial application of our Bayes risk inequalities, we presented an alternative proof of a recent admissibility result due to \citet{Sourav14LS} in Section \ref{sec:admis}. We remarked there that this admissibility result can be expressed in terms of the normalized minimax risk \eqref{nori}. We proved the inequality
%\begin{equation*}
%\inf_{n \geq 1} \inf_{\Theta \subseteq \R^n} \Rnm(\Theta) \geq 10^{-8}.
%\end{equation*}
%It is also possible to show (by taking $n = 1$ and $\Theta = [-a, a]$ for $a$ sufficiently small) that the left hand side above is not larger than $1/2$. We therefore have
%\begin{equation*}
%10^{-8} \leq \inf_{n \geq 1} \inf_{\Theta \subseteq \R^n} \Rnm(\Theta) \leq \frac{1}{2}.
%\end{equation*}
%It will be very interesting to obtain sharper constants in the above bound.
%
%

\section*{Acknowledgement}
Adityanand Guntuboyina  is supported by NSF Grant DMS-1309356. The authors would like to thank Michael I. Jordan and Sivaraman Balakrishnan for helpful discussions.


\newpage

\appendix
\section{Proofs and Additional Results for Section \ref{sec:Bayes_01} on Bayes Risk Lower Bound for Zero-one Loss}

\subsection{Proof of Lemma~\ref{lem:lower_property}}
Recall the expression \eqref{eq:phi} of $\phi_f(a, b)$. We first fix
$b$ and show that $g(a): a \mapsto \phi_f(a,b)$ is a
non-increasing for $a \in  [0, b]$. There is nothing to prove if $b =
0$ so let us assume that $b > 0$. We will consider the cases $0 < b <
1$ and $b = 1$ separately. For $0 < b < 1$, note that for every $a \in (0, b]$, we
have,
\begin{equation*}
  g_L'(a) = f_L' \left(\frac{a}{b} \right) - f_R' \left(\frac{1 -
      a}{1 -b} \right),
\end{equation*}
where $g_L'$ and $f_L'$ represent left derivatives and $f_R'$
represents right derivative (note that $f_L'$ and $f_R'$ exist because
of the convexity of $f$). Because $\frac{a}{b} \leq \frac{1-a}{1-b}$ for
every $0 \leq a \leq b$ and $f$ is convex, we see that
$$ g_L'(a) \leq  f_R' \left(\frac{a}{b} \right) -f_R' \left(\frac{1 -
      a}{1 - b} \right) \leq 0$$ for every $a \in (0, b]$ which
  implies that $g(a)$ is non-increasing on $[0, b]$.

When $b = 1$, we have $g_{L}'(a) = f_L'(a) - f'(\infty)$ which is
always $\leq 0$ because $f$ is convex (note that $f'(\infty) = \lim_{x
\uparrow \infty} f(x)/x = \lim_{x \uparrow \infty} (f(x) - f(1))/(x -
1)$).

The convexity and continuity of $g$ follow from the convexity of
$f$ and the expression for $\phi_f$.

Next, we fix $a$ and show that $h(b): b \mapsto \phi_f(a,b)$ is
non-decreasing for $b \in  [a, 1]$. For every $b \in [a, 1)$, we have,
\begin{equation}
  h_R'(b) = f\left( \frac{a}{b} \right) - \frac{a}{b} f'_L\left(\frac{a}{b} \right)-f\left(\frac{1-a}{1-b} \right) + \frac{1-a}{1-b} f'_R\left(\frac{1-a}{1-b} \right),
  \label{eq:h_R}
\end{equation}
where $h_R'$ represents the right derivative of $h$. By the convexity of $f$,
\begin{equation}
  \label{eq:h_R_1}
  f\left( \frac{a}{b} \right) -f\left(\frac{1-a}{1-b} \right) \geq f_R' \left(\frac{1-a}{1-b} \right)  \left(\frac{a}{b} - \frac{1-a}{1-b} \right).
\end{equation}
Combining \eqref{eq:h_R} with \eqref{eq:h_R_1}, we obtain that,
\begin{equation*}
h_R'(b) \geq \frac{a}{b} \left(f_R' \left(\frac{1-a}{1-b} \right)- f'_L\left(\frac{a}{b} \right) \right) \geq \frac{a}{b} \left(f_L' \left(\frac{1-a}{1-b} \right)- f'_L\left(\frac{a}{b} \right) \right) \geq 0,
\end{equation*}
where the last inequality is because that $\frac{a}{b} \leq \frac{1-a}{1-b}$ for every $0 \leq a \leq b$ and $f$ is convex. The non-negativity of $h_R'(b)$ implies that $h(b)$ is non-decreasing on $[a, 1]$.

\subsection{A Variant of Fano's Inequality from \cite{BraunPokutta}}
\label{sec:supp_Fano}
One of the main results in \cite{BraunPokutta} (Proposition 2.2)
establishes the following variant of Fano's inequality. Consider the
setting of Lemma \ref{lem.dpi}. In particular, recall the quantities
$R^{\dr}$ and $R_Q^{\dr}$ from \eqref{eq:def_R_Q} and also the sets
$B(a), a \in \ac$ from \eqref{ba}. \cite[Proposition
2.2]{BraunPokutta} proved the following: for any decision rule $\dr$,
\begin{equation}\label{bp}
  R^{\dr} \geq \frac{-I(w, \mathcal{P}) - H(R^{\dr}) - \log w_{\max}}{\log \left[ (1 - w_{\min})/w_{\max} \right]},
\end{equation}
where $H(x) := -x \log x - (1-x) \log(1-x)$, $w_{\min} := \inf_{a \in
  \ac} w(B(a))$ and $w_{\max} := \sup_{a \in  \ac} w(B(a))$.

Below we provide a proof of this inequality using Lemma
\ref{lem.dpi}. The proof given in \cite{BraunPokutta} is quite
different proof. Using \eqref{main.eq1} from Lemma \ref{lem.dpi} with
$f(x)= x \log x$, we have for any decision rule
\begin{equation*}
  \int_{\Theta} D_f(P_{\theta} \| Q) w(d\theta) \geq R^{\dr} \log
  \frac{R^{\dr}}{R_Q^{\dr}} + (1 - R^{\dr}) \log \frac{1 - R^{\dr}}{1
    - R_Q^{\dr}}.
\end{equation*}
We can rewrite this as
\begin{equation}\label{eq:Braun}
  \int_{\Theta} D_f(P_{\theta} \| Q) w(d\theta) \geq -H(R^{\dr}) -
  R^{\dr} \log R_Q^{\dr} - (1 - R^{\dr}) \log(1 - R_Q^{\dr})
\end{equation}
where $H(x) := -x \log x - (1-x) \log(1-x)$. Since $L$ in Lemma
\ref{lem.dpi} is zero-one valued.
\begin{equation}
  R_Q^{\dr} =1- \E_{Q} w(B(\dr(X)))
  \label{eq:R_Q}
\end{equation}
where $\E_Q$ denotes expectation taken under $X \sim Q$ and and
$B(\dr(X))$ is defined in \eqref{ba}.
As a result, we have
\begin{equation}\label{qb}
 1- \max_{a \in \ac} w(B(a)) \leq R_Q^{\dr} \leq 1- \min_{a \in \ac} w(B(a)).
\end{equation}
Using the bounds in~\eqref{qb} on the right hand side of \eqref{eq:Braun}, we deduce
\begin{equation*}
  \int_{\Theta} D_f(P_{\theta} \| Q) w(d\theta) \geq -H(R^{\dr}) - R^{\dr} \log \left(1 - w_{\min} \right) - (1 - R^{\dr})  \log w_{\max}.
\end{equation*}
where $w_{\min} := \inf_{a \in \ac} w(B(a))$ and $w_{\max} := \sup_{a \in  \ac} w(B(a))$ for notational simplicity.
Taking the infimum on the left hand side above over all probability measures $Q$, we obtain
\begin{equation*}
 I(w, \mathcal{P}) \geq -H(R^{\dr}) - R^{\dr} \log \left(1 - w_{\min} \right) - (1 - R^{\dr})  \log \left( w_{\max} \right).
\end{equation*}
Provided $w_{\min} + w_{\max} < 1$, one can rewrite the above
inequality as \eqref{bp}. This completes the proof of \eqref{bp}.

\subsection{Proof of Corollary \ref{cor:lb_01_f}}\label{colr}
\label{sec:cor_lb_01_f}

\begin{enumerate}
    \item \textbf{Proof of inequality \eqref{eq:chi}}: Applying
      Theorem \ref{man} with $f(x) = x^2 - 1$, we obtain
    \begin{equation*}
    I_{\chi^2}(w, \Ps) \geq \frac{(R_0-R)^2}{R_0(1-R_0)}
    \end{equation*}
    Because $R \leq R_0$, we can invert the above to obtain
    \eqref{eq:chi}.

  \item \textbf{Proof of inequality \eqref{eq:total}}:
    Theorem~\ref{man} with $f(x) = |x - 1|/2$ gives
    \begin{equation*}
     I_{TV}(w, \Ps) \geq  \frac{R_0}{2} \left|\frac{R}{R_0} -1 \right|  + \frac{1-R_0}{2} \left|\frac{1-R}{1-R_0} -1 \right| =R_0-R,
    \end{equation*}
    where the last equality uses the fact that $R \leq R_0$. Inverting
    the above inequality, we obtain \eqref{eq:total}.

  \item \textbf{Proof of inequality \eqref{eq:hellinger_3}}:
   Theorem \ref{man} with $f(x) = f_{1/2}(x) = 1 - \sqrt{x}$ gives
   \begin{equation}\label{enim}
     I_{f_{1/2}}(w, \Ps) \geq 1 - \sqrt{R R_0} - \sqrt{(1 - R)(1 -
       R_0)}.
   \end{equation}
  Assume that $P_{\theta}$ has density $p_{\theta}$ with respect to a
  common dominating measure $\mu$. We shall show below that
  \begin{equation}\label{exha}
    I_{f_{1/2}}(w, \Ps) = 1 - \sqrt{\int_{\samp} u^2 d\mu} \qt{where
      $u := \int_{\Theta} \sqrt{p_{\theta}} w(d\theta)$. }
  \end{equation}
  To see this, fix a probability measure $Q$ that has a density $q$
  with respect to $\mu$. We can then write
  \begin{equation*}
    \int_{\Theta} D_{f_{1/2}} (P_{\theta} \| Q) w(d\theta) = 1 -
    \int_{\samp} \sqrt{q} \left( \int_{\Theta} \sqrt{p_{\theta}}
      w(d\theta) \right) d\mu = 1 - \int_{\samp} \sqrt{q u^2} d\mu
  \end{equation*}
  It follows then from the Cauchy-Schwarz inequality that
\begin{equation*}
  \int_{\Theta} D_{f_{1/2}}(P_{\theta}||Q) w(\mathrm{d}\theta)=1-\int_{\samp} \sqrt{q u^2}\;\mathrm{d}\mu \geq 1- \sqrt{\int_{\samp} u^2 \; \mathrm{d} \mu},
\end{equation*}
with equality holding when $q$ is proportional to $u^2$. This proves
\eqref{exha}. We now see that
\begin{align}
\label{eq:hellinger_2}
 \int_{\samp} u^2 \; \mathrm{d} \mu & =\int_{\Theta} \int_{\Theta} \int_{\samp} \sqrt{p_{\theta}} \sqrt{p_{\theta'}} \; \mathrm{d}\mu~ w(\mathrm{d} \theta) w(\mathrm{d} \theta')  = 1- \frac{1}{2} h^2 %\int_{\Theta} \int_{\Theta} H^2(P_{\theta}\| P_{\theta'})   w(\mathrm{d} \theta) w(\mathrm{d} \theta') \nonumber
\end{align}
where $h^2$ is defined as
      \begin{equation}\label{eq:h_square}
        h^2=\int_{\Theta} \int_{\Theta}  H^2(P_{\theta}\| P_{\theta'})  w(\mathrm{d} \theta)w(\mathrm{d} \theta').
        \end{equation}
This, together with
\eqref{enim} and \eqref{exha}, gives the inequality
\begin{equation}\label{eq:hellinger_1}
  \sqrt{R R_0} + \sqrt{(1 - R) (1 - R_0)} \geq \sqrt{1 - \frac{h^2}{2}}
\end{equation}
Now under the assumption $h^2 \leq 2 R_0$, the right hand side of the
inequality \eqref{eq:hellinger_1} lies between $\sqrt{1-R_0}$ and
$1$. On the other hand, it can be checked that, as a function in $R$,
the left hand side of \eqref{eq:hellinger_1} is strictly increasing
from $\sqrt{1-R_0}$ (at $R=0$) to 1 at ($R=R_0$).  Therefore, from
\eqref{eq:hellinger_1}, we know that $R \geq \widehat{R}$ where
$\widehat{R} \in [0, R_0]$ is the solution to the equation
obtained by replacing the inequality \eqref{eq:hellinger_1} with an
equality. One can solve this equation and obtain two solutions. One of
two solutions can be discarded by the fact that $R \leq R_0$. %and
%                                the solution has to be continuous in
%                                $\int_{\samp} u^2 \; \mathrm{d} \mu$.
The other solution is given by:
\begin{equation*}
  \widehat{R} = R_0-(2R_0-1) \frac{h^2}{2} - \sqrt{R_0(1-R_0)} \sqrt{h^2(2-h^2)}
\end{equation*}
and thus we have $R \geq \widehat{R}$ which proves inequality
\eqref{eq:hellinger_3}.

We note that the lower bound on $R$ in \eqref{eq:hellinger_3} only
holds under the condition $h^2  \leq 2 R_0$.  When $h^2 > 2 R_0$,
inequality \eqref{eq:hellinger_3} holds for every $R \in [0, R_{Q^*}]$
and thus cannot provide a non-trivial lower bound on $R$. As an
example, when $\Theta = \ac = \{1, \dots, N\}$, $L(\theta, a) =
\mathbbm{I}\{\theta \neq a\}$ and $w$ is the uniform prior on
$\Theta$, it is easy to see that $R_0 = 1 - (1/N)$ and
\begin{equation}
  h^2= \frac{1}{N^2} \sum_{\theta \neq \theta'} H^2(P_{\theta}\|
  P_{\theta'}) \leq 2 \frac{N(N-1)}{N^2} =2 R_{Q^*}.
\end{equation}
Inequality~\eqref{eq:hellinger_3} therefore is equivalent to
\begin{equation*}
    R \geq 1- \frac{1}{N} -\frac{N-2}{N} \frac{h^2}{2}-
    \frac{\sqrt{N-1}}{N} \sqrt{h^2(2-h^2)}.
\end{equation*}
This recovers the result in Example \RNum{2}.6 in \cite{Aditya:fdiv}.
\end{enumerate}

\subsection{Derivations of Le Cam's Inequality (Two Hypotheses) and Assouad's Lemma and other Results from Corollary \ref{cor:lb_01_f}}
\label{sec:supp_LeCam}

To demonstrate the application of Corollary \ref{cor:lb_01_f}, we apply it to derive the two hypotheses version of Le Cam's inequality (with total variation distance) and Assouad's lemma (see Theorem 2.12 in \citep{Tsybakov:nonpara}).

The simplest version of the Le Cam's inequality, the so-called two-point argument, is an easy corollary of~\eqref{eq:total}. Indeed, applying~\eqref{eq:total} with $\Theta = \ac = \{\theta_0, \theta_1\}$, $L(\theta, a) = \mathbbm{I}\{\theta \neq a\}$ and $w\{0\} = w\{1\} = 1/2$ (and note that $R_0=1/2$), we obtain that for any distribution $Q$ on $\mathcal{X}$,
\begin{equation*}
  \frac{1}{2} \left(\|P_{\theta_0} - Q\|_{TV} + \|P_{\theta_1} - Q\|_{TV} \right) \geq I_{TV}(w, \mathcal{P}) \geq 1/2 - R.%_{Bayes}(w, L).
\end{equation*}
Taking $Q = (P_{\theta_0} + P_{\theta_1})/2$, we obtain Le Cam's inequality:
\begin{equation}\label{eq:LeCam}
  R_{\rm minimax} \geq \frac{1}{2} \left(1 - \|P_{\theta_0} - P_{\theta_1}\|_{TV} \right).
\end{equation}
The more involved Le Cam's inequality considers $\Theta = \ac = \Theta_0 \cup \Theta_1$ for two disjoint subsets $\Theta_0$ and $\Theta_1$ and loss function $L(\theta, a) = \mathbbm{I} \{\theta \in \Theta_1, a \in \Theta_2\} + \mathbbm{I}\{\theta \in \Theta_2, a \in \Theta_1\}$. The inequality states that for every pair of probability measures $w_0$ and $w_1$ concentrated on $\Theta_0$ and $\Theta_1$ respectively,
\begin{equation}\label{eq:LeCam_set}
  R_{\rm minimax} \geq \frac{1}{2} \left(1 - \|m_0 - m_1\|_{TV} \right)
\end{equation}
where $m_{0}$ and $m_1$ are marginal densities given by $m_{\tau}(x) = \int p_{\theta}(x) w_{\tau}(\mathrm{d}\theta)$ for $\tau = 0, 1$. To prove~\eqref{eq:LeCam_set}, consider the prior $w = (w_0 + w_1)/2$. Under this prior, the problem is easily converted to the previous binary testing problem.  In particular, the data generating process under the prior $w$ can be viewed as first sampling $\tau \sim \mathrm{Uniform}\;\{0,1\}$ and then $X \sim m_{\tau}$. The decision $a \in \mathcal{A}$ can be converted into the binary decision $\hat{\tau}=\mathbbm{I}(a \in \Theta_1)$. The loss function is $L(\tau, \hat{\tau}) = \mathbbm{I}(\tau \neq \hat{\tau}) $. The Bayes risk under the prior $w$ can be re-written as,
\begin{equation}
  R_{\rm Bayes}(w, L; \Theta)  = \frac{1}{2} \inf_{\hat{\tau}} \sum_{\tau=0,1}  \int_{\mathcal{X}}   \mathbbm{I}(\tau \neq \hat{\tau}(x))  m_{\tau}(x)  \mu(\mathrm{d} x),
  \label{eq:LeCam_set_1}
\end{equation}
which has the same form as the Bayes risk in the earlier binary testing problem. Applying the same argument as for proving \eqref{eq:LeCam}, we obtain the lower bound on the Bayes risk in \eqref{eq:LeCam_set_1}, $R_{\rm Bayes}(w, L; \Theta) \geq \frac{1}{2} \left(1 - \|m_0 - m_1\|_{TV} \right)$, which further implies \eqref{eq:LeCam_set}.

Another classical minimax inequality involving the total variation distance is Assouad's inequality~\citep{Assouad:83} which states that if $\Theta = \mathcal{A} = \{0,1\}^d$ and the loss function $L$ is defined by the Hamming distance, i.e., $L(\theta, a) = \sum_{i=1}^d \mathbbm{I}(\theta_i \neq a_i)$, then
\begin{eqnarray}
  \Rmx \geq \frac{d}{2} \min_{L(\theta, \theta')=1} \left(1-\|P_{\theta} - P_{\theta'}\|_{TV} \right).
  \label{eq:assouad}
\end{eqnarray}
This inequality is also a consequence of~\eqref{eq:total}: let $w$ be the uniform probability measure on $\Theta$ and $L_1(\theta, a)= \mathbbm{I}(\theta_1 \neq a_1)$. Under $w$, the marginal distribution of the first coordinate is $w_1\{0\} = w_1\{1\} = 1/2$. Let $m_\tau(x) :=\sum_{\theta: \theta_1=\tau} p_{\theta}(x)/2^{d-1}$  for $\tau \in \{0,1\}$ be the corresponding marginal density of $X$ and let $Q(x)= \frac{1}{2} \left( m_0(x)+m_1(x) \right)$. Applying the same argument as for proving \eqref{eq:LeCam}, we obtain that the minimax risk for the zero-one valued loss function $L_1(\theta, a)$ is bounded below by $\frac{1}{2} \left( 1- \|m_0 - m_1\|_{TV} \right) \geq \frac{1}{2} \min_{L(\theta, \theta')=1} \left(1-\|P_{\theta} - P_{\theta'}\|_{TV} \right)$. Repeating this argument for $L_i(\theta, a) := \mathbbm{I}\{\theta_i \neq a_i\}$ for $i = 2, \dots, d$ and adding up the resulting bounds, we obtain~\eqref{eq:assouad}.

By using Le Cam's inequality (see, e.g., Lemma 2.3 in \citep{Tsybakov:nonpara}) which states that:
\begin{equation*}
  \|P_{\theta} - P_{\theta'}\|_{TV} \leq \sqrt{H^2(P_{\theta} \| P_{\theta'}) \left(1 - \frac{1}{4} H^2(P_{\theta} \| P_{\theta'}) \right)},
\end{equation*}
the inequality in \eqref{eq:assouad} further implies the Hellinger distance version of Assouad's inequality in the book \citet[Theorem 2.12]{Tsybakov:nonpara}, i.e.,
\begin{equation}\label{eq:hell_assouad}
  R_{\rm minimax} \geq \frac{d}{2} \min_{L(\theta, \theta') = 1} \left\{1 - \sqrt{H^2(P_{\theta} \| P_{\theta'}) \left(1 - \frac{1}{4} H^2(P_{\theta} \| P_{\theta'}) \right)} \right\} .
\end{equation}

\subsection{Comparison of the Bounds for Different Divergences}
\label{sec:compare}
We provide some qualitative comparisons of Bayes risk lower bounds given by Theorem \ref{man} for different power divergences. %Specifically, we compare the lower bound \eqref{gf} (corresponding to KL divergence) with which provided by Theorem \ref{man}.
In particular, let us consider the discrete setting where $\Theta = \ac = \{\theta_1, \dots, \theta_N\}$, $L(\theta, a) = \mathbbm{I}\{\theta \neq a\}$, and $w$ is the discrete uniform. Note that in such a ``multiple testing problem" setup, $R_0$ is equal to $1 - (1/N)$. We take $N$ sufficiently large so that $R_0$ is close to 1. To establish minimax lower bounds, a typical approach is to reduce the estimation problem to a multiple hypotheses testing problem in the aforementioned setup, then try to prove that the Bayes risk $R \geq c>0$ (see Section 2.2. in \cite{Tsybakov:nonpara}). Without loss of generality, we take $c = 1/2$ and we shall see how the three  inequalities  \eqref{gf}, \eqref{eq:chi} and \eqref{eq:hellinger_3} work to establish $R \geq 1/2$.

Let us start with \eqref{gf} corresponding to KL divergence, which is equivalent to the classical Fano's  inequality \eqref{eq:classcial_Fano} in the discrete setting. To establish $R \geq 1/2$, the following condition should hold:
\begin{equation}\label{nef}
  I(w, \Ps) \leq \frac{1}{2} \log \left(\frac{N}{4} \right).
\end{equation}
We remark that $I(w, \Ps)$ is at most $\log N$ even if every the pairwise KL divergence $D(P_{\theta_i} \| P_{\theta_j})$ equals $\infty$ for $i \neq j$.  This fact will be clear from the inequality \eqref{ybp} from Section \ref{sec:upper_f_informativity} (let $M=N$ and $Q_j=P_{\theta_j}$ for $1 \leq j \leq M$). The upper bound on $I(w, \mathcal{P})$ in  \eqref{ybp} further provides a sufficient condition to verify \eqref{nef}. %The inequality \eqref{ybp} in Section \ref{sec:upper_f_informativity} provides a sufficient condition for upper bounding $I(w, \Ps)$.

Now we turn to \eqref{eq:chi} corresponding to the chi-squared divergence. Since $R_0 = 1-(1/N)$, inequality \eqref{eq:chi} implies a sufficient condition for $R \geq 1/2$:
\begin{equation}\label{nec}
  I_{\chi^2}(w,\Ps) \leq \frac{N^2}{N-1} \left(\frac{1}{2} - \frac{1}{N} \right)^2.
\end{equation}
When $N$ is large, the above condition is equivalent to $I_{\chi^2}(w, \Ps) \leq N/4$.  Note that the maximum possible value of $I_{\chi^2}(w, \Ps)$ in this discrete setting is $N - 1$ (even when $\chi^2(P_{\theta_i} \| P_{\theta_j})  = \infty$ for every $i \neq j$) and this follows from our upper bounds on $f$-informativity for a class of power divergences in \eqref{eq:upper_f_power} (see Section \ref{sec:upper_f_informativity}).

%We remark that the maximum possible value of $I_{\chi^2}(w, \Ps)$ is $N - 1$ even if $\chi^2(P_{\theta_i} \| P_{\theta_j})  = \infty$ for every $i \neq j$. The inequality \eqref{eq:upper_f_power} in Section \ref{sec:upper_f_informativity} provides a sufficient condition for upper bounding $I_{\chi^2}(w, \Ps)$.

The conditions \eqref{nef} and \eqref{nec} don't imply each other. The chi-squared divergence is always greater than the KL divergence (see Lemma 2.7 in \cite{Tsybakov:nonpara}), but the upper bound required by~\eqref{nec} is also weaker than that required by~\eqref{nef}. For both divergences, constructing more hypotheses (i.e., choosing $N>2$) is often helpful for showing $R \geq 1/2$.

For the Hellinger distance (inequality \eqref{eq:hellinger_3}), we claim that it gives no more useful bounds than those obtained by a simple two point argument. To see this, since $R_0 = 1 - (1/N)$,  inequality \eqref{eq:hellinger_3} implies
\begin{equation*}
  R \geq 1 - \frac{1}{N} - \frac{N-2}{N} \frac{h^2}{2} - \frac{\sqrt{N-1}}{N} \sqrt{h^2(2 - h^2)}
\end{equation*}
where $h^2 = \sum_{i, j} H^2(P_{\theta_i} || P_{\theta_j})/N^2$. When $N$ is large, the above inequality reduces to effectively $R \geq 1 - (h^2/2)$. Therefore a sufficient condition for $R \geq 1/2$ is $h^2 \leq 1$, which is equivalent to,
\begin{equation*}
  \frac{1}{N(N-1)/2} \sum_{i < j} H^2(P_{\theta_i}|| P_{\theta_j}) \leq \frac{N}{N-1}.
\end{equation*}
When $N$ is large, the above displayed condition implies the existence of $i < j$ for which $H^2(P_{\theta_i}|| P_{\theta_j}) \leq 1$. Let $\tilde{w}$ denote the prior $\tilde{w}\{i\} = \tilde{w}\{j\} = 1/2$. It is easy to see that the Bayes risk for $\tilde{w}$ equals
$
  R_{\rm Bayes}(\tilde{w}) = \frac{1}{2} \left(1 - \|P_{\theta_i} - P_{\theta_j}\|_{TV} \right).
$
By Le Cam's inequality (see Lemma 2.3 in \cite{Tsybakov:nonpara}), we have,
\begin{equation*}
  R_{\rm Bayes}(\tilde{w}) \geq \frac{1}{2} \left( 1 - H(P_{\theta_i} || P_{\theta_j}) \sqrt{1 - \frac{H^2(P_{\theta_i}|| P_{\theta_j})}{4}} \right)
\end{equation*}
Since $H(P_{\theta_i} || P_{\theta_j}) \leq 1$, it is easy to verify from the above that $R_{\rm Bayes}(\tilde{w}) \geq 1/8$. Therefore in this discrete setting, if inequality \eqref{eq:hellinger_3} implies $R_{\rm Bayes}(w) \geq 1/2$, then there is a much simpler two point prior $\tilde{w}$ for which $R_{\rm Bayes}(\tilde{w}) \geq 1/8$. It shows that for Hellinger distance, considering $N>2$ hypotheses is not more useful than using a pair of hypotheses. The reason is that the Hellinger informativity can be written as an expression involving pairwise Hellinger distances. In particular, it can be seen from the proof of inequality \eqref{eq:hellinger_3} that
\[
  I_{f_{1/2}}(w, \Ps) = 1 - \Bigl(1 - \frac{1}{2N^2} \sum_{i, j} H^2(P_{\theta_i} || P_{\theta_j})\Bigr)^{1/2}.
\]
In contrast, the mutual information, $I(w, \Ps)$, cannot be written in terms of $D(P_{\theta_i} \| P_{\theta_j})$ for $i \neq j$ (recall that $I(w, \Ps)$ is always at most $\log N$ even when $D(P_{\theta_i} \| P_{\theta_j}) = \infty$ for all $i \neq j$). The same holds for $I_{\chi^2}(w, \Ps)$ as well (which is always at most $N - 1$ even if $\chi^2(P_{\theta_i} \| P_{\theta_j}) = \infty$ for all $i \neq j$).

%If we want to derive minimax risk lower bound up to a multiplicative constant, then the bound induced by the Hellinger distance is not very useful.

If the eventual goal of obtaining Bayes risk lower bounds is to obtain lower bounds up to multiplicative constants on the minimax risk, then the bound in \eqref{eq:hellinger_3} gives no more useful bounds than those obtained by the simple two point argument. In this sense, inequality \eqref{eq:hellinger_3} induced by Hellinger distance is not as useful as inequalities \eqref{gf} and \eqref{eq:chi}. In fact, the Hellinger distance is seldom used in lower bounding minimax risk involving many hypotheses (for example, none of the minimax rates in the examples of  \cite{Tsybakov:nonpara} involving multiple hypotheses testing are established via Hellinger distance).

\section{Proofs and Additional Results for Section
  \ref{sec:upper_f_informativity} on Upper Bounds on
  $f$-informativity} \label{ange}

\subsection{Proof of Lemma \ref{lem:concavity}}\label{sec:proof_concave}

 Let $\phi(t) \equiv t^r$ with $\phi'(t)=r t^{r-1}$ and
  $\phi''(t)=r(r-1) t^{r-2}$  and $\varphi(t)=t^{1/r}$ with
  $\varphi'(t)=\frac{1}{r}t^{(1-r)/r}$. Then
  $$f(u)=\varphi\left(\int_{T}  \phi(u(t)) \mu(\mathrm{d}t) \right).$$
  To prove the concavity of $f(u)$, considering the scalar function
  \begin{eqnarray}
    h(s) = \varphi\left(\int_{T}  \phi(u(t)+s v(t) ) \mu(\mathrm{d}t) \right),
  \end{eqnarray}
  for arbitrary $u, v \in  L^{r}_{\mu}(T)$. We notice that concavity
  of $f$ is equivalent to concavity at zero for all functions of the
  form $h$, and we therefore only have to show that $h''(0) \leq
  0$. Let $g(s)= \int_{T}  \phi(u(t)+s v(t) ) \mu(\mathrm{d}t)$,
  \begin{align*}
    h'(s) = & \varphi'(g(s)) \int_T \phi'(u(t)+sv(t))v(t) \mu(\mathrm{d} t) \\
    h''(s)  = &\varphi''(g(s))\left( \int_T \phi'(u(t)+sv(t))v(t) \mu(\mathrm{d} t) \right)^2  \\
               & + \varphi'(g(s)) \int_T \phi''(u(t)+sv(t))v^2(t) \mu(\mathrm{d} t)
  \end{align*}
  By plugging in the definitions of $\phi(t)$, $\varphi(t)$, $g(s)$ and setting $s=0$, we have
  \begin{eqnarray*}
    h''(0)=\frac{1-r}{f(u)}\left( \left(f(u)^{1-r} \int_T u(t)^{r-1} v(t) \mu(\mathrm{d}t) \right)^2- f(u)^{2-r} \int_T u(t)^{r-2} v^2(t) \mu(\mathrm{d}t)\right)
  \end{eqnarray*}
  Applying the Cauchy-Schwarz inequality $$\left(\int_T a(t)b(t)
    \mu(\mathrm{d}t)\right)^2  \leq  \left(\int_T a(t)^2
    \mu(\mathrm{d}t)\right) \left(\int_T b(t)^2
    \mu(\mathrm{d}t)\right) $$ with
  $a(t)=\left(\frac{f(u)}{u(t)}\right)^{-r/2}$ and $b(t) = v(t)
  \left(\frac{f(u)}{u(t)}\right)^{1-r/2}$ and noticing that $r<1$, we
  have $h''(0)\leq 0$, which completes the proof.

\subsection{Example Demonstrating the Effectiveness of Theorem~\ref{cd}}\label{effb}
In this example, we show the tightness of the upper bound in \eqref{co1} in terms of chi-squared divergence ($\alpha = 2$). In particular, let the distribution $P$ be the $n$-fold product of $N(0,1)$ and $Q_{\xi}$ be the $n$-fold product of $N(\xi,1)$ where $\xi \sim N(0,1)$. It is straightforward to show that the marginal distribution $\bar{Q}$ is a $n$-dimensional Gaussian distribution with mean $\mathbf{0}$ and covariance matrix $I_n+ \mathbf{1}_n \mathbf{1}_n ^T$, where $\mathbf{1}_n$ denotes the $n$-dimensional all one vector and $I_n$ the $n\times n$ identity matrix.

  Since $\chi^2(P||Q_{\xi})=\exp(n \xi^2)-1$, the right hand side of \eqref{co1} equals to $\sqrt{2n+1}-1$. The term   $\chi^2(P||\bar{Q})$ on the left hand side of \eqref{co1}  is difficult to evaluate. However, we can lower bound $\chi^2(P||\bar{Q})$ using the following standard inequality $\exp\left( D(P||\bar{Q}) \right) -1 \leq \chi^2(P||\bar{Q})$ (see Lemma 2.7 in \citet{Tsybakov:nonpara}).
  By the closed-form expression for KL divergence between two multivariate Gaussian distributions, we have $D(P||\bar{Q})=\frac{1}{2} \left( \log(n+1) - n/(n+1) \right)$ and thus
  \begin{eqnarray*}
    e^{-1/2} \sqrt{n+1}-1 \leq \exp\left( D(P||\bar{Q}) \right) -1  \leq \chi^2(P||\bar{Q})
  \end{eqnarray*}
  As we can see, the upper bound $\sqrt{2n+1}-1$ in \eqref{co1} is quite tight and $\chi^2(P||\bar{Q})$ is on the order of $\sqrt{n}$.

  \subsection{Proof of Corollary~\ref{nd}}\label{sec:cor_nd}
Fix $0 < \delta \leq A^{-1/2}$. Partition the entire parameter space $\Theta$ into small hypercubes each with side length $\delta$. For each such hypercube $S$ and let $\pi_S$ denote the probability measure $w$ conditioned to be in $S$ i.e., $\pi_S(C) := w(C)/w(S)$ for measurable set $C \subseteq S$.

For every decision rule $\dr(X)$, clearly
\begin{equation*}
  \int_{\Theta} \E_{\Theta} L(\theta, \dr(X)) w(\mathrm{d}\theta) = \sum_{S} w(S) \int_S \E_{\theta} L(\theta, \dr(X)) d\pi_S(\theta)
\end{equation*}
where the sum above is over all hypercubes $S$ in the partition. This implies therefore that
\begin{equation*}
  R_{\rm Bayes}(w, L; \Theta) \geq \sum_{S} w(S) R_{\rm Bayes}(\pi_S, L; S).
\end{equation*}
The proof will therefore be completed if we show that
\begin{equation}\label{rsp}
  R_{\rm Bayes}(\pi_S, L; S) \geq \frac{1}{2} e^{-2p} 8^{-p/d} \delta^p V^{-p/d} \int_S \left(\frac{1}{r_{\delta}(\theta)}  \right)^{p/d} \pi_S(d\theta)
\end{equation}
for every fixed hypercube $S$. So let us fix $S$ and, for notational simplicity, let $\pi := \pi_S$. We will
use~\eqref{eq:I_bayes_KL} to prove a lower bound on $R_{\rm Bayes}(\pi_S, L; S)$. Note first that
\begin{multline} \label{eq:fano_mutual}
  \inf_{Q} \int_S D(P_{\theta} || Q) \pi(\mathrm{d}\theta)  \leq   \int_S  \int_S D(P_{\theta} || P_{\theta'}) \pi(\mathrm{d}\theta) \pi(\mathrm{d}\theta')   \\
   \leq   A \max_{\theta \in S, \theta' \in S} \|\theta-\theta'\|_2^2      \leq   A d \delta^2 =: \Iu .
\end{multline}
Also, letting $f_w^{\max}$ and $f_w^{\min}$ be the maximum and minimum values of $f_w$ in $S$, we have
\begin{eqnarray*}
  \sup_{a \in S } \pi(B_t(a, L)) \leq \frac{f_w^{\max}}{w(S)} \text{Vol}(B_t(a, L ))   \leq \frac{f_w^{\max} V t^{d/p}}{f_w^{\min} \delta^d} .
\end{eqnarray*}
Let $\widetilde \theta$ be an arbitrary point in the set $S$. Since $S$ has diameter $\sqrt{d} \delta$, the set $\{\theta: \|\theta-\widetilde \theta\|_2 \leq \sqrt{d} \delta \}$
contains $S$.  We obtain from the definition of $r_\delta(\theta)$
that $f_w^{max}/f_w^{min} \leq r_{\delta}(\tilde{\theta})$ so that
\begin{equation*}
  \sup_{a \in S } \pi(B_t(a, L)) \leq r_{\delta}(\tilde{\theta}) V \delta^{-d} t^{d/p}.
\end{equation*}
Thus, by~\eqref{eq:fano_mutual}, the choice
\begin{eqnarray*}
  t= e^{-2p A \delta^2} \delta^p \left( \frac{1}{8 V r_{\delta}(\widetilde \theta)}\right)^{p/d},
\end{eqnarray*}
leads to $\sup_{a \in S} \pi(B_t(a, L)) < \frac{1}{4} e^{-2\Iu}$. Employing \eqref{eq:I_bayes_KL}, we deduce
\begin{equation*}
  R_{\rm Bayes}(\pi, L; S) \geq \frac{1}{2} e^{-2p A \delta^2} \delta^p \left( \frac{1}{8 V r_{\delta}(\widetilde \theta)}\right)^{p/d} \geq \frac{1}{2}e^{-2p} \delta^p \left( \frac{1}{8 V  r_{\delta}(\widetilde \theta)}\right)^{p/d}
\end{equation*}
where we used the fact that $\delta^2 \leq 1/A$. Because $\tilde{\theta} \in S$ is arbitrary, we can write
\begin{eqnarray*}
    R_{\rm Bayes}(\pi, L; S) & \geq &  \frac{1}{2} e^{-2p} \delta^p (8V)^{-p/d} \sup_{\tilde \theta\in S} \left( \frac{1}{r_{\delta}(\tilde \theta)}\right)^{p/d} \\
    & \geq & \frac{1}{2} e^{-2p} \delta^p (8V)^{-p/d} \int_{S} \left(\frac{1}{r_{\delta}(\theta)}\right)^{p/d} \pi(\mathrm{d}\theta).
\end{eqnarray*}
This proves~\eqref{rsp}.

\section{More Examples on Bayes Risk Lower Bounds}
\label{sec:Bayes_Example}

In this section, we provide more examples on the applications of derived Bayes risk lower bound in Theorem \ref{bwl} and Corollary \ref{cor:bwl}. For the clarity of the presentation, in each example, we will first present the Bayes risk lower bound and then provide the proof.

\subsection{Generalized Linear Model}

Fix $d \geq 1$ and let $\Theta = \ac = \R^d$ with $L(\theta, a) = \|\theta - a\|_2^p$ for a fixed $p > 0$. Also fix $n \geq 1$ and an $n \times d$ matrix $X$ whose rows are written as $x_1^T, \dots, x_n^T$. As in the last example, $\lambda_{\max}$ denotes the maximum eigenvalue of $X^T X/n$.

For $\theta \in \Theta$, let $P_{\theta}$ denote the joint distribution of independent random variables $Y_1, \dots, Y_n$ where $Y_i$ has the density
\begin{equation}\label{eq:glm_density}
  \exp \left[\frac{y \beta_i - b(\beta_i)}{a(\phi)}  + c(y, \phi) \right] \qt{for $y \in \R$}
\end{equation}
with $\beta_i = x_i^T \theta$ for $i = 1, \dots, n$. The parameter $\phi$ is taken to be a constant and the functions $a(\cdot), c(\cdot, \cdot)$ and $b(\cdot)$ are assumed to be known. We assume the existence of a constant $K > 0$ such that $b''(\beta) \leq K$ for all $\beta$ where $b''(\cdot)$ is the second derivative of $b(\cdot)$. This assumption indeed holds for many generalized linear models (e.g., binomial, Gaussian) and we will discuss the case (i.e., Poisson) where this assumption fails at the end of this example.

Let $w$ denote the Gaussian prior with mean zero and covariance matrix $\tau^2 I_d$. Using Corollary~\ref{nd}, we can prove that
\begin{equation}\label{mee}
  R_{\rm Bayes}(w, L; \Theta) \geq C \left[ d \min \left(\frac{a(\phi)}{n K }, \tau^2  \right) \right]^{p/2}
\end{equation}
for a constant $C$ that depends only on $p$. Let us illustrate this lower bound by considering a simple case of $p=2$. We note that the term $ \frac{d a(\phi)}{n K}$ is the well-known minimax risk of generalized linear model under  the squared loss. The parameter $\tau$ characterizes the strength of the prior information. In fact, since $\tau^2 I$ is the variance of the Gaussian prior distribution, a small value of $\tau$ provides strong prior information that each $\theta_j$ should be concentrated around $0$. When $\tau$ is large, i.e., with less prior information, the lower bound of the Bayes risk in~\eqref{mee} is the same as the minimax risk up to a constant factor. On the other hand, when $\tau$ is small, i.e., with strong prior information, the lower bound of the Bayes risk becomes $d \tau^2$, which is smaller than the minimax risk.

The proof of~\eqref{mee} will involve Corollary~\ref{nd} for which we need to determine $A, V$ and $r_{\delta}(\theta)$. As before, it is easy to check that $V = \mathrm{Vol}(B)$. To determine $A$, fix a pair $\theta_1, \theta_2$ and, letting $\beta_i^{(j)} = x_i^T \theta_j$ for $j = 1, 2$ and $i=1, \ldots, n$, observe that
\begin{equation*}
   D(P_{\theta_1} || P_{\theta_2}) =   \frac{1}{a(\phi)} \sum_{i=1}^n \left(b'(\beta^{(1)}_{i}) \left( \beta^{(1)}_{i} -\beta^{(2)}_{i} \right) -\left(b(\beta^{(1)}_{i})- b(\beta^{(2)}_{i}) \right)  \right)
\end{equation*}
By the second order Taylor expansion of $b(\beta^{(2)}_{i})$ at the point $\beta^{(1)}_{i}$, we  obtain
\begin{equation*}
  D(P_{\theta_1} || P_{\theta_2})  =  \frac{1}{a(\phi)} \sum_{i=1}^n \frac{b''(\tilde{\beta}_i)}{2} ( \beta^{(1)}_{i} -\beta^{(2)}_{i})^2
\end{equation*}
where $\tilde{\beta}_i$ lies between $\min(\beta^{(1)}_{i}, \beta^{(2)}_{i})$ and $\max(\beta^{(1)}_{i},  \beta^{(2)}_{i})$. Now because of our assumption that $b''(\cdot) $ is bounded from above by $K$, we get
\begin{align*}
  D(P_{\theta_1} \| P_{\theta_2}) & \leq \frac{K}{2 a(\phi)} \|\beta^{(1)} - \beta^{(2)} \|_2^2 = \frac{K}{2 a(\phi)} (\theta_1 - \theta_2)^T X^T X (\theta_1 - \theta_2) \\ & \leq \frac{n K \lambda_{\max}}{2 a(\phi)} \|\theta_1 - \theta_2\|^2.
\end{align*}
We can thus take $A = n K \lambda_{\max}/(2 a(\phi))$ in Corollary~\ref{nd}. Next we control $r_{\delta}(\theta)$. For given $\theta$ and $\delta$,
\begin{eqnarray*}
  r_{\delta}(\theta)= \sup \left\{ \exp\left(-\frac{1}{2 \tau^2} \left( \|\theta_1\|_2^2 - \|\theta_2\|_2^2 \right) \right) : \|\theta_i - \theta\|_2 \leq \sqrt{d} \delta \right\}.
\end{eqnarray*}
For $\theta_1, \theta_2$ with $\|\theta_i - \theta\|_2 \leq \sqrt{d} \delta,\; i = 1, 2$, we have
\begin{eqnarray*}
  \left|  \|\theta_1\|_2^2 - \|\theta_2\|_2^2  \right|  & = & \left|  \|\theta_1-\theta\|_2^2 +2 \theta^T(\theta_1-\theta) - \|\theta_2-\theta\|_2^2- 2 \theta^T(\theta_2-\theta)  \right| \\
  & \leq & \left|  \|\theta_1-\theta\|_2^2  - \|\theta_2-\theta\|_2^2 \right| + 2\| \theta\|_2 \left(\|\theta_1-\theta\|_2+\|\theta_2-\theta\|_2  \right) \\
  & \leq &  d \delta^2 + 4 \sqrt{d} \delta \|\theta\|_2.
\end{eqnarray*}
As a result $r_{\delta}(\theta)^{-p/d} \geq \exp(-p\delta^2/(2 \tau^2)) \exp(-2p\delta \|\theta\|_2/(\tau^2 \sqrt{d}))$ and hence
\begin{eqnarray*}
  \int_{\Theta} \left(\frac{1}{r_{\delta}(\theta)}\right)^{p/d} w(\mathrm{d}\theta) &\geq & \exp \left(  -\frac{p \delta^2}{2 \tau^2} \right) \int_{\Theta} \exp\left(-\frac{2p\delta }{\tau} \frac{\|\theta\|_2}{\tau \sqrt{d}} \right) w(\mathrm{d}\theta)\\
& \geq & \exp \left(  -\frac{p \delta^2}{2 \tau^2} - \frac{4 p \delta}{\tau} \right)  \int_{\Theta} \mathbb{I} \left\{\|\theta\|_2 < 2 \tau \sqrt{d} \right\} w(\mathrm{d}\theta).
\end{eqnarray*}
By Chebyshev's inequality, we have
\begin{equation}\label{eq:Gauss_con}
  \int_{\Theta}\mathbb{I} \left\{\|\theta\|_2 \geq 2 \tau \sqrt{d} \right\} w(\mathrm{d}\theta) \leq \frac{1}{4 \tau^2 d} \int_{\Theta} \|\theta\|_2^2 w(\mathrm{d}\theta) = \frac{1}{4}.
\end{equation}
Consequently,
\begin{eqnarray}
\int_{\Theta} \left(\frac{1}{r_{\delta}(\theta)}\right)^{p/d} w(\mathrm{d}\theta)\geq \frac{3}{4}\exp \left(  -\left( \frac{p \delta^2}{2 \tau^2} +\frac{ 4p \delta}{\tau} \right)\right).
\label{eq:Bayes_log_1}
\end{eqnarray}
Corollary~\ref{nd} therefore gives
\begin{equation*}
  R_{\rm Bayes}(w, L; \Theta) \geq \frac{3}{8} e^{-2p} (8 V)^{-p/d} \delta^p \exp \left(-\frac{p\delta^2}{2 \tau^2} - \frac{4 p \delta}{\tau} \right)  \qt{whenever $\delta^2 \leq 1/A$}.
\end{equation*}
We make the choice
\begin{equation*}
  \delta^2 := \min \left(1/A, \tau^2 \right) = \min \left(\frac{2 a(\phi)}{n K \lambda_{\max}}, \tau^2 \right)
\end{equation*}
which implies that the exponential term in the right hand side of~\eqref{eq:Bayes_log_1} is bounded from below by $\exp(-9p/2)$. We thus have
\begin{equation*}
  R_{\rm Bayes}(w, L; \Theta) \geq \frac{3}{8} e^{-13p/2} (8 V)^{-p/d} \left[\min \left(\frac{2 a(\phi)}{n K \lambda_{\max}}, \tau^2 \right)  \right]^{p/2}.
\end{equation*}
The inequality~\eqref{mee} now follows because $V^{1/d} \asymp d^{-1/2}$.

The assumption that $b''(\beta) \leq K$ which was used for the proof of~\eqref{mee} holds  under some widely used densities of $Y_i$ in \eqref{eq:glm_density}. For Gaussian distribution in \eqref{eq:glm_density}, we have $b(\beta)=\frac{\beta^2}{2}$ so that $b''(\beta)=1$ for $\beta \in \mathbb{R}$. For binomial distribution, $b(\beta)= \log(1+\exp(\beta))$ and $b''(\beta)=\frac{\exp(\beta)}{(1+\exp(\beta))^2} \leq \frac{1}{4}$ for all $\beta \in \mathbb{R}$. However, for Poisson distribution, $b(\beta)=\exp(\beta)$ and thus $b''(\beta) =\exp(\beta)$ is unbounded on $\mathbb{R}$. To address this issue, we restrict the prior to the subset $\widetilde{\Theta}=\{ \theta \in \Theta: \|\theta\|_2 \leq 2 \tau \sqrt{d} \}$ and define the re-scaled prior distribution $\pi$ on $\widetilde{\Theta}$ as $\pi(S)=w(S)/w(\widetilde{\Theta})$ for any measurable set $S \subseteq  \widetilde{\Theta}$. Let $B= \max_{i=1,\ldots, n} \|x_i\|_2$. For any $\beta=x_i^T \theta$ for some $i=1, \ldots, n$ and $\theta \in \widetilde{\Theta}$, we have $b''(\beta) \leq \exp(2 \tau \sqrt{d}  B):=K$. We note that such a restriction of the parameter space will not affect the order of the Bayes risk lower bound. In particular, since now $b''(\beta) \leq K$ when $\theta \in \widetilde{\Theta}$, applying the same argument, we obtain the lower bound on $R_{\rm Bayes}(\pi, L; \widetilde{\Theta})$.  By \eqref{eq:Gauss_con}, we have $w(\widetilde{\Theta})\geq 3/4$ and the lower bound on $R_{\rm Bayes}( w, L ; \Theta)$ can be easily established by noticing that $R_{\rm Bayes}( w, L ; \Theta)  \geq  w(\widetilde{\Theta})R_{\rm Bayes}( \pi, L ; \widetilde{\Theta}) \geq \frac{3}{4} R_{\rm Bayes}( \pi, L ; \widetilde{\Theta}).$

\subsection{Spiked Covariance Model}\label{sec:vee}

Fix $\Theta = \ac = B$ where $B$ is the unit Euclidean closed ball of radius one and let $L(\theta, a) := \|\theta - a\|_2^p$ for a fixed $p > 0$. Also fix $n \geq d/2$. For $\theta \in \Theta$, let $P_{\theta}$ denote the joint distribution of independent and identically distributed observations $X_1, \dots, X_n$ satisfying the Gaussian distribution with zero mean and covariance matrix $\Sigma_{\theta} := I_d + \theta \theta^T$. This is the problem of estimating the principal component for a rank-one spiked covariance model. Let $w$ denote the uniform distribution on $B$. We shall prove that
\begin{equation}\label{vee}
  R_{\rm Bayes}(w, L; \Theta) \geq C \left[\min \left(\frac{1}{2}, \frac{d}{n} \right) \right]^{p/2}
\end{equation}
where $C$ only depends on $p$.

The proof is based on the application of~\eqref{eq:main_bwl}  with $f(x) = x^2 - 1$, i.e., on inequality~\eqref{eq:I_bayes_chi}.

For this, we need to bound the term $\sup_{a \in \ac} w(B_t(a, L))$ and the $f$-informativity corresponding to the chi-squared divergence. It is easy to see that $\sup_{a \in \ac} w(B_t(a, L)) \leq t^{d/p}$.

For the $f$-informativity, we will use the bound~\eqref{eq:chi_dist_upper} with $\alpha = 2$ which requires bounding $M_{\chi^2}(\epsilon, \Theta)$. According to \cite[Theorem 4.6.1]{Aditya:PhDThesis}, for two Gaussian distributions with mean zero and covariance matrices $\Sigma_1$ and $\Sigma_2$ such that $2 \Sigma_1^{-1}  - \Sigma_2^{-1}$ is positive definite and $\|\Sigma_1 - \Sigma_2\|_F^2 \leq \frac{1}{2} \lambda^2_{\min}(\Sigma_2)$,  we have
  \begin{eqnarray}
    \chi^2\left(N_d(0, \Sigma_1) || N_d(0, \Sigma_2) \right) \leq \exp\left(\frac{\|\Sigma_1 - \Sigma_2\|^2_F}{\lambda_{\min}(\Sigma_2)^2}\right) -1.
    \label{eq:Gaussian_Chi_0}
  \end{eqnarray}
Here $\|\cdot\|_F$ denotes the Frobenius norm defined as $\|A\|_F^2 := \sum_{i, j} a_{ij}^2$ where $A = (a_{ij})$ and $\lambda_{\min}$ denotes the smallest eigenvalue.% Also $N_d(0, \Sigma)$ denotes the zero mean Gaussian distribution with covariance matrix $\Sigma$.

Using this result, we get that for $\theta_1, \theta_2 \in \Theta$ (note that $\lambda_{\min}(\Sigma_{\theta}) = 1$ for all $\theta$),
   \begin{eqnarray}
     \chi^2\left(P_{\theta_1} || P_{\theta_2} \right) \leq\exp\left(n\|\Sigma_{\theta_1} - \Sigma_{\theta_2}\|^2_F\right) -1,
     \label{eq:Gaussian_Chi}
   \end{eqnarray}
provided
   \begin{equation}\label{acon}
2 \Sigma_{\theta_1}^{-1} - \Sigma_{\theta_2}^{-1} \text{ is positive definite and }  \|\Sigma_{\theta_1} - \Sigma_{\theta_2}\|_F^2 \leq 1/2.
   \end{equation}
In the sequel, whenever we employ~\eqref{eq:Gaussian_Chi}, the conditions~\eqref{acon} hold. But, for ease of presentation, instead of verifying~\eqref{acon} for every application of~\eqref{eq:Gaussian_Chi}, we will simply assume~\eqref{eq:Gaussian_Chi} and verify the necessary conditions at the end of the proof.  Assuming~\eqref{eq:Gaussian_Chi}, we see that $\chi^2(P_{\theta_1}\|P_{\theta_2}) \leq \epsilon^2$ provided $\|\Sigma_{\theta_1} - \Sigma_{\theta_2}\|_F^2 \leq \log(1 + \epsilon^2)/n$. Now for $\theta_1, \theta_2 \in \Theta$
\begin{multline*}
  \|\Sigma_{\theta_1} - \Sigma_{\theta_2}\|_F^2  = \|\theta_1 \theta_1^T - \theta_2 \theta_2^T\|_F^2 = \|\theta_1 \theta_1^T - \theta_1 \theta_2^T + \theta_1 \theta_2^T - \theta_2 \theta_2^T\|_F^2 \\
\leq 2 \left(\|\theta_1\|_2^2 + \|\theta_2\|_2^2 \right) \|\theta_1 - \theta_2\|_2^2 \leq 4 \|\theta_1 - \theta_2\|_2^2.
\end{multline*}
It follows therefore that the $\epsilon^2$-covering number in the chi-squared divergence can be bounded from above by the $\sqrt{\log(1 + \epsilon^2)}/(2\sqrt{n})$-covering number of $B$ under the usual Euclidean norm. Consequently
\begin{equation*}
  M_{\chi^2}(\epsilon, \Theta) \leq \left(\frac{36 n}{\log (1 + \epsilon^2)}\right)^{d/2} \qt{provided \; $\log(1 + \epsilon^2) \leq 4n$}.
\end{equation*}
We now set $\epsilon$ to satisfy $\log(1+\epsilon^2)=\min \left(n/2, d\right)$ so that Corollary~\ref{cor:chi_dist_upper} gives
\begin{eqnarray*}
   I_{\chi^2}(w, \Ps) &\leq & M_{\chi^2}(\epsilon)(1+\epsilon^2) - 1\\
               & \leq & \exp \left(\min\left(\frac{n}{2}, d\right) \right)  \left[36 \max\left(2, \frac{n}{d} \right) \right]^{d/2} -1 =: \Iu.
\end{eqnarray*}
It follows that $ \sup_{a \in \ac} w(B_t(a, L)) < \frac{1}{4} (1 + \Iu)^{-1}$ provided $t = (4(1 + \Iu))^{-p/d}$. Inequality~\eqref{eq:I_bayes_chi} then proves
\begin{equation*}
  R_{\rm Bayes}(w, L; \Theta) \geq \frac{1}{2} \left(4(1 + \Iu) \right)^{-p/d} \geq \frac{1}{2} (24 e)^{-p} \left[\min \left(\frac{1}{2}, \frac{d}{n} \right) \right]^{p/2}
\end{equation*}
which implies~\eqref{vee}.

It remains to justify the conditions~\eqref{acon} when we used~\eqref{eq:Gaussian_Chi}.  It should be clear that for this, we only need to verify~\eqref{acon} when
\begin{equation}\label{didi}
  \|\Sigma_{\theta_1} - \Sigma_{\theta_2}\|_F^2 \leq \frac{\log(1+\epsilon^2)}{n} = \min \left( \frac{1}{2}, \frac{d}{n} \right).
\end{equation}
We only need to check that $2 \Sigma_{\theta_1}^{-1} - \Sigma_{\theta_2}^{-1}$ is positive definite under the above condition. For this, observe that by Weyl's inequality,
 \begin{eqnarray*}
     \lambda_{\min}\left(2 \Sigma_{\theta_1}^{-1} -\Sigma_{\theta_2}^{-1}\right) \geq \lambda_{\min}\left( 2 \Sigma_{\theta_1}^{-1}\right) - \lambda_{\max}\left(\Sigma_{\theta_2}^{-1}\right) = \frac{2}{1+\|\theta_1\|_2^2} - 1 \geq 0.
 \end{eqnarray*}
This implies that $2 \Sigma_{\theta_1}^{-1} - \Sigma_{\theta_2}^{-1}$ is positive semi-definite
and $\|\theta_1\|_2=1$ is a necessary condition for $\lambda_{\min}\left(2 \Sigma_{\theta_1}^{-1} -\Sigma{\theta_2}^{-1}\right)=0$. Under the condition that $\|\theta_1\|_2=1$, by Sherman-Morrison formula,
 \begin{eqnarray*}
    2\Sigma_{\theta_2}^{-1} - \Sigma_{\theta_1}^{-1}=I_d - \theta_1 \theta_1^T+\frac{\theta_2 \theta_2^T}{1+\theta_2^T \theta_2}.
 \end{eqnarray*}
 It is then easy to check that  $\lambda_{\min}\left(2 \Sigma_{\theta_1}^{-1} -\Sigma_{\theta_2}^{-1}\right)=0$ only if $\theta_2$ is orthogonal to $\theta_1$. However, when $\|\theta_1\|_2=1$ and $\theta_2$ is orthogonal to $\theta_1$,  $\|\Sigma_{\theta_1} - \Sigma_{\theta_2}\|^2_F= \|\theta_1\|_2^2+ \|\theta_2\|_2^2 > 1$,
 which contradicts \eqref{didi}. Therefore $2 \Sigma_{\theta_1}^{-1} - \Sigma_{\theta_2}^{-1}$ is positive definite and this completes the proof of~\eqref{vee}.

\subsection{Gaussian Model with General Loss}
\label{sec:Gaussian}
In this example, we consider Gaussian location model with continuous
prior with a bounded Lebesgue density and general loss
functions. Here, we do not specify the form of the prior and loss. We
only present this example to illustrate applications of Theorem
\ref{bwl} and Corollary \ref{cor:bwl}. Our main bound is inequality
\eqref{hok}. This bound however might be suboptimal for specific
priors $w$ because we do not use knowledge about the specific form of
$w$. However, when the specific form of $w$ is available, the argument
can often be easily modified to improve inequality \eqref{hok}. We
provide examples of this at the end of this subsection.

\subsubsection{Gaussian Model with Squared Loss}

Fix $d \geq 1$. Suppose $\Theta = \ac = \R^d$ and let $L(\theta, a) := \|\theta - a\|_2^2$ where $\|\cdot\|_2$ is the usual Euclidean norm on $\R^d$. For each $\theta \in \R^d$, let $P_{\theta}$ denote the Gaussian distribution with mean $\theta$ and covariance matrix $\sigma^2 I_d$  ($\sigma^2 > 0$ is a constant). For every prior $w$ on $\R^d$ with a Lebesgue density bounded by $W > 0$, we have
\begin{equation}\label{hok}
  R_{\rm Bayes}(w, L; \Theta) \gtrsim \frac{d \sigma^4 W^{-2/d}}{(\sigma^2 + V)^2}
\end{equation}
where
\begin{equation}\label{eq:V}
    V := \min_{s \in \R^d} \int_{\Theta} \frac{1}{d} \sum_{i=1}^d (\theta_i- s_i)^2 w(\mathrm{d} \theta).
\end{equation}

To prove~\eqref{hok}, we shall apply~\eqref{eq:main_bwl} with $f(x) = x \log x$, i.e., we apply~\eqref{eq:I_bayes_KL}. The resulting $f$-informativity (a.k.a mutual information) can be bounded in the following way. Because $I(w, \Ps) \leq \int D(P_{\theta}\|Q) w(\mathrm{d}\theta)$ for every $Q$. In particular, we take $Q$ to be the Gaussian  distribution with mean $t$ and covariance matrix                                                                                                                                                                                                                                                                                                                              $(\sigma^2 + V) I_d$, where $t= \argmin_{s \in \R^d} \int_{\Theta} \frac{1}{d} \sum_{i=1}^d (\theta_i- s_i)^2 w(\mathrm{d} \theta)$, i.e., $t_i =\int_{\Theta} \theta_i w(\mathrm{d} \theta)$ is ``center" of the prior.
Then, we obtain
\begin{equation*}
  I(w, \Ps)  \leq \int_\Theta D\left( N\left( \theta, \sigma^2 I_d \right)  || N\left(t, \left(\sigma^2+V\right) I_d \right)  \right) w(\mathrm{d}\theta).
\end{equation*}
Using the standard formula for the KL divergence between two Gaussians, we deduce that
\begin{equation*}
  I(w, \Ps) \leq \frac{1}{2} \int_{\Theta} \left[\frac{\sum_{i=1}^d ((\theta_i-t_i)^2 - V)}{\sigma^2 + V} + d \log \frac{\sigma^2 + V}{\sigma^2} \right] w(\mathrm{d}\theta)
\end{equation*}
which by \eqref{eq:V} implies that
\begin{equation}\label{ibi}
  I(w, \Ps) \leq \frac{d}{2} \log \frac{\sigma^2 + V}{\sigma^2}.
\end{equation}
Let $\Iu$ denote the right hand side above. To apply~\eqref{eq:I_bayes_KL}, we also need an  upper bound on $\sup_{a \in A} w\left(B_t(a,L)\right)$. Because of the assumption that the Lebesgue density of $w$ is bounded from above by $W$, we get
\begin{equation}\label{psu1}
   \sup_{a \in A} w\left(B_t(a,L) \right) \leq  W t^{d/2} \mathrm{Vol}(B)
\end{equation}
where $B$ is the Euclidean ball with unit radius. Thus the choice
\begin{equation*}
  t = c W^{-2/d} \mathrm{Vol}(B)^{-2/d} \frac{\sigma^4 }{(\sigma^2 + V)^2},
\end{equation*}
for a small enough universal positive constant $c$, ensures $\sup_{a \in A} w\{B_t(a)\} < \frac{1}{4} e^{-2 \Iu}$ (recall that $\Iu$ is the right hand side of\eqref{ibi}).  Consequently, inequality~\eqref{eq:I_bayes_KL} implies that $R_{\rm Bayes} \geq t/2$. The proof of~\eqref{hok} is now completed using the standard fact: $\mathrm{Vol}(B)^{1/d} \asymp d^{-1/2}$.

However, since the form of the prior $w$ is unspecified in this
example, the simple upper bound on $ \sup_{a \in A} w\left(B_t(a,L)
\right) $ in \eqref{psu1} could be loose. But this can be easily fixed
when the concrete form of the prior is available. For example,  for a
spiked model with a large $W$ (see an example of mixture prior in
Remark \ref{rem:Bayes_decomp} in the main text), the lower bound in
\eqref{hok} could be sub-optimal but can be easily tightened using the
proposed chaining technique in Remark \ref{rem:Bayes_decomp} in the
main text. For another example, let $w$  be the uniform
prior on the hyper-rectangle $H = [-\epsilon,\epsilon]\times
[-1,1]^{d-1}$ for some very small $\epsilon$. Here
inequality~\eqref{psu1} is equivalent to
  \[
   \sup_{a \in A} w\left(B_t(a,L) \right) \leq  W t^{d/2} \mathrm{Vol}(B).
\]
When $\epsilon\to 0$, we have $W\to \infty$ so that the upper bound is fairly loose. However, since $H$ is the support of $w$, we can also use the following upper bound:
\[
   \sup_{a \in A} w\left(B_t(a,L) \right) \leq  W t^{d/2} \mathrm{Vol}(B\cap H).
\]
When $\epsilon\to 0$, we have $W\to \infty$ but $\mathrm{Vol}(B\cap
H)\to 0$. In particular, the product limit $\lim_{\epsilon\to 0}
W\mathrm{Vol}(B\cap H)\to 0$ is finite. It converges to the maximum
value of $w\left(B_t(a,L) \right)$ where $w$ is restricted in a
$(d-1)$-dimensional subspace of $\R^d$. Once we replace
inequality~\eqref{psu1} by the above upper bound, the associated Bayes
risk lower bound will be tight.

\subsubsection{Gaussian Model with General Loss}

Consider the same setup as in the previous example but now allow the loss function to be $L(\theta, a) = \|\theta - a\|^2$ for an arbitrary norm $\|\cdot\|$ (not necessarily the Euclidean norm) on $\R^d$. In this case, we obtain the following Bayes risk lower bound:
 \begin{eqnarray}
   R_{\rm Bayes}(w, L; \Theta) \gtrsim \frac{\sigma^4 W^{-2/d}}{(\sigma^2 + V)^2} \frac{d^2}{(\E \|Z\|_* )^2}.
   \label{eq:Gaussian_location_norm}
 \end{eqnarray}
where $Z$ is a standard Gaussian vector and $\|\cdot\|_*$ is the dual norm corresponding to $\|\cdot\|$ defined by $\|x\|_* := \sup \{\left<x, y \right> : \|y\| \leq 1\}$. The quantities $W$ and $V$ are as defined in  the previous example.

The proof of~\eqref{eq:Gaussian_location_norm} is largely similar to that of~\eqref{hok}. We use~\eqref{eq:I_bayes_KL} along with~\eqref{ibi} for controlling $I(w, \Ps)$. To control $\sup_{a \in \ac} w(B_t(a, L))$, we again use the fact that the Lebesgue density of $w$ is bounded from above by $W$ to obtain
\begin{equation}\label{psu3}
   \sup_{a \in A} w\left(B_t(a,L) \right)  \leq  W \mathrm{Vol} \left\{\theta \in \R^d : \|\theta\| < \sqrt{t} \right\}.
\end{equation}
To deal with the volume term above, we use Urysohn's inequality to obtain an upper bound in terms of the volume of the unit Euclidean unit ball $B$. The  original reference for Urysohn's inequality is~\cite{Urysohn:24} but it has been recently  used in a statistical context by~\cite{MaWu:13}. Urysohn's inequality gives
 \begin{equation}\label{psu2}
  \left(\frac{\mathrm{Vol} \left\{\theta \in \R^d : \|\theta\| < \sqrt{t} \right\}}{\mathrm{Vol}(B) } \right)^\frac{1}{d} \leq \frac{\sqrt{t}}{\sqrt{d}} \E\|Z\|_* \qt{with $Z \sim N(0, I_d)$}.
\end{equation}
Inequalities~\eqref{psu3} and~\eqref{psu2} together give
 \begin{equation*}
   \sup_{a \in A} w\left(B_t(a,L) \right) \leq W t^{d/2} \mathrm{Vol}(B) \left(\frac{\E \|Z\|_*}{\sqrt{d}} \right)^d.
 \end{equation*}
The choice
\begin{equation*}
  t = c \mathrm{Vol}(B)^{-2/d} \frac{W^{-2/d} \sigma^4}{(\sigma^2 + V)^2} \frac{d}{(\E\|Z\|_*)^2}
\end{equation*}
for a small enough universal positive constant $c$ ensures $\sup_{a \in A} w\{B_t(a)\} < \frac{1}{4} e^{-2 \Iu} $ ($\Iu$ is the right hand side of~\eqref{ibi}). The proof of~\eqref{eq:Gaussian_location_norm} is then completed by noting that $\mathrm{Vol}(B)^{1/d} \asymp d^{-1/2}$.

\section{Proof of Lemma \ref{lemma:distance} in Section \ref{sec:smoothed-analysis}}
\label{sec:distance}

Consider the $i$-th instance $x_i\in \R^d$ sampled from $N(\theta_{z_i}; I_{d\times d})$, where $z_i$ is the membership.  Note that for any $j\in [k]\backslash\{z_i\}$, the distance between $\theta_{z_i}$ and $\theta_j$ is lower bounded by $D$. We have
\begin{equation}\label{eq:dist_1}
	\ltwos{x_i - \theta_{j}}^2 - \ltwos{x_i - \theta_{z_i}}^2 =
	\ltwos{\theta_{j} - \theta_{z_i}}^2 - 2\langle \theta_{j} - \theta_{z_i}, x_i - \theta_{z_i}\rangle.
\end{equation}
The random variable $\langle \theta_{j} - \theta_{z_i}, x_i - \theta_{z_i}\rangle$ satisfies distribution $\normal(0; \ltwos{\theta_{j} - \theta_{z_i}}^2)$. Let $\Phi$ be the CDF of the standard normal distribution. Then with probability $ \Phi(\frac{\ltwos{\theta_{j} - \theta_{z_i}}-1}{2})$, we have
\begin{equation}\label{eq:dist_2}
	\langle \theta_{j} - \theta_{z_i}, x_i - \theta_{z_i}\rangle \leq
	 \ltwos{\theta_{j} - \theta_{z_i}}\cdot \frac{\ltwos{\theta_{j} - \theta_{z_i}} - 1}{2}.
\end{equation}
Combining \eqref{eq:dist_1} and \eqref{eq:dist_2}, we have
\begin{equation}\label{eq:dist_3}
	\ltwos{x_i - \theta_{j}}^2 - \ltwos{x_i - \theta_{z_i}}^2 \geq \ltwos{\theta_{j} - \theta_{z_i}}.
\end{equation}
On the other hand, the triangular inequality implies
\begin{equation}\label{eq:dist_4}
	\ltwos{x_i - \theta_{j}} + \ltwos{x_i - \theta_{z_i}} \leq 2\ltwos{x_i - \theta_{z_i}} + \ltwos{\theta_{j} - \theta_{z_i}}.
\end{equation}
The random variable $\ltwos{x_i - \theta_{z_i}}^2$ satisfies a chi-square distribution with $d$ degrees of freedom. It is upper bounded by $\beta d$ with probability at least $1 - \exp(\frac{d}{2}(1-\beta+\log \beta))$ for any $\beta > 1$~\citep{dasgupta2003elementary}. Putting \eqref{eq:dist_3} and \eqref{eq:dist_4} together, we have
\[
	\ltwos{x_i - \theta_{j}} - \ltwos{x_i - \theta_{z_i}} = \frac{\ltwos{x_i - \theta_{j}}^2 - \ltwos{x_i - \theta_{z_i}}^2}{\ltwos{x_i - \theta_{j}} + \ltwos{x_i - \theta_{z_i}}} \geq \frac{\ltwos{\theta_{j} - \theta_{z_i}}}{2\ltwos{\theta_{j} - \theta_{z_i}} + \sqrt{\beta d}} \geq \frac{3}{6+\sqrt{\beta d}}
\]
with probability at least $\Phi(\frac{D-1}{2}) - \exp(\frac{d}{2}(1-\beta+\log \beta))$. By choosing $D = c\sqrt{\log(n k/\delta)}$ and $\beta = c \log(n k/\delta)/d$ for a sufficiently large constant $c$, this probability is lower bounded by $1 - \delta/(nk)$.
Applying union bound, the inequality holds for any $(i,j)$ pair with probability at least~$1 - \delta$.

\section{Proof of Theorem \ref{theorem:sparse-regression-lower-bound} in Section \ref{sec:sparse_linear}}
\label{sec:supp_sparse}

We start with a simplified case where the random index set $K$ is given to the estimator. Knowing this information makes the problem easier, and makes the Bayes risk lower. In addition, it reduces the $d$-dimensional regression problem to a $k$-dimensional problem where a closed-form of the Bayes risk can be derived, which establishes the following lower bound:
\begin{claim}\label{claim:sparse-trivial-lower-bound}
For any $\tau > 0$, the Bayes risk is lower bounded by:
\begin{align*}
\RBayes(w,\Lest;\Theta) \geq \frac{1}{1 + \kappa_u^2  \tau^2 n/\sigma^2}\cdot k\tau^2, \quad\mbox{and}\quad \RBayes(w,\Lpre;\Theta) \geq \frac{1}{1 + \kappa_\ell^2  \tau^2 n/\sigma^2}\cdot \kappa_\ell^2 k\tau^2.
\end{align*}
\end{claim}
\noindent
See Section~\ref{sec:proof-sparse-trivial-lower-bound} for the proof.\\

For the rest of this proof, we establish stronger lower bounds using the fact that the index set $K$ is unknown.
It is easy to verify that for any random variable $X$ sampled from $N(0,1)$, the probability of $|X| \geq 1/2$ is greater than $1/2$. Consider a subset of the parameter space $\Theta$:
\begin{align}
	\Thetabar \defeq \Big\{\theta\in \Theta:~\norms{\theta}_2^2 \leq 2k\tau^2 \mbox{ and } \sum_{i=1}^d \mathbb{I}[|\theta_i| \geq \tau/2 ] \geq k/2\Big\}.
\end{align}
For a random vector $\theta$ sampled from the prior distribution $w$, the quantity $\ltwos{\theta}^2/\tau^2$  satisfies a chi-square distribution with $k$ degrees of freedom. For any $k\geq 1$, the event $\norms{\theta}_2^2 \leq 2k\tau^2$ happens with probability at least $0.84$. Given an index set $K$, for any $i\in K$ the random variable $\mathbb{I}[|\theta_i| \geq \tau/2]$ satisfies the Bernoulli distribution with  parameter  greater than $1/2$, so that the event $\sum_{i=1}^d \mathbb{I}[|\theta_i| \geq \tau/2$ happens with probability at least $1/2$. Combining these two lower bounds and applying union bound, we obtain $w(\Thetabar) \geq 1/2 - (1-0.84) > 1/4$. As a consequence, if we define a distribution $\wbar$ over the subset $\Thetabar$ by $\wbar(A) \defeq w(A\cap \Thetabar)/w(\Thetabar)$, then Remark 11 implies that
\begin{align}
	\RBayes(w,L;\Theta) \geq w(\Thetabar)\cdot \RBayes(\wbar,L;\Thetabar) \geq \frac{1}{4} \RBayes(\wbar,L;\Thetabar).
\end{align}
Hence it suffices to focus on the Bayes risk for the marginal prior $\wbar$.

Let the action space $\A \defeq \R^d$ and let the loss function be either the estimation error $\Lest$ or the prediction error $\Lpre$. In order to lower bound the Bayes risk, it suffices to bounded the chi-square informativity $I_{\chi^2}(\wbar, \mathcal{P})$ and the quantity $\sup_{a\in \A} \wbar(B_t(a,L))$, then applying Corollary~\ref{cor:bwl}. We begin with an upper bound on the chi-square informativity.

\begin{claim}\label{claim:sparse-informativity}
For any $\tau > 0$, the chi-square informativity is bounded by:
\begin{align}\label{eqn:sparse-informativity-bound}
	I_{\chi^2}(\wbar, \mathcal{P}) + 1 \leq \exp(2\kappa_u^2 \tau^2 kn/\sigma^2)
\end{align}
\end{claim}
\noindent See Section~\ref{sec:proof-sparse-informativity} for the proof.\\

Next, we upper bound the quantity $\sup_{a\in \A} \wbar(B_t(a,L))$. We begin by claiming a property of all Euclidean balls of small enough radius.

\begin{claim}\label{claim:euclidean-ball-covering-bound}
For any point $a\in \R^d$, let $B(a,r)$ be the Euclidean ball of radius $r$ centering at $a$. If $r \leq \frac{1}{8}\sqrt{k}\tau$, then there is a universal constant $c > 0$ such that
\[
	\sup_{a\in \A} {\wbar}(B(a,r)) \leq  \frac{c^k}{(d/k^2)^{k/4}}\Big(\frac{r}{\sqrt{k}\tau}\Big)^k.
\]
\end{claim}

\noindent See Section~\ref{sec:proof-sparse-euclidean-ball} for the proof.

\paragraph{Lower bound on estimation error}
For the estimation error, we obtain by Claim~\ref{claim:euclidean-ball-covering-bound} that for any $t \leq \frac{1}{64} k\tau^2$, the following upper bound holds:
\begin{align}\label{eqn:est-error-prior-bound}
	\sup_{a\in \A} \wbar(B_t(a,\Lest)) = \sup_{a\in \A} \wbar(B(a,\sqrt{t})) \leq \frac{c^k}{(d/k^2)^{k/4}}\Big(\frac{\sqrt{t}}{\sqrt{k}\tau}\Big)^k.
\end{align}
Combining Claim~\ref{claim:euclidean-ball-covering-bound} with  inequality~\eqref{eqn:est-error-prior-bound}, and applying inequality \eqref{eq:I_bayes_chi} in Corollary~\ref{cor:bwl}, we obtain the lower bound:
\begin{align*}
	\RBayes(\wbar,\Lest;\Thetabar) \geq \frac{1}{2} \sup \left\{0 < t \leq \frac{k \tau^2}{64} : \Big(\frac{\sqrt{t}}{\sqrt{k}\tau}\Big)^k \leq \frac{(d/k^2)^{k/4}}{c^k}\cdot  \frac{1}{4}\exp(- 2\kappa_u^2 \tau^2 kn/\sigma^2)  \right\}.
\end{align*}
The right-hand side is lower bounded by any scalar $t$ satisfying:
\[
	t \leq \frac{1}{64} k\tau^2 \quad\mbox{and}\quad \frac{\sqrt{t}}{\sqrt{k}\tau} \leq \frac{(d/k^2)^{1/4}}{c}\cdot  \frac{1}{4^{1/k}}\exp(- 2\kappa_u^2 \tau^2 n/\sigma^2)
\]
It implies that for some universal constant $c' > 0$, we have:
\begin{align}
	\RBayes(\wbar,\Lest;\Thetabar) &\geq c'\,k\tau^2\,\min\Big\{1, \exp(\frac{1}{2}\log(d/k^2) - 4\kappa_u^2 \tau^2 n / \sigma^2)
\Big\}\nonumber\\
&= c'\,k\tau^2\,\exp\Big(\min\Big\{0, \frac{1}{2}\log(d/k^2) - \frac{4\kappa_u^2 \tau^2 n}{\sigma^2}\Big\}\Big)\nonumber\\
&= c'\,k\tau^2\,\exp\Big( - \frac{4\kappa_u^2 n}{\sigma^2} \Big[\tau^2 - \frac{\sigma^2 \log(d/k^2)}{8\kappa_u^2 n} \Big]_+\Big)\nonumber\\
&\geq  c'\,k\tau^2\,\exp\Big( - \frac{4\kappa_u^2 n}{\sigma^2} \Big[\tau^2 - \frac{\sigma^2 \log (d/k)}{16 \kappa_u^2 n} \Big]_+\Big),\label{eqn:lest-refined-bound}
\end{align}
where the last inequality uses the assumption $d > k^3$ and its implication $\log(d/{k^2}) >\frac{1}{2}\log(d/k)$.

Combining inequality~\eqref{eqn:lest-refined-bound} with Claim~\ref{claim:sparse-trivial-lower-bound} yields the lower bound:
\[
	\RBayes(w,\Lest;\Theta) \geq  c'\,k\tau^2 \max\Big\{\frac{1}{1 + \kappa_u^2  \tau^2 n/\sigma^2}, \exp\Big( - \frac{4\kappa_u^2 n}{\sigma^2} \Big[\tau^2 - \frac{\sigma^2 \log(d/k)}{16\kappa_u^2 n} \Big]_+\Big)\Big\},
\]
which completes the proof.

\paragraph{Lower bound on prediction error}

For the prediction error, we consider an arbitrary vector $a\in \R^d$ and an arbitrary scalar $t$ satisfying $\sqrt{t} \leq \frac{\kappa_\ell}{16}\sqrt{k}\tau$. Let $\theta'$ be the vector in $\Thetabar$ which minimizes the term $\frac{1}{n}\ltwos{X(a-\theta')}^2$. If the inequality $\frac{1}{n}\ltwos{X(a-\theta')}^2 > t$ is true, then we have
\begin{align}\label{eqn:prediction-error-prior-bound-1}
	\sup_{a\in \A} \wbar(B_t(a,\Lpre)) = 0.
\end{align}
Otherwise, we assume that $\frac{1}{n}\ltwos{X(a-\theta')}^2 \leq t$. Then for any vector $\theta\in \Thetabar$ satisfying $\frac{1}{n}\ltwos{X(\theta-a)}^2 \leq t$, we have the upper bound
\[
\frac{1}{n}\ltwos{X(\theta-\theta')}^2 \leq (n^{-1/2}\ltwos{X(\theta-a)} + n^{-1/2}\ltwos{X(a-\theta')})^2 \leq (\sqrt{t} + \sqrt{t})^2  \leq 4t.
\]
It means that $B_t(a,\Lpre)\subseteq B_{4t}(\theta',\Lpre)$. Since the vector $\theta'$ is $k$-sparse, the the sparse eigenvalue condition implies that for any vector $\theta\in \Thetabar$, if $\Lpre(\theta, \theta') \leq 4t$, then $\ltwos{\theta-\theta'} \leq \frac{2\sqrt{t}}{\kappa_\ell}$, so that $B_{4t}(\theta',\Lpre) \subseteq B(\theta', \frac{2\sqrt{t}}{\kappa_\ell})$. Using  Claim~\ref{claim:euclidean-ball-covering-bound}, we have
\begin{align}\label{eqn:prediction-error-prior-bound-2}
	\sup_{a\in \A} \wbar(B_t(a,\Lpre)) \leq \sup_{a\in \A} \wbar(B(\theta', \frac{2\sqrt{t}}{\kappa_\ell})) \leq \sup_{a\in \A} \wbar(B_t(a,\Lest)) \leq \frac{c^k}{(d/k^2)^{k/4}}\Big(\frac{2\sqrt{t}}{\kappa_\ell \sqrt{k}\tau}\Big)^k.
\end{align}
Combining equation~\eqref{eqn:prediction-error-prior-bound-1} and inequality~\eqref{eqn:prediction-error-prior-bound-2} we obtain
\begin{align}\label{eqn:prediction-error-prior-bound}
	\sup_{a\in \A} \wbar(B_t(a,\Lpre)) \leq \frac{c^k}{(d/k^2)^{k/4}}\Big(\frac{2\sqrt{t}}{\kappa_\ell \sqrt{k}\tau}\Big)^k\quad\mbox{for any }
	\sqrt{t} \leq \frac{\kappa_\ell}{16}\sqrt{k}\tau.
\end{align}
Comparing inequalities~\eqref{eqn:est-error-prior-bound} and~\eqref{eqn:prediction-error-prior-bound}, we find that they differ by a factor of $(2/\kappa_\ell)^k$. Thus, following the same steps for deriving inequality~\eqref{eqn:lest-refined-bound}, we can find a universal constant $c'' > 0$ such that:
\begin{align}\label{eqn:lpre-refined-bound}
\RBayes(\wbar,\Lpre;\Thetabar) \geq  c''\,\kappa_\ell^2 k\tau^2\,\exp\Big( - \frac{4\kappa_u^2 n}{\sigma^2} \Big[\tau^2 - \frac{\sigma^2 \log (d/k)}{16 \kappa_u^2 n} \Big]_+\Big)
\end{align}
Combining inequality~\eqref{eqn:lpre-refined-bound} with
Claim~\ref{claim:sparse-trivial-lower-bound} yields:
\[
	\RBayes(w,\Lpre;\Theta) \geq  c''\, \kappa_\ell^2  k\tau^2 \max\Big\{\frac{1}{1 + \kappa_\ell^2  \tau^2 n/\sigma^2}, \exp\Big( - \frac{4\kappa_u^2 n}{\sigma^2} \Big[\tau^2 - \frac{\sigma^2 \log(d/k)}{16\kappa_u^2 n} \Big]_+\Big)\Big\},
\]
which completes the proof.

\subsection{Proof of Claim~\ref{claim:sparse-trivial-lower-bound}}
\label{sec:proof-sparse-trivial-lower-bound}

The Bayes risk of the original problem is lower bounded by that of the following simplified problem: estimating $\theta$ when the index set $K$ is known, and without loss of generality, we assume that $K = [k]$.  For this case, let $X'$ be the submatrix consisting of the first $k$ columns of matrix $X$, and let $\theta'$ be the subvectors consisting of the first $k$ coordinate of vectors $\theta$. Given the response vector $y$, the posterior distribution of $\theta'$ is equal to
\[
	p(\theta'|y) \propto p(\theta')p(y|\theta') = N(\theta';0,\tau^2 I)\, N(y;X\theta', \sigma^2 I)
	\propto N\Big(\theta'; \Sigma^{-1} (X')^\top y, \sigma^2 \Sigma^{-1} \Big),
\]
where $\Sigma \defeq (X')^\top X' + \frac{\sigma^2}{\tau^2}I$ is a shorthand notation. As a consequence, the Bayes estimator $\thetahat$ is given by $\thetahat_K = \Sigma^{-1} (X')^\top y$ and $\thetahat_{-K} = 0$. The Bayes risk on the estimation error is lower bounded by:
\[
	\RBayes(w,\Lest;\Theta) = \E[\ltwos{\Sigma^{-1} (X')^\top y - \theta'}^2] = \sigma^2 \tr(\Sigma^{-1}) \geq \frac{\sigma^2}{\kappa_u^2  \tau^2 n + \sigma^2}\cdot k\tau^2,
\]
where the last inequality uses the sparse eigenvalue condition --- it guarantees  that all eigenvalues of the matrix $\Sigma$ are less  than or equal to $n\kappa_u^2 + \sigma^2/\tau^2$.

The Bayes estimator for minimizing the prediction error is also given by $\thetahat_K = \Sigma^{-1} (X')^\top y$ and $\thetahat_{-K} = 0$. Thus, the Bayes risk is lower bounded by:
\begin{align*}
	\RBayes(w,\Lpre;\Theta)& = \frac{1}{ n}\E[\ltwos{X'(\Sigma^{-1} (X')^\top y - \theta')}^2] = \frac{\sigma^2}{n} \tr(X'\Sigma^{-1}(X')^\top)\\
	&\geq \frac{\sigma^2}{\kappa_\ell^2 \tau^2 n + \sigma^2}\cdot \kappa_\ell^2 k\tau^2,
\end{align*}
where the last inequality uses the sparse eigenvalue condition --- it guarantees that all eigenvalues of the matrix $X'\Sigma^{-1}(X')^\top$ are greater than or equal to $\frac{n \kappa_\ell^2}{n \kappa_\ell^2 + \sigma^2/\tau^2}$.

\subsection{Proof of Claim~\ref{claim:sparse-informativity}}
\label{sec:proof-sparse-informativity}

Corollary~\ref{cor:chi_dist_upper} shows that the chi-square informativity can be bounded using the covering number $M_{\chi^2}(\epsilon, \Thetabar)$. Consider the zero vector $\theta_0 \defeq 0$ and an arbitrary vector $\theta \in \Thetabar$. Their response vectors are generated from $P_{\theta_0} \defeq N(0, \sigma^2I)$ and $P_\theta \defeq N(X\theta, \sigma^2I)$, so that the chi-square divergence between $P_{\theta_0}$ and $P_{\theta}$ is equal to $\chi^2(P_{\theta_0}\|P_{\theta}) = \exp(\ltwos{X\theta}^2/\sigma^2) - 1$.
By the sparse eigenvalue condition and the fact that $\ltwos{\theta}^2 \leq 2 k \tau^2$, we have
\[
\chi^2(P_{\theta_0}\|P_{\theta}) \leq \exp(\kappa_u^2 n \ltwos{\theta}^2/\sigma^2) - 1 \leq \exp(2\kappa_u^2 \tau^2 k n /\sigma^2) - 1.
\]
It means that if we choose $\epsilon^2 = \exp(2\kappa_u^2 \tau^2 k n /\sigma^2) - 1$, then $M_{\chi^2}(\epsilon, \Thetabar) = 1$, so that the chi-square informativity is bounded by
\begin{align}\label{eqn:sparse-informativity-bound-1}
	I_{\chi^2}(\wbar, \mathcal{P}) + 1 \leq (1+\epsilon^2)M_{\chi^2}(\epsilon, \Thetabar) =  \exp(2\kappa_u^2 \tau^2 k n /\sigma^2).
\end{align}

\subsection{Proof of Claim~\ref{claim:euclidean-ball-covering-bound}}
\label{sec:proof-sparse-euclidean-ball}

Consider an arbitrary vector $a\in \R^d$, and let $I_a$ be the set of indices defined by:
\[
	I_a = \{i\in [d]:~|a_i| \geq \tau/4\}.
\]
If $|I_a| > 2k$,  then for any $\theta\in\Thetabar$, there are at least $k+1$ coordinates such that $\theta_i = 0$ but $|a_i| \geq \tau/4$. It means that $\ltwos{a - \theta}> \frac{1}{4}\sqrt{k}\tau$. Since $r \leq \frac{1}{8}\sqrt{k}\tau$, we have  $\wbar(B(a, r)) = 0$.

Otherwise, we assume that $|I_a| \leq 2k$. Given an index set $K$, let $w_K$ and $\wbar_K$ be the conditional version of the prior distribution $w$ and $\wbar$, conditioning on the fact that the $k$-sparse index set is $K$. Recall that for any $\theta$ in the support of $\wbar_K$, there are at least $k/2$ coordinates such that $|\theta_i| \geq \tau/2$. If $|I_a \cap K| < k/4$, then there at least $k/4$ coordinates such that $|\theta_i| \geq \tau/2$ but $|a_i| < \tau/4$. It means that $\ltwos{a - \theta}> \frac{1}{8}\sqrt{k}\tau$ for any $\theta$ in the support of $\wbar_K$, and as a consequence, we have $\wbar_K(B(a,r)) = 0$.

Thus, a necessary condition for $\wbar_K(B(a,r)) > 0$ to hold is $|I_a \cap K| \geq k/4$. Given  $|I_a| \leq 2k$, the number of index set $K$ satisfying this constraint is bounded by ${2k \choose k/4}{d-k/4 \choose 3k/4}$. To prove this bound, notice that every set $K$ satisfying $|I_a \cap K| \geq k/4$ can be generated by the following two-step procedure: first, generate $k/4$ element from $I_a$; second, generate the remaining $3k/4$ elements from the remaining $d-k/4$ integers of $\{1,\dots,d\}$. There are totally ${2k \choose k/4}{d-k/4 \choose 3k/4}$ ways of generating the set. We note that the same $K$ can have multiple ways to generate, so that the above combinatorial number is a strict upper bound on the number of sets.

For any set $K$ satisfying the above constraint, we have:
\begin{align}\label{eqn:from-truncated-normal-back-to-normal}
	\wbar_K(B(a,r)) \leq 4 w_K(B(a,r)) \leq 4 w_K(B(0,r)),
\end{align}
where the last equation holds because $w_K$ represents an isotropic normal distribution in $\R^k$, so that the maximum probability is achieved by centering at the origin. The right-hand side of inequality~\eqref{eqn:from-truncated-normal-back-to-normal} the probability a $k$-dimension normal random variable $X\sim N(0, \tau^2 I_{k\times k})$ satisfying $\ltwos{X}\leq r$. As we showed in the proof of Lemma~\ref{lemma:gaussian-prior}, this probability is bounded by $(\frac{c\,r}{\sqrt{k}\tau})^k$ for a universal constant $c > 0$. Putting pieces together, we have
\begin{align*}
	\sup_{a\in \A} \wbar(B(a,r)) \leq \frac{{2k \choose k/4}{d-k/4 \choose 3k/4}}{{d\choose k}} \cdot 4\Big(\frac{c\,r}{\sqrt{k}\tau}\Big)^k.
\end{align*}
By the definition of the combinatorial numbers, we have:
\begin{align*}
	\frac{{2k \choose k/4}{d-k/4 \choose 3k/4}}{{d\choose k}} & = {2k \choose k/4} \frac{k!(d-k/4)!}{(3k/4)! d!} \leq {2k \choose k/4}\frac{k!}{(3k/4)!}  \frac{1}{d^{k/4}}\\
	&\leq \frac{(2k)^{k/2}}{d^{k/4}} = \Big( \frac{d}{4k^2}\Big)^{- k/4}.
\end{align*}
Combining the two upper bounds above completes the proof.

\vskip 0.2in

\bibliographystyle{chicago}
\bibliography{AG}

\end{document}